\documentclass{article}
\usepackage{arxiv}
\usepackage{comment}
\usepackage[utf8]{inputenc} 
\usepackage[T1]{fontenc}    
\usepackage{hyperref}       
\usepackage{url}            
\usepackage{booktabs}       
\usepackage{amsmath}
\usepackage{amssymb}
\usepackage{amsfonts}       
\usepackage{nicefrac}       
\usepackage{microtype}      
\usepackage{xcolor}
\usepackage{graphicx}
\usepackage{stmaryrd}
\usepackage{bm}
\usepackage[normalem]{ulem}
\usepackage[tickmarkheight=0.1cm,textsize=tiny]{todonotes}
\usepackage{algorithmicx}
\usepackage{algorithm}
\usepackage{algpseudocode}
\usepackage{caption}
\usepackage{subcaption}
\usepackage{lineno}
\usepackage{orcidlink}
\usepackage[inline,final]{showlabels}
\usepackage{graphicx,latexsym,tabularx,xcolor}
\usepackage{xpatch}
\usepackage{enumerate}
\makeatletter
\xpatchcmd{\algorithmic}
  {\ALG@tlm\z@}{\leftmargin\z@\ALG@tlm\z@}
  {}{}
\makeatother

\hypersetup{
    colorlinks=true, 
    linkcolor=blue, 
    urlcolor=red, 
    citecolor=magenta,
    linktoc=all 
}

\newlength\myindent 
\newcommand{\iph}{{i+\frac{1}{2}}}
\newcommand{\imh}{{i-\frac{1}{2}}}

\newcommand{\half}{\frac{1}{2}}

\newenvironment{proof}{\noindent\textbf{Proof\ }}{\hspace*{\fill}$\Box$\medskip}

\newtheorem{theorem}{Theorem}

\newcommand{\abs}[1]{\left| #1\right|}
\newcommand{\no} {\nonumber}
\newcommand{\dt}{\Delta t}
\newcommand{\noi} {\noindent}

\title{A well-balanced second-order finite volume approximation for a coupled system of granular flow }
\author{Aekta Aggarwal\\
Operations Management and Quantitative Techniques\\
Indian Institute of Management Indore\\
Indore--453556\\
\texttt{aektaaggarwal@iimidr.ac.in}
\And
Veerappa Gowda ~G.~D. \\
Centre for Applicable Mathematics\\
Tata Institute of Fundamental Research\\
Bangalore -- 560065\\
\texttt{gowda@tifrbng.res.in} \\
\And
Sudarshan~Kumar~K.\thanks{Corresponding author} \\
School of Mathematics\\
Indian Institute of Science Education and Research\\
Trivandrum --695551\\
\texttt{sudarshan@iisertvm.ac.in} \\
}
\usepackage[sort]{cite}

\begin{document}
\maketitle
\begin{center}
  This paper is dedicated to Prof. Adimurthi on the occasion of his 70th birthday  
\end{center}

\begin{abstract}
A  well-balanced  second-order finite volume scheme is proposed and analyzed for a $2\times2$ system of non-linear partial differential equations which describes the dynamics of growing sandpiles created by a vertical source on a flat, bounded rectangular table
in multiple dimensions. To derive a second-order scheme, we combine a MUSCL type spatial reconstruction  with strong stability preserving Runge-Kutta time stepping method. The resulting scheme is ensured to be  well-balanced
 through a modified limiting approach that  allows the scheme to reduce to well-balanced first-order scheme near the steady state while maintaining the second-order accuracy away from it. The well-balanced property of the scheme is proven analytically in one dimension and demonstrated numerically in two dimensions. Additionally, numerical
experiments reveal that the second-order scheme reduces  finite time oscillations, takes fewer time iterations for achieving the steady state and gives sharper resolutions of the physical structure of the sandpile, as compared to the existing first-order schemes of the literature.
\end{abstract}


\section{Introduction}
\label{sec:I}
The study of granular
 matter  dynamics has been gaining interest among applied mathematicians in the last few years. 
 A wide array of models can be found in the literature, ranging from kinetic models to hyperbolic differential equations. For detailed discussion of these models, refer to \cite{V21} and the  references therein.
 This area of research has seen numerous endeavors focusing on the theoretical aspects of differential equations, as evidenced by works such as \cite{B:6,CCS1,CCS,10010,V5,12,11}. Additionally, considerable efforts have been devoted to the numerical approximation of these proposed models, as seen in \cite{Vita,V2,adimurthi2016,adimurthi2016a}.

In this work, our focus is on the model equations introduced in \cite{11}, commonly referred to as the Hadler and Kuttler (\textbf{HK}) model. This model comprises of a coupled system of non-linear partial differential equations and is widely recognized for describing the evolution of a sandpile formed by pouring dry sand grains onto a flat and bounded table surface denoted as $\Omega$. The evolution of sandpile is governed by these equations under the influence of a time-independent non-negative vertical source represented by $f\in L^1 (\Omega).$ It is assumed that all sand grains are uniform in size, thus disregarding phenomena like segregation or pattern formation. Additionally, external factors such as wind or stress fields within the bulk of the medium are not taken into account. The \textbf{(HK)} model reads as:\begin{eqnarray}
u_{t}= (1-|\nabla u|)v, \, \, \, \, &\mbox{in}\, \, \, \, \, \, \, & \Omega\times (0, T] \label{HKw1}\\
v_{t}-\nabla. (v\nabla u)=- (1-|\nabla u|)v+f, \, \, \, \, &\mbox{in}\, \, \, \, \, \, \, &\Omega\times (0, T] \label{HKw2}\\
u (\textbf{x}, 0)=u_0 (\textbf{x}), \, \, \, v (\textbf{x}, 0)=v_0 (\textbf{x}), \, \, \, \, \, \, \, &\mbox{in}\, \, \, \, \, \, \, &\Omega \label{HKw3}
\end{eqnarray}
where at any point $(\textbf{x}, t)\in \Omega\times [0, T],$
  $u (\textbf{x}, t)$ denotes the local height of the pile containing the grains at rest and is called the \textit{standing} layer, and $v (\textbf{x}, t)$ denotes the \textit{rolling} layer, formed by the grains 
that roll on the surface of the pile until they are captured by the standing layer. 
Further, the boundary $\partial\Omega$ can be split into two parts: $\Gamma_o$, an open non-empty subset of $\partial\Omega$ where the sand can fall 
down from the table and $\Gamma_w=\partial\Omega\setminus\Gamma_o$ where the sand is blocked by a wall. From modelling point of view, a wall of arbitrary height  can be imagined on $\Gamma_w$ so that no sand can trespass this wall, while on $\Gamma_o$, the table is ``open''.
If $\Gamma_o=\partial\Omega$, then the problem is called as \textit{open table problem}, otherwise it is called \textit{partially open table problem}. The system \eqref{HKw1}-\eqref{HKw2} is supplemented with the following  boundary conditions: \begin{eqnarray}\label{HKw4}
    u=0\, \, \, \, \mbox{in}\, \, \, \, \, \, \Gamma_o, \, \, \, \, v\displaystyle\frac{\partial u}{\partial \nu}=0\, \, \, \, \mbox{in}\, \, \, \, \, \, \Gamma_w
\end{eqnarray}
a detailed discussion can be found in \cite{11,Vita,adimurthi2016,adimurthi2016a}.
For stability reasons, $\abs{\nabla{u}}$ cannot exceed 1. Moreover, at any equilibrium, the profile of  $|\nabla u|$ must be maximal
where transport occurs (that is, where $v>0$). The exchange of the grains between the two layers occurs through an exchange term  $ (|\nabla u|-1)v$ which is independent of the slope orientation
and  can be characterized as erosion/deposition\@.
The equilibrium of the system \eqref{HKw1}-\eqref{HKw4} is given by: 
\begin{equation}\label{EWd}
\begin{split}
-\nabla. (v\nabla u)&=f, \quad \mbox{in} \,\,\Omega,\\
|\nabla u|&=1, \quad  \mbox{on}\,\, \{v>0\}\\
|\nabla u| \leq1,\,u, v&\geq0, \quad \mbox{ in }\, \, \, \,  \Omega\\
 u=0\quad \mbox{in}\,\, \Gamma_o,\,\,\, &v\displaystyle\frac{\partial u}{\partial \nu}=0\, \, \mbox{in}\,\, \Gamma_w
 \end{split}
 \end{equation}
A complete mathematical theory for the existence of solutions of (\textbf{HK}) model at finite time and at equilibrium is still not completely settled and
is not
covered by standard existence and 
uniqueness results available for hyperbolic balance laws, see~\cite{AS1,AS3,Shen,11,B:6,CCS1,CCS} for some limited results.  Recent attention has also been devoted to exploring the slow erosion limit of the model in one dimension, as evident in ~\cite{COCLITE2017337,amadori2012front,colombo2012lipschitz,guerra2014existence,bressan2014semigroup} and references therein. In the past decade, numerous studies have aimed to develop robust numerical  schemes approximating (\textbf{HK}) model, with the ability to preserve the discrete steady states and
the physical properties of the model efficiently, see \cite{Vita,V21,adimurthi2016,adimurthi2016a}. In this context, finite difference schemes capturing the discrete steady states were proposed and analyzed in \cite{Vita,V21} and well-balanced finite volume schemes were developed in \cite{adimurthi2016,adimurthi2016a} using the basic principle of conservation laws with discontinuous flux. The class of finite volume schemes based on conservation laws with discontinuous flux have been used in the last decade for various real life applications, see \cite{burger2018,burger2016,sudarshan2014,adimurthi2016b}. The schemes proposed in \cite{adimurthi2016,adimurthi2016a} were shown to be well-balanced and capable of capturing the sharp crests at the equilibrium state more efficiently than existing methods. However, these schemes exhibited oscillations near the initial condition, persisting for a significant duration, leading to a delay in reaching  the steady state. This issues give rise to an important question: is it possible to control these oscillations and reduce the time taken to reach the steady state by moving to a high-order scheme? Simultaneously, it also put forth the question of whether these high-order scheme could result in a sharper resolution of the discrete steady state.


In various scenarios, 
the well-balanced schemes for hyperbolic systems have been of keen interest over the past few decades, see \cite{bermudez1994upwind,greenberg1996well,gosse2000well,harten1983upstream,berthon2016fully,michel2016well,ghitti2020fully,desveaux2022fully} and the references therein. It has been observed that capturing moving steady states or those with complex structures, like that of \eqref{EWd}, can be a challenging task in general. To the best of our knowledge, there have been no studies on second-order schemes for \eqref{EWd} in the existing literature. However, in the case of shallow water equations, second-order schemes were proposed and analyzed  in \cite{michel2016well,desveaux2022fully}. It has been noted in these studies that the second-order extensions may not inherently possess the well-balanced characteristics, and an adaptation algorithm is required to ensure this property. In this worK to derive a second-order scheme 
we employ a MUSCL-type spatial reconstruction \cite{van1979towards} along with a strong stability preserving Runge-Kutta time stepping method \cite{gottlieb1998otalVD,gottliebsigalshu2001}, which is basically an extension of  the first-order scheme of \cite{adimurthi2016,adimurthi2016a}.
In Section \ref{Well}, we illustrate that this  second-order scheme  is not well-balanced for the state variable $v.$ To overcome this difficulty, we modify the proposed second-order scheme with an adaptation procedure  similar to that of \cite{desveaux2022fully}, to develop a well-balanced second-order scheme. The procedure involves a modified  limitation strategy in the linear reconstruction of the approximate solution at each time step. We establish that the resulting scheme is well-balanced and is able to accurately capture the discrete steady state.

The rest of this  paper is organized as follows. In Section \ref{sec: An-approximation-Scheme}, we focus on deriving the second-order numerical scheme and provide a concise overview of the first-order numerical scheme proposed in \cite{adimurthi2016,adimurthi2016a}. The stability analysis for the second-order scheme is presented in Section \ref{sec: Stability-results}. A discussion about the well-balance property of the second-order scheme is outlined in Section \ref{Well}. In Section \ref{sec:adaptive1d}, we elucidate the second-order adaptive scheme and analytically establish its well-balance property. Section \ref{sec: Scheme-for-2} deals with the extension of the first-order scheme from one dimension to two dimensions, along with the adaptation procedure in the two-dimensional context. In Section \ref{sec: Numerical-Experiment}, we provide numerical examples in both one and two dimensions to showcase the performance of the proposed second-order adaptive scheme in comparison to the non-adaptive second-order scheme and the first-order schemes of \cite{adimurthi2016,adimurthi2016a}. We finally draw our conclusion in Section~\ref{sec:conclusion}.
\section{Numerical schemes in one-dimension}\label{sec: An-approximation-Scheme}
We now present the numerical algorithm approximating \eqref{HKw1}-\eqref{HKw2} in one dimension. First, we briefly review the first-order finite volume schemes of \cite{adimurthi2016,adimurthi2016a} and then present the second-order scheme.
Let $\Omega:= (0, 1).$ As in \cite{adimurthi2016}, we rewrite \eqref{HKw1}-\eqref{HKw3} as follows:
\begin{eqnarray}
u_t+F_1 (\alpha,v)=0,& &\mbox{ in } \Omega\times (0,T]\label{eq:deposit}\\
v_t+F_2 (\alpha,v,B)_x-F_1 (\alpha,v)=0,& & \mbox{ in } \Omega\times (0,T]
\label{eq:erosion2}\\
u (x,0) = v (x,0) = 0,& &  \mbox{ in } \Omega \nonumber
\end{eqnarray}
where  $\alpha=u_x,\,B (x)=\displaystyle\int_0^xf (\xi)d\xi,\, F_1 (\alpha,v)= (|\alpha|-1)v,$ and $ F_2 (\alpha,v,B)= -\alpha v-B (x)$. For $\Delta x,\,\Delta t>0,$ we define $\lambda:=\Delta t/\Delta x$ and consider equidistant spatial grid points $x_{i+\half}:=i\Delta x,$ for non-negative integers $i\in \mathcal{M} :=\{0,1, .....M\}$ and temporal grid points $t^n:=n\Delta t$ 
for non-negative integer $n\in \mathcal{N}_T :=\{0,1, .....N_T\}$, such that $x_{\half}=0,\, x_{M+\half}=1$ and $T=t^{N_T}$. Let $ x_i = \frac{1}{2} (x_\iph+x_\imh)$ for $i\in\{1, .....M\}$.  We also denote the cell $C_i^n:=C_i\times C^n,$ where $C_i:=[x_{i-\half}, x_{i+\half})$ and  $C^{n}:=[t^n,t^{n+1}).$
Let $u_{i+\half}^{n}$ be an approximation of the solution $u$ 
calculated at grid points $x_{i+\half}$ at time $t^n$\@. 
For each $ (i,n)\in \mathcal{M}\times\mathcal{N}_T$, define 
\begin{align*}
\alpha_{i}^{n}:=\frac{u_{i+\half}^{n}-u_{i-\half}^{n}}{\Delta x} \mbox{ and }
\;v_{i}^{n}:=\frac{1}{\Delta x}\int_{C_i}v (x, t^n)dx
  \end{align*}
as the approximation for $\alpha, v$ in the cell $C_{i}^n$. Also, following \cite{adimurthi2016a} we use the notation
$D_f:= \overline{\{ x: f (x)\ne 0\}}.$
\subsection{First-order scheme}\label{sec:foscheme} The first-order finite volume scheme for the system (\ref{HKw1})-(\ref{HKw2}), formulated in \cite{adimurthi2016,adimurthi2016a} is given by
\begin{align}\label{first}
\begin{split}
    u_{i+\half }^{n+1}&=u_{i+\half}^{n}-\Delta t G_{i+\half}^{n},\, (i,n)\in (\mathcal{M}\setminus\{M,0\})\times \mathcal{N}_T\\
v_{i}^{n+1}&=v_{i}^{n}-\lambda (H_{i+\half}^{n}-H_{i-\half}^{n})+\Delta t S_i^{n},\, (i,n)\in (\mathcal{M}\setminus\{0\})\times\mathcal{N}_T
\end{split}
\end{align}
Here, the terms $G^n_{i+\half}$ and $H_{i+\half}^{n}$  are the numerical fluxes associated with the fluxes $F_1$ and $F_2,$ respectively at $ (x_{i+\half},t^n)$ and are given by:
\begin{align*}
    G^n_{{ i+ \half}}= G (\alpha_i^n, v_i^n, \alpha_{i+1}^n,\; v_{i+1}^n), \;H^n_{{ i+ \half}}= H (\alpha_i^n, v_i^n, \alpha_{i+1}^n, v_{i+1}^n,B_i,B_{i+1})
\end{align*}
with 
\begin{align}
G (a,b, c,d)&= \max\{ (|\max\{a,0\}| - 1) b, (|\min\{c,0\}|-1)d \}
\label{alphaflux}\\
\no \\
H (a, b, c, d,e_1,e_2) &= \begin{cases}
 (-ab- e_1), & -a \ge 0, -c \ge 0 \\
 (-cd- e_2), & -a< 0, -c \le 0 \\
\displaystyle\frac{ ( - c e_1+ a e_2)}{ (c-a)},   & -a< 0, -c > 0 \\
 (-a b - e_1), & b> d \;\;{\mbox{and}}\;\;-a \geq 0,-c \leq  0\\
 (-c d - e_2), & b < d,  \;\;{\mbox{and}}\;\;-a \geq 0,-c \leq  0\\
-\frac{1}{2} (ab + cd+ e_1+e_2), & b=d  \;\;{\mbox{and}}\;\;-a \geq 0,-c \leq  0
\end{cases}
\label{hflux}
\end{align} 
The term $S_i^{n}$ is given as 
\begin{equation*}\label{wS01}
 S_i^{n}=v_{i}^{n} (|\alpha_{i}^{n}|-1) 
\end{equation*}
The scheme \eqref{first} is implemented with initial conditions given by
$u_{i+\half}^{0}=0, \mbox{ for all } i \in \mathcal{M} \mbox{ and } v_{i}^{0}=0, \mbox{ for  all } i \in \mathcal{M}\setminus\{0\}$. Regarding the boundary conditions, we mainly consider two types: open table ($\Gamma_o=\{0,1\}$) and partially open table ($\Gamma_o=\{0\}$). In each case, the boundary conditions are specified as follows:

\underline{Boundary condition for $u:$}
\begin{itemize}
\item $\Gamma_o=\{0,1\}:$
$$u_{\half}^{n}=u_{M+\half}^{n}=0$$
\item $\Gamma_o=\{0\}:$ 
\begin{eqnarray*}
\begin{array}{lll}
u^n_\half = 0,\quad 
u_{M+\half}^{n+1}=\left\{\begin{array}{ccl}
u_{M-\half}^{n}&\, \, \mbox{if}\, \, & D_f=[X_1, X_2], X_2<1 \\
u_{M-\half}^{n}+\Delta x\;\max (\alpha_{M-1}^{n}, 0)&\, \, \mbox{if}\, \, & D_f=[X_1, X_2], X_2=1
\end{array}\right. 
\end{array}
\end{eqnarray*}
\end{itemize}
\underline{Boundary conditions for $v:$} The  boundary condition for $v$  are set through the fluxes  as done in \cite{adimurthi2016,adimurthi2016a}.
\begin{itemize}
\item  $\Gamma_o=\{0,1\}:$
\begin{align*}
\begin{split}
    H_{\half}^n=-\alpha_{1}^nv_1^n-B_1, \quad\;& H_{M+\half}^n= -\alpha_M^n v_M^n -B_M
\end{split}
\end{align*}
\item $\Gamma_o=\{0\}:$ 
\begin{align}
H_{\half}^n=-\alpha_{1}^nv_1^n-B_1,\quad
   \label{wF1} H_{M+\half}^{n}=\left\{\begin{array}{ccl}
-\alpha_M^{n}v_M^{n}-B_M& \, \, \mbox{if}\, \, & \alpha_M^{n}\le 0 \\
-B_{M+1}& \, \, \mbox{if}\, \, & \alpha_M^{n}>0
\end{array}\right. 
\end{align}
\end{itemize}
\subsection{Second-order scheme}\label{second} We now describe the  second-order extension of the first-order scheme described in the previous section. To construct a second-order scheme, we employ  a MUSCL type spatial reconstruction and a two stage strong stability preserving Runge-Kutta method in time. To begin with, in each cell $C_i,$ we construct a piecewise linear function $z_{\Delta x}$ which is defined  by
$$z_{\Delta x} (x,t):=z_i^n+\frac{ (x-x_i)}{\Delta x} (z^n_{i+\half,L}-z^n_{i-\half,R}), \;\;z=\alpha, v \mbox{ and } B$$
where
\begin{eqnarray} \label{slop1}
z_{i+\half,L}^{n}=z_i^n+ \half Dz^n_i,\;
z_{i-\half,R}^{n}=z_i^n- \half Dz^n_i
\end{eqnarray}
and
\begin{align}
Dz^n_i&=2 \theta\,\mbox{\mbox{minmod}}\left (z^n_i -z^n_{i-1},\half ({z^n_{i+1}-z^n_{i-1}}), z^n_{i+1}-z^n_i\right),\;\theta \in [0,1]
\label{slope}
\end{align}
represent the slopes obtained using the minmod limiter,
where the minmod function is defined as
\begin{align}
\textrm{minmod} (a_{1}, \cdots , a_{m}):= \begin{cases}\mathrm{sgn} (a_{1}) \min\limits_{1 \leq k \leq  m}\{\lvert a_{k}\rvert\}, \hspace{0.55cm} \mbox{if} \ \mathrm{sgn} (a_{1}) = \cdots = \mathrm{sgn} (a_{m})\\
0, \hspace{3.45cm}\mbox{otherwise}
\end{cases}
\end{align}
Note that $\theta=0$ in \eqref{wF1} gives the first- order scheme, while $\theta=0.5$ gives the usual \mbox{minmod} limiter. 
For each $i$, we can find $\sigma_{i}^L $
  and $\sigma_{i}^R $
 with  $0 \le \sigma_{i}^L,\sigma_{i}^R \leq 1$ such that 
\begin{eqnarray}\label{slope1}
\begin{array}{lll}
z_{i+\half,L}^{n}=z_{i}^n+\theta \sigma_i^R (z_{i+1}^n-z_{i}^n), \quad
z_{i+\half,R}^{n}=z_{i+1}^n-\theta \sigma_{i+1}^L (z_{i+1}^n-z_{i}^n) 
\end{array}
   \end{eqnarray}
This ensures that for  all $ 0 \le \theta \le 1,$ 
$$  \min \{z_i^n, z_{i+1}^n\} \leq z_{i+\half,L}^{n},z_{i+\half,R}^{n} \le \max\{z_i^n, z_{i+1}^n\} $$ 
{\color{black} The time steps of the second-order scheme using the Runge-Kutta method are defined as follows:}\\
\underline{\bf RK step-1}: Define
\begin{align}\label{av*}
\begin{split}
   u_{i+\half}^{n,*}&:=u_{i+\half}^n-  \Delta t G_{ i+ \half}^{n}\\
{v}^{n,*}_i&:=v_i^n- \lambda (H_{{ i+ \half}}^{n}-H_{i-\half}^{n})+\dt S_{ i+ \half, L}^n 
\end{split}
\end{align}
where
\begin{align*}
    G_{{ i+ \half}}^{n} &= G (\alpha_{ i+ \half, L}^{n}, v_{ i+ \half,L}^{n}, \alpha_{ i+\half, R}^{n}, v_{ i+ \half, R}^{n})\\
    H_{{ i+ \half}}^{n} &=H (\alpha_{ i+\half,L}^{n}, v_{ i+\half, L}^{n}, \alpha_{ i+ \half, R}^{n}, v_{ i+\half, R}^{n},B_{ i+ \half,L}, B_{ i+ \half, R})\\
    S_{ i+ \half, L}^n&=-g (\alpha_{i+\half,L}^n, v_{i+\half,L}^n)
\end{align*}
and $g (a,b)= (|a|-1)b.$ Also, we have 
\begin{align}\label{eq:alphark}
    \alpha_i^{n,*} = \frac{ u^{n,*}_{i+\half}-u^{n,*}_{i-\half}} {\Delta x}
\end{align}
\noi \underline{\bf RK step-2:} Given the values $\alpha_i^{n,*}$  and $v^{n,*}_{i}$ from RK step-1, we proceed to construct corresponding $\alpha_{ i+ \half,L}^{n,*}, \alpha_{ i+ \half, R}^{n,*}, v_{ i+ \half,L}^{n,*}$  and $v_{ i+ \half, R}^{n,*} $ by using \eqref{slop1}-\eqref{slope} for $z=\alpha^*,v^*$  and $B^*$. Here, we define 
\begin{align*}
\begin{split}
   u_\iph^{n,**}&:=u_\iph^{n,*}-  \lambda ( G_{ i+ \half}^{n,*} - G_{i-\half}^{*}) \\
v^{n,**}_{i}&:=v^{n,*}_{i}- \lambda (H_{{ i+ \half}}^{n,*}-H_{i-\half}^{n,*})+\dt S_{ i+ \half,L}^{n,*} 
\end{split}
\end{align*}
where
\begin{align*}
G_{{ i+ \half}}^{n,*} &= G (\alpha_{ i+ \half,L}^{n,*}, v_{ i+ \half,L}^{n,*}, \alpha_{ i+ \half,R}^{n,*}, v_{ i+ \half, R}^{n,*}) \no\\
H_{{ i+ \half}}^{n,*}& = H (\alpha_{ i+ \half,L}^{n,*}, v_{ i+ \half,L}^{n,*}, \alpha_{ i+ \half,R}^{n,*}, v_{ i+ \half,R}^{n,*},B_{{ i+\half},L},B_{{ i+ \half},R })\no\\
S_{ i+ \half, L}^{n,*}&=-g (\alpha_{i+\half,L}^{n,*} v_{i+\half,L}^{n,*})
\end{align*}
Finally, the second-order scheme is written as
\begin{equation}
    \begin{split}
u_\iph^{n+1}&=\half ({u_\iph^n+u_\iph^{n,**}}{})\label{rkfalpha}  \\
v_i^{n+1}&=\half ({v_i^n+v^{n,**}_{i}}{})
    \end{split}
\end{equation}
\noindent

\subsection{Boundary conditions for the second-order scheme}\label{sec:bcsocheme}
\underline{Case 1:} (open table problem, $\Gamma_0= \{ 0,1\}$)
in the second-order case, the computation of fluxes  for $i = \frac{3}{2}$ and $ i= M-\half$ requires the ghost cell values for the linear reconstruction. To simplify this, we set the slopes $Dz_i^n$ in the first and last cells to zero. Further, 
the boundary condition for $u$ is implemented in the same way as in the first-order scheme given in Section~\ref{sec:foscheme}, specifically  for grids corresponding to $i = \half $ and $ M+\half.$ Similarly for the case of $v.$

\noindent
\underline{Case 2:} (partially open table problem with $\Gamma_0 =\{0\}$) in this case as well, we set the slopes   $Dz_i^n$  in the first and last cells to zero, and the boundary conditions for $u$ and $v$ at the left boundary, where $x =0$ are same as the first-order case in Section~\ref{sec:foscheme}. At the right boundary, where $x =1,$ the boundary condition for $v$ remains the same as in Section~\ref{sec:foscheme}. However, the approximation for $u$  at this boundary is evolved as
\begin{eqnarray*}
\begin{array}{lll}
u_{M+\half}^{n+1}=\left\{\begin{array}{ccl}
u_{M-\half}^{n},&\, \, \mbox{if}\, \, & D_f=[X_1, X_2], X_2<1 \\
u_{M-\half}^{n}+ \Delta x \max (\alpha_{M-\half,L}^{n}, 0),&\, \, \mbox{if}\, \, & D_f=[X_1, X_2], X_2=1
\end{array}\right. 
\end{array}\
\end{eqnarray*}

\section{Stability results in  one-dimension\label{sec: Stability-results}}
In this section, we show that the numerical solutions obtained using the second-order scheme \eqref{rkfalpha} satisfy  the physical properties  in the case of open table problem. Simlarly, one can derive the results for the partially open table problem.
\begin{theorem}\label{stablity}
 Let $f\ge0$ in $\Omega$\@. Assume that 
 ${\displaystyle \sup_{i\in \mathbb{Z}} |\alpha_i^0|\leq1},$ 
  ${u_\iph^0}\ge0$ and
  $v_i^{0}\ge0$, for all $i\in \mathbb{Z},$
then the numerical scheme (\ref{rkfalpha}) under the CFL conditions:
 \begin{equation}
{\displaystyle {\lambda \max_{i}v_{i}^{n}\le \half}},\label{CFL1}\end{equation}
\begin{equation}
 {\lambda\leq \displaystyle \half-\Delta t}\label{CFL2}\, \, \end{equation} 
satisfies the following properties for all $i \in \mathbb{Z}$:
\begin{enumerate}[ (i)]
\item \label{P1} ${\displaystyle \sup_{i\in \mathbb{Z}}|\alpha_i^{n+1}|\leq 1}$, $u_\iph^{n+1}\ge u_\iph^{n}\ge 0$
\item \label{P0} $v_i^{n+1}\ge 0$
 \end{enumerate}

\end{theorem}
\begin{proof}
    From \eqref{eq:alphark} we can write
\begin{align*}
\alpha_{i}^{n,*}&=\half ({\alpha_{ i+ \half, L}^{n}+\alpha_{i-\half,R}^{n}}{})-  \lambda ( G_{ i+ 1/2}^{n} - G_{i-1/2}^{n}) \\
&=K_1 (\alpha_{i-\half,L}^{n},\alpha_{i-\half, R}^{n},\alpha_{ i+ \half,L}^{n}, \alpha_{ i+\half, R}^{n},v_{i-\half,L}^{n},v_{i-\half,R}^{n}, v_{ i+ \half, L}^{n},v_{ i+ \half,R}^{n}) 
\end{align*}
Now, we aim to show that $K_1$ is non-decreasing in $ \alpha_{i\pm\half,L}^{n}$ and $\alpha_{i\pm \half, R}^{n}$ under the  CFL conditions (\ref{CFL1}).
For each $i,$ we have 
\begin{align*}
\displaystyle\frac{\partial G_{ i+ \half}^{n}}{\partial \alpha_{i-\half,L}^{n}}&=0 \\  
\displaystyle\frac{\partial G_{ i+ \half}^{n}}{\partial \alpha_{ i+ \half, L}^{n}}&=
\left\{ \begin{array}{lll}
v_{ i+ \half,L}^{n}&\mbox{if}\, \, \, G_{ i+ \half}^{n}= (|\max (\alpha_{ i+\half, L}^{n}, 0)|-1)v_{\iph,L}^{n}\, \, \mbox{and}\, \, \alpha_{ i+ \half, L}^{n}>0 \\ 
0&\mbox{otherwise}
\end{array}\right.\label{HH1} 
\end{align*}
and
\begin{equation*}\label{HH2} \displaystyle\frac{\partial G_{ i+ \half}^{n}}{\partial \alpha_{ i+ \half, R}^{n}}=
\left\{ \begin{array}{lll}
-v_{ i+ \half, R}^{n}&\mbox{if}\, \, \, G_{ i+ \half}^{n}= (|\min (\alpha_{ i+ \half,R}^{n}, 0)|-1)v_{i+\half, R}^{n}\, \, \mbox{and}\, \, \alpha_{ i+ \half,R}^{n}<0 \\
0&\mbox{otherwise}
\end{array}\right.
\end{equation*}
This shows that $G^n_{ i+ \half}$ is non-decreasing in $\alpha_{ i+ \half, L}^{n}$ and non-increasing in $\alpha_{ i+ \half, R}^{n}$. Hence, we have\\
$$\displaystyle\hspace{-.3cm} \frac{\partial K_1}{\partial \alpha_{i-\half, L}^{n}}=\lambda\displaystyle \frac{\partial G_{i-\half}^{n}}{\partial \alpha_{i-\half,L}^{n}}\ge0 \mbox{ and }
\displaystyle\hspace{.2cm}\frac{\partial K_1}{\partial \alpha_{ i+ \half, R}^{n}}=-\lambda\displaystyle \frac{\partial G_{ i+ \half}^{n}}{\partial \alpha_{ i+ \half,R}^{n}}\ge0$$ Further, using the CFL condition (\ref{CFL1}), it yields
$$\displaystyle\hspace{.2cm} \frac{\partial K_1}{\partial \alpha_{ i+ \half,L}^{n}}=\half-\lambda \frac{\partial G_{ i+ \half}^{n}}{\partial \alpha_{ i+ \half, L}^{n}}=\half- \lambda v_{ i+ \half, L}^{n} \ge 0 \; \mbox{ and}
\displaystyle\hspace{.2cm} \frac{\partial K_1}{\partial \alpha_{i-\half, R}^{n}}=\half+\lambda \frac{\partial G_{i-1/2}^{n}}{\partial \alpha_{i-\half,R}^{n}}=\half-\lambda v_{i- \half,R}^{n} \ge 0 $$ 
This proves that   $K_1$ is increasing in  each of its variables  $ \alpha_{i\pm\half,L}^{n}$ and $\alpha_{i\pm \half, R}^{n}$   under the CFL condition (\ref{CFL1}). Thus we can write
\begin{align*}
-1&=K_1 (-1, -1, -1,-1, v_{i-\half,L}^{n},v_{i-\half, R}^{n}, v_{ i+\half,L}^{n},v_{ i+ \half,R}^{n})\no \\
&\leq K_1 (\alpha_{i-\half,L}^{n},\alpha_{i-\half,R}^{n},\alpha_{ i+\half,L}^{n}, \alpha_{ i+ \half,R}^{n},v_{i-\half,L}^{n},v_{i-\half,R}^{n}, v_{ i+\half,L}^{n},v_{ i+ \half,R}^{n})\\ &=\alpha_i^{n,*}
\leq  K_1 (1, 1, 1,1, v_{i-\half,L}^{n},v_{i-\half,R}^{n}, v_{ i+\half,L}^{n},v_{ i+ \half,R}^{n})=1\no \end{align*}
i.e., 
\begin{equation*}
-1\leq\alpha_{i}^{n,*}\leq1
\end{equation*}
By setting 
\begin{eqnarray*}
\alpha_i^{n,**}=K_1 (\alpha_{i-1/2 L}^{n,*},\alpha_{i-1/2 R}^{n,*},\alpha_{ i+ \half, L}^{n,*}, \alpha_{ i+ \half,R}^{n,*},v_{i-\half, L}^{n,*},v_{i-\half,R}^{n,*}, v_{ i+\half,L}^{n,*},v_{ i+\half,R}^{n,*})
\end{eqnarray*}
under the CFL condition (\ref{CFL1}), it can be similarly shown that
\begin{equation*}
-1\leq\alpha_{i}^{n,**}\leq1
\end{equation*} 
Since 
\begin{align*}
    \alpha_i^{n+1} = \half ( \alpha_i^{n}+\alpha_i^{n,**})
\end{align*}
it follows that
\begin{equation*}
 -1\leq\alpha_{i}^{n+1}\leq1
\end{equation*}
under the CFL condition (\ref{CFL1}). 
Note that $ G_{{ i+ \half}}^n \le 0$ as we have $-1 \le \alpha^n_{i\pm \half,L}, \alpha^n_{i\pm \half,R} \le 1.$ Consequently, from the expression
\begin{align*}
    u^{n,*}_\iph = u^n_\iph -\Delta t G_\iph^n
\end{align*}
it gives that   $u^{n,*}_\iph \ge u^n_\iph.$
In the similar way, given that $-1\le \alpha^{n,*}_{i\pm \half,L}, \alpha^{n,*}_{i\pm \half,R} \le 1,$ we can see that  
$u^{n,**}_\iph \ge u^{n,*}_\iph  \ge u^n_\iph.$
Further, 
\begin{align*}
    u^{n+1}_{i+\half} = \half ( u^{n}_\iph +u^{n,**}_\iph) \ge u^{n}_\iph
\end{align*}
This proves \eqref{P1}.

Now, let us consider the equation 
 \[v_i^{n,*}=v_{i}^{n}-\lambda (H_{i+\half}^{n}-H_{i-\half}^{n})+\Delta t v_{ i+ \half,L}^{n} (|\alpha_{ i+ \half,L}^{n}|-1) \] 
and denote
 \[ v_i^{n,*}=K_2 (\alpha_{i-\half,L}^{n},\alpha_{i-\half, R}^{n},\alpha_{ i+ \half, L}^{n}, \alpha_{ i+ \half,R}^{n},v_{i-\half, L}^{n},v_{i-\half,R}^{n}, v_{ i+\half, L}^{n},v_{ i+ \half,R}^{n})\]
 To prove that $K_2$ is non-decreasing in each variables $v_{i\pm \half, L}^{n}$ and $v_{i\pm \half, R}^{n},$ we consider the following two cases and  other cases follow in a  similar manner.

 \textbf{Case 1: } $-\alpha_{i\pm \half,L}^{n}\ge0, -\alpha_{i\pm \half, R}^{n} \geq 0$.\label{Case 1} With these condition, the numerical fluxes now read as 
 \begin{align*}
     H_{i\pm \half}^{n}&=-\alpha_{ i\pm \half,L}^{n}v_{ i\pm \half,L}^{n}-B_{ i\pm \half,L}
 \end{align*}
 and we have \begin{align*}
    K_2&=\half ({v_{ i+ \half,L}^{n}+v_{i-\half,R}^{n}}{})-\lambda (-\alpha_{ i+ \half,L}^{n}v_{ i+ \half,L}^{n}-B_{ i+ \half,L})\\&\quad+\lambda (-\alpha_{i-\half,L}^{n}v_{i-\half, L}^{n}-B_{i-\half,L}))+\Delta tv_{ i+ \half,L}^{n} (|\alpha_{ i+ \half,L}^{n}|-1)
\end{align*} Upon differentiation, we find:
\begin{align*}
   \frac{\partial K_2}{\partial v_{i-\half, L}^{n}}&=\displaystyle \frac{\partial H_{i-\half}^{n}}{\partial v_{i-\half,L}^{n}}=-\lambda\alpha_{i-\half,L}^{n}\geq0\\
   \frac{\partial K_2}{\partial v_{ i+ \half,R}^{n}}&=0,\quad 
   \frac{\partial K_2}{\partial v_{i-\half,R}^{n}}=\half \\
   \frac{\partial K_2}{\partial v_{ i+ \half, L}^{n}}&=\half-\lambda (-\alpha_{ i+ \half, K_2}^{n})+\Delta t (|\alpha_{ i+ \half, L}^{n}|-1)\\
   &=\half-\lambda|\alpha_{ i+\half,L}^{n}|+\Delta t (|\alpha_{ i+ \half,L}^{n}|-1)
\end{align*}
Further, together with the CFL condition \eqref{CFL2} and  the assumption $|\alpha_{ i+\half, L}^{n}| \leq 1,$ it is easy to see that
$$|\alpha_{ i+ \half, L}^{n}|\lambda\leq\displaystyle \half+\Delta t (|\alpha_{ i+ \half,L}^{n}|-1)$$ This inequality leads to 
$\displaystyle \frac{\partial K_2}{\partial v_{ i+\half, L}^{n}}\geq0$ in the above expressions. Therefore, it can be concluded that under the CFL condition \eqref{CFL2} $K_2$ is non-decreasing in each of its variables $v_{i\pm \half, L}^{n}$ and $v_{i\pm \half, R}^{n}.$ \\
 \textbf{Case 2}: $-\alpha_{i-\half, L}^{n}\ge0, -\alpha_{ i+ \half, L}^{n}\le 0, -\alpha_{i\pm \half, R}^{n}\le0.$ \label{Case 3}
We prove for the case $-\alpha_{i-\half, R}^{n}<0$ and $v_{i-\half, R}^{n}=v_{i-\half, L}^{n},$ and other cases can be proved similarly. The numerical fluxes are given by
\begin{align*}
   H_{i+\half}^{n}&=-\alpha_{ i+ \half, R}^{n}v_{ i+ \half, R}^{n}-B_{{ i+\half,R}}\\
  H_{i-\half}^{n}&=\half (
\displaystyle {-\alpha_{i-\half, L}^{n}v_{i-\half, L}^{n}-B_{i-\half,L}-\alpha_{i-\half, R}^{n}v_{i-\half, R}^{n}-B_{i-\half,R }}{})
\end{align*}
and
\begin{align*}
    K_2&=\half ({v_{ i+ \half, L}^{n}+v_{i-\half, R}^{n}}{})+\frac{\lambda}{2} ({-\alpha_{i-\half, L}^{n}v_{i-\half, L}^{n}-B_{i-\half,L}-\alpha_{i-\half, R}^{n}v_{i-\half, R}^{n}-B_{i-\half,R}}{}) \\
    &\quad-\lambda (-\alpha_{ i+ \half, R}^{n}v_{ i+ \half, R}^{n}-B_{{ i+ \half,R}}) 
+\Delta t v_{ i+ \half, L}^{n} (\alpha_{ i+ \half,L}^{n}-1)
\end{align*}
 Upon differentiating and by (\ref{CFL2}) it yields
 \begin{align*}
      \frac{\partial K_2}{\partial v_{i-\half, L}^{n}}&=\displaystyle \frac{\partial H_{i-\half,}^{n}}{\partial v_{i-\half, L}^{n}}=-\frac{\lambda\alpha_{i-\half, L}^{n}}{2}\geq0\\
      \frac{\partial K_2}{\partial v_{i-\half,L}^{n}}&=\displaystyle \frac{\partial H_{i-\half}^{n}}{\partial v_{i-\half, L}^{n}}=-\frac{\lambda\alpha_{i-\half, L}^{n}}{2}\geq0\\
      \frac{\partial K_2}{\partial v_{ i+ \half, R}^{n}}&=\lambda\alpha_{ i+ \half,R}^{n}\ge0\\
      \frac{\partial K_2}{\partial v_{ i+ \half, L}^{n}}&=\half+\Delta t (\alpha_{ i+ \half, L}^{n}-1)\ge 0\\
      \displaystyle \frac{\partial K_2}{\partial v_{i-\half, R}^{n}}&=\half-\frac{\lambda \alpha_{i-\half, R}^{n}}{2} \ge 0
 \end{align*}
This proves that $K_2$  is non-decreasing in each of its variables $v^n_{i\pm \half,L}$ and $v^n_{i\pm \half,R}$  under the condition (\ref{CFL2}). Consequently, we obtain $v^{n,*}_i\ge 0$ as: 
\begin{align*}
 0&\le K_2 (\alpha_{i-\half, L}^{n},\alpha_{i-\half, R}^{n},\alpha_{ i+ \half, L}^{n}, \alpha_{ i+\half, R}^{n},v_{i-\half,L}^{n},0, 0,0) \no \\
 &\le K_2 (\alpha_{i-\half, L}^{n},\alpha_{i-\half,R}^{n},\alpha_{ i+ \half,L}^{n}, \alpha_{ i+ \half, R}^{n},v_{i-\half, L}^{n},v_{i-\half,R}^{n}, v_{ i+ \half,L}^{n},v_{ i+ \half, R}^{n})=v_i^{n,*} \end{align*}
In  a similar way  and using the fact that $v^{n,*}_i\ge 0,$ we can show that $v_{i}^{n,**}\geq0$. Subsequently, through (\ref{rkfalpha}), it follows  that $v_{i}^{n+1}\geq0.$ This completes the proof of \eqref{P0}.

\end{proof}

\section{Open table problem in one-dimension and well-balance property} \label{Well}
We now show that the the second-order scheme \eqref{rkfalpha} is not well-balanced in general,
 i.e. the numerical scheme does not capture the steady
 state solution of \eqref{EWd}. 
\begin{theorem}\label{NW}
    The second-order scheme \eqref{rkfalpha}  is well balanced in the state variable $u$ but not in $v$.
\end{theorem}\begin{proof}
Let  $ (\overline{u}, \overline{v})$  denote the steady state solution of \eqref{eq:deposit}-\eqref{eq:erosion2} (a solution of \eqref{EWd}  in one dimension).
We take the particular case of $f\equiv 1$ in $[0,1].$
In this case, the exact steady state solutions are given by
\begin{eqnarray}\label{A}
 (\overline{u},\overline{v}) (x)=\left\{\begin{array}{ccl}
 (x, \half-x)& \, \mbox{if}\, & x\in[0, \displaystyle \half] \\
 (1-x,x-\half)& \, \mbox{if}\, & x\in (\displaystyle \half, 1]
\end{array}\right.\end{eqnarray} From this we have
\begin{eqnarray*}\overline{\alpha} (x)=\left\{\begin{array}{ccl}
1& \, \mbox{if}\, & x\in[0, \displaystyle \half] \\
-1& \, \mbox{if}\, & x\in (\displaystyle \half, 1]
\end{array}\right.\end{eqnarray*}  where $\overline{\alpha}= \overline{u}_x.$ Now, the discrete values of the steady states are given by  
\begin{eqnarray}\label{eqd} (\overline{\alpha}_i,\overline{v}_i)=\begin{cases}
 (1,\half-x_i),\,\,\,&i\le K\\
 (-1,x_i-\half),\,\,\,&i>K
\end{cases}
\end{eqnarray} with $x_{K+\half}=0.5.$
 To prove that the scheme is  well-balanced in the state variable $u$ and not in $v$, we substitute the discrete form \eqref{eqd} in the second-order scheme as
an initial  data and show that
\begin{align*}
u^{1}_\iph &= u_\iph^0 = \overline{u}_\iph, \quad \mbox{ for all } i  \in\mathcal{M}
\end{align*}
and
\begin{align*}
v^1_{i} &\ne v^0_i = \overline{v}_i \quad \mbox{ for all } i\in\mathcal{M}\setminus\{0\}
\end{align*}

where $\overline{u}_\iph$ is the discrete value of the steady state $\overline{u}$ at $x_\iph.$ 
It is important to note that, $D{\alpha}_i^0=0$ in  \eqref{slope},  and  consequently we have $
{\alpha}_{i +\half,L}^{0}={\alpha}_{i-\half,R}^{0}={\alpha}_i^0.
$  This leads to the simplification:  
$$
G_{{ i+ \half}}^{0} = G ({\alpha}_{ i+ \half,L}^{0}, {v}_{ i+ \half,L}^{0}, {\alpha}_{ i+ \half,R}^{0}, {v}_{ i+ \half, R}^{0})=0
$$
This gives that $u^{0,*}_\iph = u^{0}_\iph =\overline u_\iph.$ Subsequently, we have $
   {\alpha}_{i}^{0,*}={\alpha}^0_i$
and hence $
{\alpha}_{i+\half, L}^{0,*}=
{\alpha}_{i-\half, R}^{0,*} = {\alpha}_i^{0,*}.$ This, in turn, implies 
\begin{align*}
    {u}_{\iph}^{0,**}={u}_\iph^{0,*} = \overline{u}_\iph
\end{align*}
and 
\begin{eqnarray*}\label{a2}{u}_\iph^{1}=\half ({u}^0_\iph+{u}_\iph^{0,**})=\overline{u}_\iph
\end{eqnarray*}
This shows that  the second-order scheme is well-balanced in $u.$
Now, let us consider the case for  ${v}:$
note that
 \begin{align}\label{v**}
v_{i}^{0}-v_{i-1}^0&= \begin{cases}
-\Delta x  ,\,\,\,&i\le K \\
\Delta x,\,\,\,&i\ge K+2\\
0,\,\,\,&i=K+1
\end{cases}
\end{align}
since $x_{K+\half}=\half.$ From the definiton,
we have
$$D{{v}}_i^0=2\theta\begin{cases}
-\Delta x,\,\,\,&i\le K-1\\
\Delta x,\,\,\,&i\ge K+2\\0,\,\,\,&i=K,K+1
\end{cases}$$
Now, we compute $v_i^{0,*}$ from the expression
\begin{align}
v^{0,*}_i&=v^0_i- \lambda (H_{{ i+ \half}}^{0}-H_{i-\half}^{0})+\dt S_{ i+ \half, L}^0
\label{eq:v*scheme}
\end{align}
First we observe that $S_{ i+ \half,L}^0=0\; \mbox{ for all } i\in \mathcal{M}$ as  we have $\abs{\alpha^0_{i+\half,L}}=1.$
Now, in the subsequent steps, we calculate the fluxes $H^0_{i\pm\half}$  in \eqref{eq:v*scheme}. For this, first we consider the slopes
\begin{eqnarray*}
    D{{B}}_i=2 \theta\,\mbox{\mbox{minmod}}\left (B_i -B_{i-1},\half ({B_{i+1}-B_{i-1}}), B_{i+1}-B_i\right)
\end{eqnarray*}
and note that $D{{B}}_i=2 \theta\Delta x,$ as $B_i=x_i.$ This results in the  following values 
$$B_{i+\half,L}=x_i+\theta\Delta x \mbox{ and } B_{i+\half,R}=x_{i+1}-\theta\Delta x$$ 
Given that 
\begin{align*}
    \alpha_i^0 =\overline\alpha_i^0=\begin{cases}
1,\,\,\,&i\le K\\
-1,\,\,\,&i>K
\end{cases}
\end{align*} 
we obtain ${\alpha}_{i +\half,L}^{0}={\alpha}_{i-\half,R}^{0} = \alpha_i^0.$ Now, using the values $\alpha_{i+\half,L}^0, \alpha_{i+\half,R}^0, v_{i+\half,L}^0, v_{i+\half,R}^0, B_{\iph,L}$  and $B_{\iph,R}$ in \eqref{hflux}, we write the numerical flux as
\begin{align*}
H_{i+\half}^0&=\begin{cases}
v^{0}_{\iph,L}-B_{\iph,L},\,\,\,&i \ge K+1 \\
-v^{0}_{\iph,R} -B_{\iph,R} ,\,\,\,&i \le K-1 \\
-\half ( B_{\iph,L} +B_{\iph,R}),\,\,\,&i=K
\end{cases}\\
&= \begin{cases}
v^0_i+\half Dv_i^0-x_{i} - \theta\Delta x,\,\,\,&i\ge K+1 \\-v^0_{i+1}+\half Dv^0_{i+1}-x_{i+1} +\theta\Delta x ,\,\,\,&i\le K-1 \\
-\half (x_i+x_{i+1}) ,\,\,\,&i=K
\end{cases}
\end{align*}
which reduces to
\begin{eqnarray*}
H_{i+\half}^0= -\half+\theta\begin{cases}
 0, \,\,\,&i\ge K+2 \\
-\Delta x, \,\,\,&i= K+1\\
2\Delta x, \,\,\,&i\le K-2 \\
 \Delta x, \,\,\,&i= K-1 \\
0, \,\,\,& i = K
\end{cases}
\end{eqnarray*}
By calculating the difference of the fluxes 
\begin{eqnarray*}
H_{i+\half}^0-H_{i-\half}^0= \theta\begin{cases}
0  ,\,\,\,&i\ge K+3, i\le K-2 \\
\Delta x,\,\,\,&i= K+2 \\
-\Delta x,\,\,\,&K-1 \le i \le K+1
\end{cases}
\end{eqnarray*}
and inserting in \eqref{eq:v*scheme}, it yields
\begin{eqnarray}\label{v*}
v_i^{0,*}=v_i^{0}  +\theta \begin{cases}
0,\,\,\,&i\ge K+3, i\le K-2 \\
-\Delta t,\,\,\,&i= K+2 \\
\Delta t,\,\,\,&K-1 \le i \le K+1
\end{cases}
\end{eqnarray}
Now, to compute $v_i^{0,**},$ we need the following slopes
\begin{eqnarray*}
    D{{v}}_i^{0,*}=2 \theta\,\mbox{\mbox{minmod}}\left ({{v}}_i^{0,*} -{{v}}_{i-1}^{0,*},\half ({v^{0,*}_{i+1}-v^{0,*}_{i-1}}), v^{0,*}_{i+1}-v^{0,*}_i\right)
\end{eqnarray*}
By using \eqref{v*} and \eqref{v**}, we have
\begin{align*}
Dv_i^{0,*}&= 2 \theta\begin{cases}
\,\mbox{\mbox{minmod}}\left (\Delta x+\theta\Delta t,\half (2\Delta x+\theta \Delta t),\Delta x\right) ,\,\,\,&i= K+3\\
\,\mbox{\mbox{minmod}}\left (-\Delta x,\half (\theta\Delta t -2\Delta x),-\Delta x+\theta\Delta t \right) ,\,\,\,&i= K-2\\
\Delta x,\,\,\,&i\ge K+4\\
-\Delta x,\,\,\,&i
\le K-3\\
2 \theta\,\mbox{\mbox{minmod}}\left (-\Delta x+\Delta t\theta ,\half (\theta\Delta t -2\Delta x),-\Delta x\right) ,&i= K-1 \\
\,\mbox{\mbox{minmod}}\left (\Delta x-2\theta\Delta t,\Delta x-\half\theta\Delta t),\Delta x+\theta\Delta t\right) ,\,\,\,&i= K+2\\ 
0 ,\,\,\,&i= K,K+1
\end{cases}\\
&= 2 \theta\begin{cases}
\Delta x,\,\,\,&i\ge K+3\\
\Delta x-2\theta\Delta t ,\,\,\,&i= K+2\\ 
0 ,\,\,\,&i= K,K+1\\
-\Delta x+\theta\Delta t,\,\,\,&i= K-1,K-2\\
 -\Delta x,\,\,\,&i\le  K-3
\end{cases}
\end{align*}
Using this values, we compute $v_{i+\half,L}^{0,*}$ and $v_{i+\half,R}^{0,*}$ and there by obtain the  flux $H_{i+\half}^{0,*}$ ( by \eqref{hflux}) as
\begin{align*}
H_{i+\half}^{0,*} & =  \begin{cases}
v^{0,*}_{\iph,L}-B_{\iph,L},\,\,\,&i \ge K+1 \\
-v^{0,*}_{\iph,R} -B_{\iph,R} ,\,\,\,&i \le K-1 \\
-\half ( B_{\iph,L} +B_{\iph,R}),\,\,\,&i=K
\end{cases}\\
&= \begin{cases}
v^{0,*}_i+\half Dv_i^{0,*}-x_{i} - \theta\Delta x,\,\,\,&i \ge K+1 \\
-v^{0,*}_{i+1}+\half Dv^{0,*}_{i+1}-x_{i+1} +\theta\Delta x ,\,\,\,&i \le K-1 \\
-\half (x_i+x_{i+1}) ,\,\,\,&i=K
\end{cases}
\end{align*}
By substituting  the value of $v_i^{0,*}$ from \eqref{v*} in the above expression, we get
\begin{align*}
H_{i+\half}^{0,*} &= \begin{cases}
v^{0}_i+\theta\Delta t-x_{i} - \theta\Delta x,\,\,\,&i=K+1 \\
v^{0}_i-\theta\Delta t+\theta\Delta x-2 \theta^2\Delta t-x_{i} - \theta\Delta x,\,\,\,&i= K+2 \\
v^{0}_i+\theta\Delta x-x_{i} - \theta\Delta x,\,\,\,&i\ge  K+3 \\-v^{0}_{K}-\theta\Delta t-x_{i+1} +\theta\Delta x ,\,\,\,&i= K-1 \\
-v^{0}_{K-1}-\theta\Delta t-\theta\Delta x+\theta^2\Delta t-x_{i+1} +\theta\Delta x ,\,\,\,&i= K-2 \\
-v^{0}_{K-2}-\theta\Delta x+\theta^2\Delta t-x_{i+1} +\theta\Delta x ,\,\,\,&i= K-3 \\
-v^{0}_{i+1}-\theta\Delta x-x_{i+1} +\theta\Delta x ,\,\,\,&i\le K-4 \\
-\half ,\,\,\,&i=K
\end{cases}
\end{align*}
 Now, taking the flux difference, we arrive at:
\begin{eqnarray*}
H_{i+\half}^{0,*}-H_{i-\half}^{0,*}= \theta\begin{cases}
\Delta t - \Delta x,\,\,\,&i=K+1 \\
-2\Delta t-2 \theta\Delta t+ \Delta x ,\,\,\,&i= K+2 \\
\Delta t+2\theta\Delta x,\,\,\,&i\ge  K+3 \\
\Delta x -\theta\Delta t,\,\,\,&i= K-1 \\
-\Delta t,\,\,\,&i= K-2 \\
\theta\Delta t,\,\,\,&i= K-3 \\
0,\,\,\,&i\le K-4 \\
\Delta t-\Delta x,\,\,\,&i=K
\end{cases}
\end{eqnarray*}
Finally, we get
\begin{eqnarray*}
v_{i}^{0,**}= v_{i}^{0,*}-\lambda\theta\begin{cases}
\Delta t - \Delta x,\,\,\,&i=K+1 \\
-2\Delta t-2 \theta\Delta t+ \Delta x ,\,\,\,&i= K+2 \\
\Delta t+2\theta\Delta x,\,\,\,&i\ge  K+3 \\
\Delta x -\theta\Delta t,\,\,\,&i= K-1 \\
-\Delta t,\,\,\,&i= K-2 \\
\theta\Delta t,\,\,\,&i= K-3 \\
0,\,\,\,&i\le K-4 \\
\Delta t-\Delta x,\,\,\,&i=K
\end{cases}
\end{eqnarray*}
Once again  using the value of $v_i^{0,*}$ from \eqref{v*} in the above expression, we deduce that
\begin{eqnarray*}
v_{i}^{0,**}= v_{i}^{0}+\theta\begin{cases}
\Delta t-\lambda (\Delta t - \Delta x),\,\,\,&i=K+1 \\
-\Delta t-\lambda (-2\Delta t-2 \theta\Delta t+ \Delta x) ,\,\,\,&i= K+2 \\
-\lambda (\Delta t+2\theta\Delta x),\,\,\,&i\ge  K+3 \\
+\Delta t-\lambda (\Delta x -\theta\Delta t),\,\,\,&i= K-1 \\
-\lambda (-\Delta t),\,\,\,&i= K-2 \\
-\lambda (\theta\Delta t),\,\,\,&i= K-3 \\
0,\,\,\,&i\le K-4 \\
\Delta t-\lambda (\Delta t-\Delta x),\,\,\,&i=K
\end{cases}
\end{eqnarray*}
Noting that  the solution $v$ at the single time step is given by
\begin{align*}
{v}_i^{1}&=\half ({{v}^{0}_{i}+{v}^{0,**}_{i}}{} )
\end{align*}
we finally arrive at:
\begin{align*}
{v}_i^{1}&= v_{i}^{0}+\frac{\theta}{2}\begin{cases}
\Delta t-\lambda (\Delta t - \Delta x),\,\,\,&i=K+1 \\
-\Delta t-\lambda (-2\Delta t-2 \theta\Delta t+ \Delta x) ,\,\,\,&i= K+2 \\
-\lambda (\Delta t+2\theta\Delta x),\,\,\,&i\ge  K+3 \\
\Delta t-\lambda (\Delta x -\theta\Delta t),\,\,\,&i= K-1 \\
-\lambda (-\Delta t),\,\,\,&i= K-2 \\
-\lambda (\theta\Delta t),\,\,\,&i= K-3 \\
0,\,\,\,&i\le K-4 \\
\Delta t-\lambda (\Delta t-\Delta x),\,\,\,&i=K
\end{cases}
\end{align*}
This shows that the second-order scheme \eqref{rkfalpha}  is not well balanced for $\theta\ne 0.$
\end{proof}
\section{Second-order adaptive scheme in one-dimension}\label{sec:adaptive1d}
To produce a well-balanced second-order scheme, we propose a modification to the second-order scheme \eqref{rkfalpha} based on the
idea introduced in \cite{michel2016well,ghitti2020fully,desveaux2022fully}. The main principle involves using the the second-order  scheme \eqref{rkfalpha} away from the steady states and reducing it to  the first-order scheme \eqref{first} near the
steady states. This modification results in a second-order scheme that is well-balanced. The measure of the closeness to a steady state
is determined through the following procedure.
We define a 
smooth function $\Theta$ such that
\begin{align}\label{eq:thetaindicator}
    \Theta (x)&:=\frac{x^2}{x^2+\Delta x^2},\, x\in\mathbb{R}.
    \end{align}Note that $\Theta (0)=0$ and $\Theta (x)\approx 1$ for $x\ne 0$  and sufficiently small $\Delta x.$
    For each $ (i,n)\in (\mathcal{M}\setminus\{0\})\times\mathcal{N}_T$, we define 
    \begin{align}
    \begin{split}
\Theta_i^n&:=\Theta (\mathcal{E}_i^n)\\\mathcal{E}_i^n&:=\mathcal{E}^n_{i-\half}+\mathcal{E}^n_{{ i+ \half}}\\\mathcal{E}^n_{{ i+ \half}}&:=\sqrt{ (G_{i+\half}^{n})^2+[H_{i+\half}^{n}]^2 }
    \end{split}
    \label{adapt1}
    \end{align} 
where
\begin{align*}
[H^n_{{ i+ \half}}]&:= H^n_{{ i+ \half}}-H^n_{i-\half}\\
H_{i+\half}^{n}&=H (\alpha_{i}^n, v_{i}^n, \alpha_{i+1}^n, v_{i+1}^n,B_{i},B_{i+1})\\
G_{i+\half}^{n}&=G (\alpha_{i}^n, v_{i}^n, \alpha_{i+1}^n, v_{i+1}^n)
\end{align*}
Here, $G$ and $H$ are given by (\ref{alphaflux}) and (\ref{hflux}), respectively.
In each Runge-Kutta stage, previously defined left and right states \eqref{slope1} are modified for the  linear function $z=\alpha,v$ and $B$  in $[x_{i-\half}, x_{{ i+ \half}}]$ as follows:
\begin{align} 
\label{eq:slopemodified1}
z_{i+\half,L}^{n}&=z_i^n+ \half \Theta_i^nDz^n_i,\quad 
z_{i-\half,R}^{n}=z_i^n- \half \Theta_i^nDz^n_i\\ 
\label{eq:slopemodified2}
z_{i+\half,L}^{n,*}&=z_i^{n,*}+ \half \Theta_i^nDz^{n,*}_i,\quad
z_{i-\half,R}^{n,*}=z_i^{n,*}- \half \Theta_i^nDz^{n,*}_i 
\end{align}
where $Dz^{n}_i,Dz^{n,*}_i$  are defined 
 as before in Section \ref{sec: An-approximation-Scheme}. Using these modified values \eqref{eq:slopemodified1} and \eqref{eq:slopemodified2}, we compute $u^{n,*}_\iph, u^{n,**}_\iph, v^{n,*}_{i}  $ and $v^{n,**}_{i}.$ The resulting adaptive second-order scheme is now expressed as
 \begin{equation}
  \label{eq:adapt1d}
    \begin{split}
u_\iph^{n+1}&=\half ({u_\iph^n+u_\iph^{n,**}}{}) \\
v_i^{n+1}&=\half ({v_i^n+v^{n,**}_{i}}{})
    \end{split}
\end{equation}

 \begin{theorem}
 The adaptive scheme  \eqref{eq:adapt1d} is well-balanced and  second-order accurate away from the steady state.
 \end{theorem}
 \begin{proof}
To prove that the scheme is well-balanced, it suffices to show that
\[ (u_\iph^1,v_i^1)= (u_\iph^0,v_i^0) = (\overline{u}_\iph,\overline{v}_i)\]
where   $ (\overline{u}_\iph,\overline{v}_i)$ are the discrete steady states, as given in the proof of Theorem \ref{NW}.
As the first-order scheme is well-balanced (see \cite{adimurthi2016a}), it follows that  $\Theta_i^{0} = 0.$  Consequently,  in the RK step-1 of the adaptive scheme \eqref{eq:adapt1d}, we obtain 
\begin{align*}
    u^{0,*}_\iph =u^0_\iph = \overline{u}_\iph, \quad
    v^{0,*}_i = v^0_i = \overline{v}_i \mbox{ and }
    \alpha^{0,*}_i = \alpha^0_i = \overline{\alpha}_i
\end{align*}
Similarly, in the RK step-2, we obtain, 
\begin{align*}
    u^{0,**}_\iph = \overline{u}_\iph \mbox{ and }
    v^{0,**}_i =\overline{v}_i
\end{align*}
Therefore, from the adaptive scheme \eqref{eq:adapt1d}, we get      ${u}_\iph^{1}=\overline{u}_\iph$
 and ${v}_i^{1}=\overline{v}_i.$ Further, we observe that $\mathcal{E}^0_i\ne 0$ away from the steady state, leading to the fact that $\Theta^{0}_i\approx 1$ for sufficiently small  $\Delta x$ (see \cite{desveaux2022fully,ghitti2020fully,michel2016well}). This shows that the adaptive scheme \eqref{eq:adapt1d} is second-order accurate away from the steady state. 
 \end{proof}

\section{Numerical schemes in two-dimensions}\label{sec: Scheme-for-2}
 In this section, we extend the scheme constructed in the previous sections to the two dimensional case. The system of equations in two dimensions is written as 

 \begin{align}
 \label{eq:2dsystema}
     u_t+ (\sqrt{u_x^2+u_y^2}-1)v &= 0, \quad \mbox{ in } \Omega\times (0,T] \\ 
     \label{eq:2dsystemb}
     v_t+ (- u_xv -B^x)_x+ (-u_yv -B^y)_y- (\sqrt{u_x^2+u_y^2}-1)v &= 0,\quad \mbox{ in } \Omega\times (0,T]\\ 
     u (x,y,0) = v (x,y,0) &= 0, \quad \mbox{ in } \Omega \nonumber
 \end{align}
 where
$B^x=B^x (x,y)=\int_0^x f_1 (\xi,y)d\xi, B^y (x,y) = \int_0^yf_2 (x,\xi)d\xi)$ with $f_1$ and $f_2$ are such that $f =f_1+f_2.$ The splitting of the source term $f$ will be explained later. The system \eqref{eq:2dsystema}-\eqref{eq:2dsystemb} is solved with boundary conditions as given in \eqref{HKw4}.  

We consider the domain  $\Omega= (0, 1)\times (0, 1)$. Define the space grid points along x-axis as $\displaystyle x_{i+{\half}}=i\Delta x, \Delta x>0, i\in\mathcal{M}$ 
and along y-axis as $\displaystyle y_{k+{\half}}=k\Delta y,\Delta y>0,k\in\mathcal{M}$ with $ (x_{\half}, y_{\half})= (0, 0)$ and $ (x_{M+{\half}}, y_{M+{\half}})= (1, 1)$. 
For simplicity, we set $\Delta x=\Delta y=h$  and
for $\Delta t>0$, define the time discretization points $t^n=n\Delta t, n\in\mathcal{N}_T$ with $\lambda=\Delta t/h.$ 
The numerical approximation of $u$ at the point $ (x_{i+{\half}}, y_{k+{\half}}), i, k\in\mathcal{M}$ at time $t^n$ is dented by $u_{i+{\half}, k+{\half}}^{n}.$  For $i, k\in\mathcal{M}\setminus\{0\}$,
 the solution $v$ in the cell $C_{i, k}$ at time $t^n$ is given by \[v_{i, k}^{n}=\displaystyle\frac{1}{h^2}\int_{C_{i, k}}v (x, y, t^n)dxdy.\]
 A pictorial illustration of the grid $C_{i,k}$ is given in Fig \ref{fig:grid}, where we suppress the time index $n$ for simplicity.
 \subsection{Second-order scheme}
 Similar to the one-dimensional case, we derive a second-order scheme in two dimensions by employing a MUSCL type spatial reconstruction in space and strong stability preserving  Runge-Kutta method in time. To begin with,
let $\alpha:= u_x$ and $\beta:= u_y,$ and  define their approximations as follows:

\begin{align*}
\alpha_{i, k+\half}^{n}&=\frac{u_{i+\half, k+\half}^{n}-u_{i-\half, k+\half}^{n}}{h},\;i\in \mathcal{M}\setminus\{0\},k\in\mathcal{M} \\\
\beta_{i+\half, k}^{n}&=\frac{u_{i+\half, k+\half}^{n}-u_{i+\half, k-\half}^{n}}{h},\; k\in \mathcal{M}\setminus\{0\},i\in\mathcal{M}
\label{xydir2}
\end{align*}
\begin{figure}[h!]
     \centering
         \centering
         \includegraphics[width=0.53\textwidth]{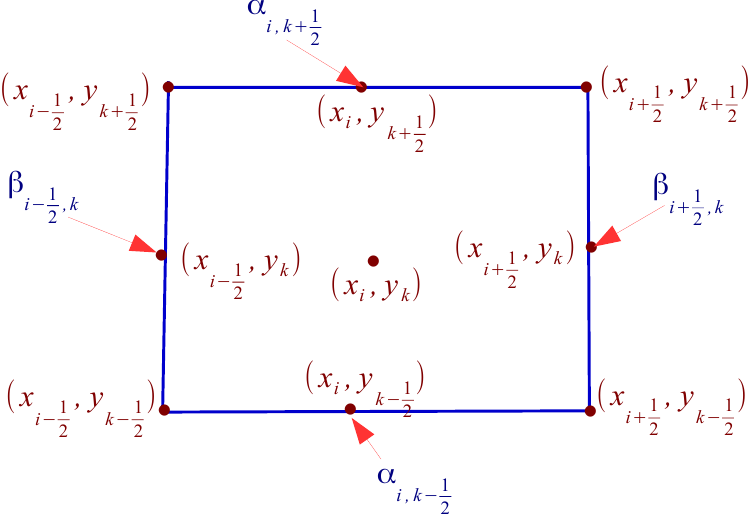}
\caption{An element $C_{i,k}$ of the two dimensional Cartesian grids }
\label{fig:grid}
\end{figure}
Next, we define the slopes corresponding to grid points   $i+\half$ and $k+\half$ as    
\begin{align*}
D\alpha^n_{i,k+\half}&={2 \theta \;\mbox{\mbox{minmod}}}\left (\alpha^n_{i+1,k+\half}-\alpha^n_{i,k+\half}, \half (\alpha^n_{i+1,k+\half}-\alpha^n_{i-1,k+\half}), \alpha^n_{i,k+\half}-\alpha^n_{i-1,k+\half}\right)
\end{align*}
for $i\in\mathcal{M}\setminus\{0\},\;k\in\mathcal{M}$
and 
\begin{align*}
D\beta_{i+\half,k}^n &=2 \theta \;\mbox{\mbox{minmod}} \left (\beta^n_{i+\half,k+1}-\beta^n_{i+\half,k}, \half (\beta^n_{i+\half,k+1}-\beta^n_{i+\half,k}), \beta^n_{i+\half,k}-\beta^n_{i+\half,k-1},\right)
\end{align*}
for $k\in\mathcal{M}\setminus\{0\}, i \in\mathcal{M}.$ Using this, we reconstruct
 piecewise linear functions with left and right end point values
\begin{eqnarray*}
\alpha^{n}_{{i+\half},L, k+\half} = \alpha^n_{i, k+\half} + \half {D\alpha^n_{i,k+\half}},&
\quad\alpha^{n}_{i-\half,R, k+\half} = \alpha^n_{i, k+\half}- \half {D\alpha_{i,k+\half}^n}{} \\
\beta^{n}_{i+\half, k+\half,L} = \beta^n_{i+\half, k}+ \half {D\beta_{i+\half,k}^n}{},&
\quad 
\beta^{n}_{i+\half, k-\half,R}=\beta^n_{i+\half, k}- \half {D\beta^n_{i+\half,k}}{}
\end{eqnarray*}
 {Now, we detail the linear  reconstruction of $v$ using the minmod limiter in $x$ and $y$ direction.} For the given values $v^n_{i,k}$  in the cell $ [x_{i-\half},x_{i+\half}] \times [y_{k-\half},y_{k+\half}],$ define the slopes as 
\begin{align*}
Dv_{i,k}^{n,x}&=2 \theta \;\mbox{\mbox{minmod}} \left (  v^n_{i+1,k}-v^n_{i,k},\half (v^n_{i+1,k}-v^n_{i-1,k}) v^n_{i,k}-v^n_{i-1,k}\right)\\
Dv_{i,k}^{n,y}&=2 \theta \;\mbox{\mbox{minmod}} \left (v^n_{i,k+1}-v^n_{i,k}, \half (v^n_{i,k+1}-v^n_{i,k-1}),v^n_{i,k}-v^n_{i,k-1}\right )
\end{align*}
and obtain the left and right values
\begin{equation}
\label{eq:vlimita}
\begin{aligned}
v^n_{i+\half,L, k} &= v^n_{i, k} +\half {Dv_{i,k}^{n,x}}{}, \quad 
v^n_{i-\half,R, k} = v^n_{i, k} - \half{Dv_{i,k}^{n,x}}{} \\
v^n_{i, k+\half,L} &= v^n_{i, k}+ \half {Dv_{i,k}^{n,y}}{},\quad  
v^n_{i, k-\half,R} = v^n_{i, k}- \half {Dv_{i,k}^{n,y}}{}
\end{aligned}
\end{equation}
\begin{eqnarray*}
v_{i+\half,L, k+\half}^{n}=\frac{v_{i+\half,L, k}^{n}+v_{i+\half,L, k+1}}{2},\;\;v_{i+\half,R, k+\half}^{n}=\frac{v_{i+\half,R, k}^{n}+v_{i+\half,R, k+1}^{n}}{2}\\
v_{i+\half, k+\half,L}^{n}=\frac{v_{i, k+\half,L}^{n}+v_{i+1, k+\half,L}^{n}}{2},\;\; v_{i+\half, k+\half,R}^{n}=\frac{v_{i, k+\half,R}^{n}+v_{i+1, k+\half,R}^{n}}{2}
\end{eqnarray*}
Note that in the $y$ direction, left and right values correspond to the values from the bottom and the top, respectively. Now, using this reconstructed values, we write the numerical fluxes at the $n$ th time step in the following lines. First, we write the approximation for $u_x$ and $u_y$ as
\begin{align*}
{ (u_x)^2 \approx \tilde  G_{i+\half, k+\half}^{n, x}=\Big (\max (|\max (\alpha_{i+\half,L, k+\half}^{n}, 0)|, |\min (\alpha_{i+\half,R, k+\half}^{n}, 0)|)\Big )^{2}}\\
 { (u_y)^2\approx \tilde G_{i+\half, k+\half}^{n, y}=\Big (\max (|\max (\beta_{i+\half,k+\half,L}^{n}, 0)|, |\min (\beta_{i+\half,k+\half,R,}^{n}, 0)|)\Big )^{2}}
\end{align*}
Define 
    { 
\begin{align*}
       {{W}_{i,k+\half}^{n,x}}&= (\sqrt{ (|\max (\alpha_{i+\half,L,k+\half}^{n} ,0)|)^2+\tilde G_{i+\half, k+\half}^{n, y}}-1)v^n_{i+\half,L,k+\half}\\
        {{W}_{i+1,k+\half}^{n,x}}&= (\sqrt{|\min (\alpha_{i+\half,R, k+\half}^{n}, 0)|^2+\tilde G_{\iph, k+\half}^{n, y}}-1)v^n_{{i+\half,R,k+\half}} \\
   {{W}_{i+\half,k}^{n,y}}&= (\sqrt{|\max (\beta^n_{{i+\half},{k+\half,}L}, 0)|^2+\tilde G_{i+\half, k+\half}^{n, x}}-1)v_{{i+\half,k+\half,}L}^{n} \\
    {{W}_{i+\half,k+1}^{n,y}}&= (\sqrt{|\min (\beta_{{i+\half, k+\half,}R}^{n}, 0)|^2+\tilde G_{i+\half, k+\half}^{n, x}}-1)v_{{i+\half, k+\half,}R}^{n}
  \end{align*}
Now, the numerical fluxes in the $x$ and $y-$ directions are given by {\[\label{W2}
       {{G}_{i+\half,k+\half}^{n,x}}=\max ( {W}_{i,k+\half}^{n,x}, {W}_{i+1,k+\half}^{n,x})\mbox{  and }
       {{G}_{i+\half,k+\half}^{n,y}}=\max ( {W}_{i+\half,k}^{n,y}, {W}_{i+\half,k+1}^{n,y})
      \]}
      respectively.
Finally, we write the term
 {
 \begin{equation}
 \label{eq:gflux2d}
 \begin{split}
 {G_{i+\half, k+\half}^{n}}
 &=\max ( {G}_{i+\half,k+\half}^{n,x}, {G}_{i+\half,k+\half}^{n,y})\\
 &=G ( \alpha^n_{i+\half,L,k+\half},\alpha^n_{i+\half,R,k+\half},\beta^n_{i+\half, k+\half,L}, \beta^n_{i+\half, k+\half,R},\\
 & \hspace{2cm} v^n_{i+\half,L,k+\half},v^n_{i+\half,R,k+\half}, v^n_{i+\half,k+\half,L}, v^n_{i+\half,k+\half,R} )
 \end{split}
 \end{equation}
 }
which approximates
\[
G (u, v)= (\sqrt{ (u_{x}^2+u_{y}^2}-1)v \] at $ (x_\iph, y_{k+\half})$ and $t^n$. 
For the approximation of $v,$ we consider the following flux  functions at the $n$ th stage:
\begin{equation}
\label{eq:hflux2d}
\begin{split}
H_{i+\half, k}^{n, x}&=H(\alpha_{i+ \half,L,k}^{n}, \alpha_{i+\half,R, k}^{n}, v_{i+\half,L,k}^{n}, v_{i+\half,R, k}^{n} , B_{i+\half,{L}, k}^{x}, B_{i+\half,{R}, k}^{x})\\
H_{i, k+\half}^{n, y}&=H (\beta_{i, k+\half,L}^{n}, \beta_{i, k+\half,R}^{n}, v_{i, k+\half, L}^{n}, v_{i, k+\half, R}^{n}, B_{i, k+\half,{L}}^{y}, B_{i, k+\half,{R}}^{y})
\end{split}
 \end{equation}

with
\begin{align*}
\alpha_{i+\half, L, k}^{n}& = \alpha_{i,k}^n +\half D\alpha^n_{i,k}\\
\alpha_{i-\half, R, k}^{n} &= \alpha_{i,k}^n -\half D\alpha^n_{i,k}\\
\beta_{i, k+\half,L}^{n} & = \beta_{i,k}^n + \half D \beta^n_{i,k} \\
\beta_{i, k-\half,R}^{n} & = \beta_{i,k}^n - \half D \beta^n_{i,k} 
\end{align*}
where the slopes are given by 
\begin{align*}
D\alpha^n_{i,k}& = 2\theta \mbox{ \mbox{minmod}}\left ( \alpha_{i+1,k}^n-\alpha_{i,k}^n, \half ( \alpha_{i+1,k}^n-\alpha_{i-1,k}^n), \alpha_{i,k}^n-\alpha_{i-1,k}^n \right)\\
D\beta_{i, k}^n & = 2 \theta  \mbox{ \mbox{minmod}}\left ( \beta_{i,k+1}^n-\beta_{i,k}^n,\half (\beta_{i,k+1}^n-\beta_{i,k-1}^n),
\beta_{i,k}^n-\beta_{i,k-1}^n\right)
\end{align*}
and $\alpha_{i,k}^n, \beta_{i,k}^n$ are computed as 
\begin{align*}
\alpha_{i,k}^n = \half ( \alpha_{i,k+\half}^n +\alpha_{i,k-\half}^n)\mbox{ and } \beta_{i,k}^n =
\half ( \beta_{i+\half, k} ^n+\beta_{i-\half, k}^n)
\end{align*}
 Further, the values $v_{i+\half,L,k}^n, v_{i+\half,R,k}^n, v_{i,k+\half,L}^n, v_{i,k+\half,R}^n$ are taken as in \eqref{eq:vlimita}. The terms $ B_{i+\half,{L}, k}^{x}, B_{i+\half,{R}, k}^{x}$ and  $ B_{i, k+\half,{L}}^{y}, B_{i, k+\half,{R}}^{y},$ are the approximations of $B^x$ and $B^y,$ respectively and are given in Section~\ref{sec:Bcompute}. 
 Using the  terms \eqref{eq:gflux2d} and fluxes in \eqref{eq:hflux2d}, we compute the approximations $u^{n,*}_{\iph,k+\half}$ and $v^{n,*}_{i,k}$ in the { RK Step-1} as 
 \begin{align*}\label{eq:vso_a}
\begin{split}
     u_{i+\half, k+\half}^{n,*}&=u_{i+\half, k+\half}^{n}-\Delta t G_{i+\half, k+\half}^{n},\\
     v_{i,k}^{n,*}&=v_{i,k}^n-\lambda\; (H_{i+\half, k}^{n, x}-H_{i-\half, k}^{n, x})-\lambda (H_{i, k+\half}^{n, y}-H_{i, k-\half}^{n, y}) \\&\quad+  \Delta t (\sqrt{ (\alpha^n_{i+\half,L,k})^2 + (\beta^n_{i,k+\half,L})^2-1 } )v^n_{i,k}
\end{split}
\end{align*}
We compute the values $u^{n,**}_{\iph,k+\half}$ and $v^{n,**}_{i,k}$ in the { RK Step-2} stage analogously to the one-dimensional case. Finally, the second-order scheme for  (\ref{HKw1}) in two-dimensions  is given by 
\begin{equation}
    \begin{split}
u_{\iph,k+\half}^{n+1}&=\half ({u_{\iph,k+\half}^n+u_{\iph+k+\half}^{n,**}}{})\label{rkfalpha2}  \\
v_{i,k}^{n+1}&=\half ({v_{i,k}^n+v^{n,**}_{i,k}}{}) 
    \end{split}
\end{equation}
\subsection{ Second-order adaptive scheme}
{Similar to the one-dimensional case, we now use the adaptation procedure on the second-order scheme \eqref{rkfalpha2} to ensure the well-balance property.} 
Now, define 
\begin{align*}
\mathcal{E}^n_{{ i+ \half}, { k+ \half}}&:=\sqrt{\Big|[H_{i+\half,k}^{n,x}]+[H_{i,k+\half}^{n,y}]\Big|^2+ (G_{i+\half,k+\half}^{n}})^2 
\end{align*}
where
\begin{align*}
[H^{n,x}_{{ i+ \half,k}}]&:= H^{n,x}_{{ i+ \half},k}-H^{n,x}_{i-\half,k}\\
[H^{n,y}_{{ i,k+\half}}]&:= H^{n,y}_{{ i,k+\half}}-H^{n,y}_{i,k-\half}
 \end{align*}
with the fluxes $G^n_{\iph,k+\half},$  $H_{i+\half, k}^{n, x}$ and $ H_{i+\half, k}^{n, y} $ are given by
\begin{align*}
    G^n_{i+\half,k+\half}& = G ( \alpha^n_{i,k+\half},\alpha^n_{i+1,k+\half},\beta^n_{i+\half, k}, \beta^n_{i+\half, k+1}, v^n_{i,k+\half},v^n_{i+1,k+\half}, v^n_{i+\half,k}, v^n_{i+\half,k+1} )\\
    H_{i+\half, k}^{n, x}&=H (\alpha_{i,k}^{n}, \alpha_{i+1, k}^{n}, v_{i,k}^{n}, v_{i+1, k}^{n} , B_{i, k}^{x}, B_{i+1, k}^{x})\\
H_{i, k+\half}^{n, y}&=H (\beta_{i, k}^{n}, \beta_{i, k+1}^{n}, v_{i, k}^{n}, v_{i, k+1}^{n}, B_{i, k}^{y}, B_{i, k+1}^{y})
\end{align*}
where $G$ and $H$  are computed as in \eqref{eq:gflux2d} and $\eqref{eq:hflux2d},$ respectively. 
Further, define the steady state indicator in each cell as 
\begin{eqnarray*}\Theta_{i,k}^n:= n\Theta_{i,k}^{n,x}+\Theta_{i,k}^{n,y}
\end{eqnarray*}
where 
\begin{eqnarray*}\Theta_{i,k}^{n,x}:=\Theta (\mathcal{E}_{i,k+\half}^n),\;\;\mathcal{E}_{i,k+\half}^n=\mathcal{E}^n_{i-\half,k+\half}+\mathcal{E}^n_{{ i+ \half,k+\half}}
\end{eqnarray*}
and 
\begin{eqnarray*}\Theta_{i,k}^{n,y}:=\Theta (\mathcal{E}_{i+\half,k}^n),\;\; \mathcal{E}_{i+\half,k}^n=       \mathcal{E}^n_{i+\half,k-\half}+\mathcal{E}^n_{{ i+ \half,k+\half}}
\end{eqnarray*} 
The function $\Theta$ is as defined in \eqref{eq:thetaindicator}. Finally, the adaptive second-order scheme is derived by modifying the  slopes using $\Theta^{n}_{i,k}$ following a similar approach as  in the one dimensional case.

\subsection{ Computation of the terms $B^x$ and $B^y$ in two dimensions}\label{sec:Bcompute} The computation of the terms $B^x$ and $B^y$ is  as given in \cite{adimurthi2016a,adimurthi2016}. For  completeness, we briefly review the construction in this section. Absorbing the source term with the convection terms is done by decomposing the  source function $f$ using the concept of \textit{Transport Rays} (blue lines in Fig. \ref{fig:transport}). This yields better results at the steady states, see \cite{adimurthi2016,adimurthi2016a}.
The terms $B^x$ and $B^y$ are defined as 
\begin{equation*}
B^{x} (x, y)=\intop_{0}^{x}f_{1} (\xi, y)d\xi\, \, \,\,\,\, \forall y \in[0, 1],\quad 
B^{y} (x, y)=\intop_{0}^{y}f_{2} (x, \xi)d\xi\, \, \,\,\, \forall x\in[0, 1]\label{eq: ByBy}
\end{equation*}
where $f_1$ and $f_2$ have to be chosen appropriately such that the source term $f=f_1+f_2$. The splitting is given as 
\begin{equation}\label{angler1}
 f_1=f\cos^2\varphi,\quad  f_2 =f\sin^2\varphi
\end{equation}
where $\varphi$ is 
 the angle which the transport ray $R_{ (x, y)}$ makes with the positive $x$-axis. In the case of open table problem, the transport rays are as given in Fig. \ref{fig:transport} (a) and  the functions $f_1$ and $f_2$ are determined as 
\begin{figure}
     \centering
     \begin{subfigure}[b]{0.36\textwidth}
         \centering
         \includegraphics[width=\textwidth]{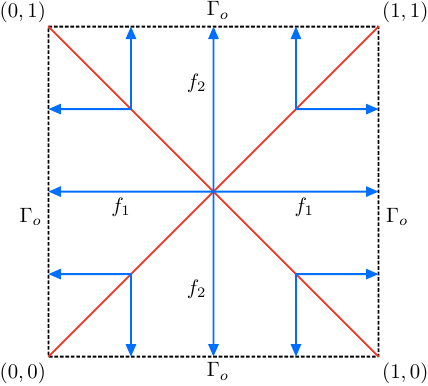}
         \caption{Open table problem}
     \end{subfigure}\hspace{0.6cm}
     \begin{subfigure}[b]{0.35\textwidth}
         \centering
         \includegraphics[width=\textwidth]{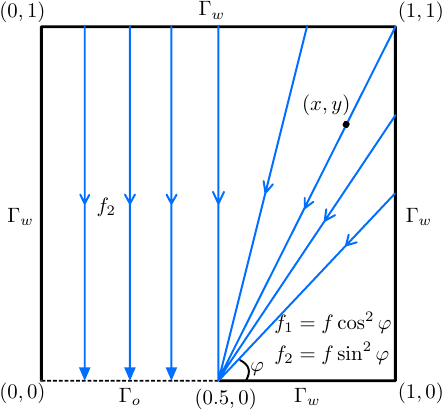}
         \caption{Partially open table problem}
     \end{subfigure}
     \caption{The domain $\Omega$ and the boundary $\Gamma$ with transport rays.}
     \label{fig:transport}
     \end{figure}

\begin{equation}\label{f1}
 f_{1} (x, y)=\left\{ \begin{array}{ccl}
f (x, y) & \, \, \mbox{if}\, \, & 	\{x\ge y, x+y\ge1\}\cup\{x\le y, x+y\le1\} \\\\ 	
0 & \, \, \mbox{otherwise}\, \, &
\end{array}\right.
\end{equation}
and 
\begin{equation}\label{f2}
f_{2} (x, y)=\left\{ \begin{array}{ccl}
f (x, y) & \, \, \mbox{if}\, \, & \{x\ge y, x+y\le1\}\cup\{x\le y, x+y\ge1\} \\\\ 
0 & \, \, \mbox{otherwise}\, \, &
\end{array}\right.
\end{equation}
For more details, refer to \cite{adimurthi2016a}.
In the case of  partially open table problem, illustrated in Fig. \ref{fig:transport} (b), the angle is computed as $\varphi= \tan^{-1} (\frac{y}{x-\half}),$ and the functions $f_1$ and $f_2$ are computed as in   \eqref{angler1}. The numerical approximation of $B^x$ and $B^y$ are given by 
\begin{equation}
B_{i,k}^{x}=\displaystyle \int_{0}^{x_{i}} f_{1} (\xi, y_{k})d\xi, \quad 
B_{i,k}^{y}=\displaystyle \int_{0}^{y_{k}} f_{2} (x_{i}, \xi)d\xi 
\end{equation}
and to approximate these integrals, we use composite trapezoidal rule. For more details, see \cite{adimurthi2016}.

\subsection{Boundary conditions}\label{sec:boundary}
Here, we consider two types of problems similar to the one dimensional case.

\noindent
\textbf{Case 1} (open table problem):
in this case, for computing the $u$ variable, we impose the boundary conditions by setting 

\begin{align*}
    u^n_{\half, k+\half} &=  u^n_{M+\half, k+\half} =0 \quad\\
    u^n_{i+\half,\half} &= u^n_{i+\half, M+\half}  =0, \quad i, k \in\mathcal{M}
\end{align*}
Moreover, to compute the fluxes at the interior vertices of the boundary cells, we need the slopes in the boundary cells. For simplicity, we set these slopes to zero. The boundary conditions for $v$ are prescribed through the  fluxes $H^x$ and $H^y$ at the boundary. As we set the slopes in all boundary cells to  zero, the corresponding fluxes are given by 
\begin{align*}
H_{\half, k}^{n, x}&=-\alpha_{1, k}^{n}v_{1, k}^{n}-B^{x}_{1, k}, \quad
H_{M+\half, k}^{n, x}=-\alpha_{M, k}^{n}v_{M, k}^{n}-B^{x}_{M, k} \\
H_{i, \half}^{n, y}&=-\beta_{i, 1}^{n}v_{i, 1}^{n}-B^{y}_{i, 1}, \quad
H_{i, M+\half}^{n, y}=-\beta_{i, M}^{n}v_{i, M}^{n}-B^{y}_{i, M},\quad i, k \in\mathcal{M}
\end{align*} 
 \noindent
 \textbf{Case 2} (partially open table problem):
we consider the partially open table problem in two-dimensions, where we choose the domain with wall boundaries as given in
Fig.\ref{fig:transport} (b) (see \cite{adimurthi2016}).
In this scenario, the evolution of the solution $u$ at the boundary vertices involves prescribing its values appropriately. For the solution $v,$ we impose boundary conditions through numerical fluxes. The conditions for $i, k \in\mathcal{M}$ are outlined in the following lines:
\begin{itemize}

\item Left vertical boundary
\begin{align*}
u^{n+1}_{\half,k+\half} &=
u^{n}_{\frac{3}{2},k+\half} +h  
\max (\alpha^n_{2,k+\half},0)\\
 H^{n,x}_{\half,j} &= \begin{cases}
-\alpha^n_{1,k}v^n_{1,k}-B^x (x_1,y_k), & \mbox{ if } \alpha^n_{1,k}\ge 0\\
-B^x (x_{0},y_k) & \mbox{otherwise}
\end{cases}
\end{align*} 
\item 
Right vertical boundary 
\begin{align*}
u^{n+1}_{M+\half,k+\half} &=  u^{n}_{M-\half,k+\half} + h
\max (\alpha^n_{M-1,k+\half},0)\\
H^{n,x}_{M+\half,k} &= \begin{cases}
-\alpha^n_{M,k}v^n_{M,k}-B^x (x_M,y_k), & \mbox{ if } \alpha^n_{M,k}\le 0\\
-B^x (x_{M+1},y_k) & {\mbox{otherwise}}
\end{cases}
\end{align*}
    \item Bottom horizontal boundary
\begin{align*}
u_{i+\half,\half}^{n+1}&=\left\{\begin{array}{ccl}
0&\, \, \mbox{if}\, \, & x_{i+\half}\leq 0.5 \\
u_{i+\half,\frac{3}{2}}^{n}+h\max (\beta_{i+\half,2}^n, 0)&\, \, \mbox{if}\, \, & x_{i+\half}>0.5
\end{array}\right.\\
 H^{n,y}_{i,\half} &= \begin{cases}
-\beta^n_{i,1}v^n_{i,1}-B^y (x_i,y_1), & \mbox{ if } x_{i+\half}\ge 0.5 \mbox{ and }  \beta_{i,1}\ge 0\\
-B^y (x_i,y_{-1}) & \mbox{ if } x_{i+\half}\ge 0.5  \mbox{ and }  \beta_{i,1}< 0\\
-\beta_{i,1}v^n_{i,1}-B^y (x_i,y_1), & \mbox{ if } x_{i+\half}> 0.5
\end{cases}
\end{align*}
\item Top horizontal boundary 
\begin{align*} 
u^{n+1}_{i+\half,M+\half} &= u^{n}_{i+\half,M-\half} +h  
\max (\beta^n_{i+\half, M-1},0)\\
 H^{n,y}_{i, M+\half} &= \begin{cases}
-\beta^n_{i,M}v^n_{i,M}-B^y (x_i,y_M), & \mbox{ if } \beta^n_{i,M}\le 0\\
-B^y (x_i,y_{M+1}) & {\mbox{otherwise}}
\end{cases}
\end{align*}
\end{itemize}
Finally, at the corner vertices of the rectangular domain, solution $u$ is computed by taking the average of solution computed through horizontal and vertical directions. 
Note that all the boundary conditions are imposed in each stage of the RK time stepping.
\section{Numerical experiments}\label{sec: Numerical-Experiment} We denote by FO, SO and SO-$\Theta$ first-order, second-order  and second-order adaptive schemes, respectively. Through out this section, we denote by  $||\cdot||_\infty$ and $||\cdot||_1,$ the supremum norm and the $L^1$ norm, respectively. 
\subsection{Examples in one-dimension (1D)}
In this section, we study  one-dimensional examples, specifically addressing the problem \eqref{eq:deposit}-\eqref{eq:erosion2} with the source function given by 
\begin{align}
f (x) = 0.5, \mbox{ for all } x \in [0,1]
\label{eq:1dfullsource}
\end{align}
in the computational domain $[0,1].$ Further, $u_s$ and $v_s$ denote the exact steady state solutions and are given by \begin{align}\label{B}\begin{split}
    u_s (x) &= \min (x, 1-x),\, x \in [0,1]\\
v_s (x) &= \begin{cases}
        \displaystyle\int_x^\frac{1}{2} f (\xi) d\xi& \mbox{ if }  x \in [0, \frac{1}{2}]  \\\\
        \displaystyle\int_\frac{1}{2}^x f (\xi) d\xi & \mbox{ if } x \in (\frac{1}{2},1] 
    \end{cases}\end{split}
\end{align}
where $f$ is the source function given by \eqref{eq:1dfullsource}.
We consider various test cases  in 1D to understand the significance of the proposed SO-$\Theta$ scheme. The boundary conditions in each case are employed as detailed in Section \ref{sec:foscheme} and \ref{sec:bcsocheme}. The initial conditions are set as $u = 0$ and $v=0$ in $[0,1].$ 
    
\textbf{Example 1 (convergence test case-1D)} In this test case, we verify the experimental order of convergence (E.O.C.) of the proposed  SO-$\Theta$ scheme  away from the steady state and compare it with that of FO and SO schemes. The source function is given by \eqref{eq:1dfullsource}.
We compute the solution at time $T = 1.3$ and  numerical solutions are evolved with a time step $\Delta t = 0.3 h $ ( i.e.,$\lambda = 0.3$). Since the exact solution of the problem is not available away from the steady state, the E.O.C. is computed using a reference solution which is obtained by the SO scheme with a fine mesh of size $\Delta x= 1/8000.$ We denote by $u_r$ and $ v_r$ the reference solutions corresponding to  $u$ and $v,$ respectively. The results are given in Table ~\ref{table:eoc1d}. 
\begin{table}[h!]
\begin{center}  {\small   \addtolength{\tabcolsep}{-2.8pt} 
\begin{tabular}{|l|l|l|l|l|} \hline
$\Delta x$ & $||u_{\Delta x} - u_r||_\infty$ & E.O.C. & $||v_{\Delta x}-v_r||_1$ & E.O.C.\\ 
\hline
\multicolumn{5}{|c|}{FO scheme}  \\\hline 
0.025   &        0.0122 &     -              &           0.0085 &      -               \\ \hline           
0.0125   &       0.0069  &    0.8122 &       0.0052 &     0.7100     \\ \hline    
0.00625   &      0.0040 &   0.7728&       0.0032  &     0.7022 \\ \hline        
0.003125  &      0.0023   &   0.7941 &       0.0019  &     0.7899  \\ \hline       
0.0015625 &      0.0013  &   0.8098 &       0.0011 &     0.7655\\ \hline  
\multicolumn{5}{|c|}{SO scheme}  \\\hline 
0.025  &         0.0058 &     -                     &    0.0049 &     -          \\ \hline                
0.0125  &        0.0028  &    1.03110     &   0.0024 &     1.0023  \\ \hline      
0.00625  &       0.0014  &   1.0422      &  0.0012  &     1.0237 \\ \hline       
0.003125  &      0.0007   &  1.0483       & 0.0006 &      1.0473  \\ \hline        
0.0015625  &     0.0003    & 1.0656      &0.0003 &    0.9933\\ \hline  
 \multicolumn{5}{|c|}{SO-$\Theta $ scheme}  \\\hline 
0.025  &         0.0062 &     -                   &      0.0050&      -    \\ \hline                    
0.0125  &        0.0030 &   1.0674  &       0.0025 &     1.0395\\ \hline      
0.00625  &       0.0014 &  1.0655 &       0.0012 &     1.0031  \\ \hline       
0.003125  &      0.0007  &  1.0633 &      0.0006 &     1.0423  \\ \hline     
0.0015625  &     0.0003 &   1.0749  &      0.0003 &     1.0053 \\ \hline
\end{tabular}
}
\end{center} 
\caption{Example~1 (1D): errors of 
  numerical solutions produced by FO, SO and SO-$\Theta$ schemes computed up to the final time $t = 1.3$ with a time step of $\Delta t = 0.3h.$}
\label{table:eoc1d}
\end{table}
From the definition of $\Theta$ in  \eqref{adapt1}, it becomes apparent that away from the steady state, both SO and SO-$\Theta$ schemes approach each other. This observation aligns with the findings from the numerical experiment, where both schemes exhibit nearly identical convergence rates.

Next, at the same time $T=1.3,$ we will compare the results  for a larger CFL, say $\lambda = 0.45,$  at  a value within the permissible limit as in Theorem \ref{stablity}. The solutions corresponding to FO, SO and SO-$\Theta$ schemes with ${\Delta x} = 1/100$  are compared against the same reference solution mentioned earlier. The results are given in Fig. \ref{fig:expbftime}.
\begin{figure}
     \centering
     \begin{subfigure}[b]{0.49\textwidth}
         \centering
         \includegraphics[width=\textwidth]{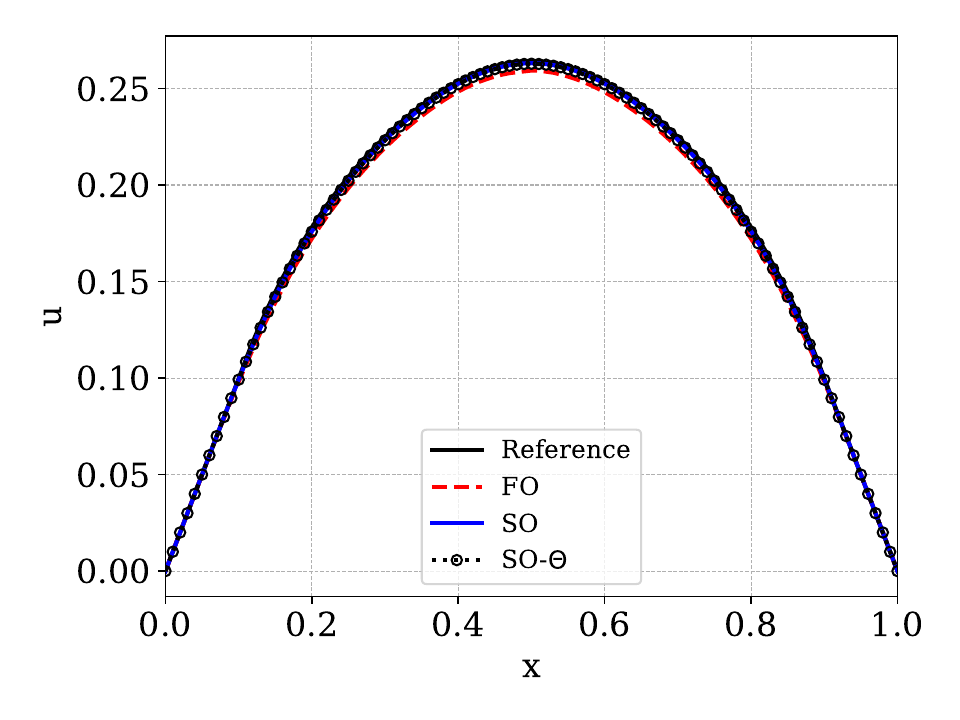}
         \caption{$u_{\Delta x}$}
     \end{subfigure}
     \begin{subfigure}[b]{0.49\textwidth}
         \centering
         \includegraphics[width=\textwidth]{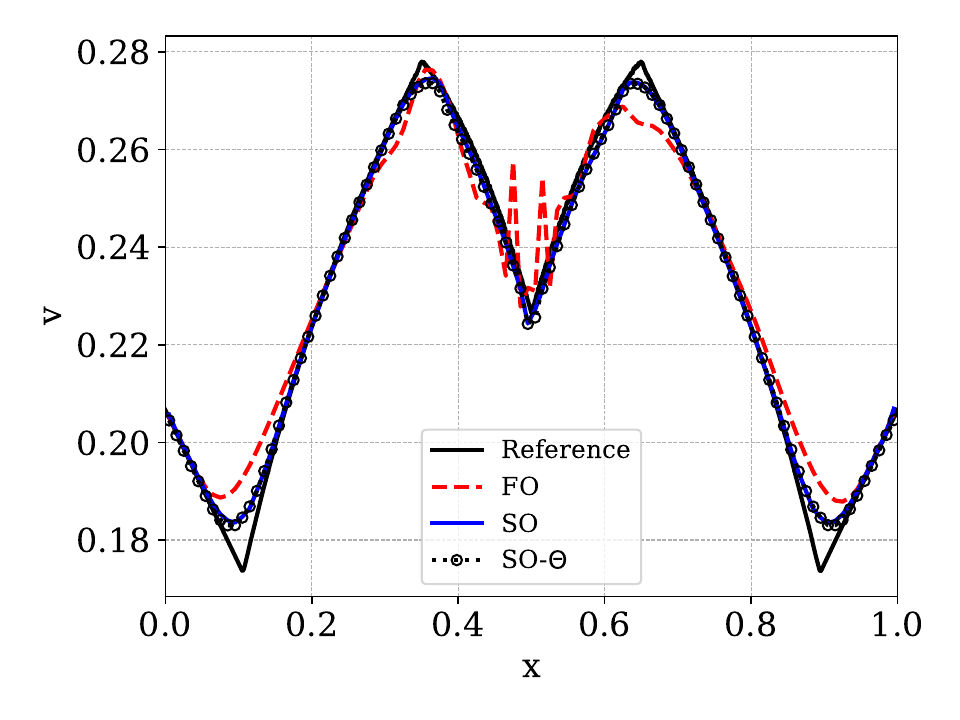}
         \caption{$v_{\Delta x}$}
     \end{subfigure}
\caption{Example 1 (1D): numerical solutions computed at time $t = 1.3.$  with $\Delta x = 1/100,\, \Delta t = 0.45 \Delta x.$ }
\label{fig:expbftime}
\end{figure} It is observed that, even though the FO scheme is stable with this $ \lambda= 0.45,$ it produces oscillations. This is in contrast to the SO and SO-$\Theta$ schemes. This indicates the robustness of the proposed SO-$\Theta$ schemes with larger values of $\lambda$ away from the steady state.



\textbf{Example 2 (well balance test case-1D)}\begin{table}[h!]
\begin{center}  {\small   \addtolength{\tabcolsep}{-2.8pt} 
\begin{tabular}{|l|l|l|} \hline
$\Delta x$ & $||u_{\Delta x}-\overline{u}||_\infty$ & $||v_{\Delta x}-\overline{v}||_1$ \\ 
\hline
\multicolumn{3}{|c|}{FO scheme}  \\\hline 
0.02  &          6.9388e-18   &                       7.0546e-18 \\ \hline                          
0.01  &          1.3878e-17 &                         8.7499e-17   \\ \hline         
0.005 &          1.3878e-17 &                          9.4632e-17   \\ \hline   
0.0025  &        6.9389e-18  &                          1.2179e-15 \\ \hline      
\multicolumn{3}{|c|}{SO scheme}  \\\hline 
   0.02    &        6.9389e-18      &                     0.0001       \\ \hline                 
0.01     &       6.9389e-18     &                       3.4875e-05\\ \hline           
0.005    &       1.3878e-17    &                     8.7188e-06    \\ \hline         
0.0025   &       6.9389e-18     &                       2.1797e-06  \\ \hline   
 \multicolumn{3}{|c|}{SO-$\Theta $ scheme}  \\\hline 
  0.02     &       6.9389e-18   &                      5.3950e-17    \\ \hline                         
0.01   &         6.9389e-18   &                         6.7966e-17   \\ \hline           
0.005  &         1.3878e-17  &                       7.2185e-17    \\ \hline        
0.0025  &        6.9389e-18   &                         6.6996e-16 \\ \hline    
\end{tabular}
}
\end{center} 
\caption{Example~2 (1D):  well-balance test for FO, SO and SO-$\Theta$ schemes. Numerical errors produced by FO, SO and SO-$\Theta$ schemes in a single time step, i.e., at the final time $T = \Delta t$ with $\Delta t = 0.45\Delta x.$}
\label{table:wellbalance1D}
\end{table}
The  purpose of this example is to illustrate the well-balance property of the  SO-$\Theta$ scheme for the problem \eqref{eq:deposit}-\eqref{eq:erosion2}. The simulations are carried out for a single time step, i.e., $T=\Delta t$, with $\lambda = 0.45,$ where the initial condition is set as the exact steady state solution \eqref{B}.

We compute the errors $||u_{\Delta x}-u_s||_\infty$ and $||v_{\Delta x}-v_s||_1$ for various values of $\Delta x$ and present the results in Table \ref{table:wellbalance1D}. It is observed that the SO-$\Theta$ scheme achieves the well-balance property, consistent with the FO scheme. In contrast, the SO scheme fails to capture this well-balance property.
This observation agrees with the result in  Lemma~\ref{NW}.

\textbf{Example 3 (steady state solution test case-1D)} In this example, we compute numerical solutions with the initial condition $u_0 = 0$ and $v_0 = 0,$ and evolve them up to the steady state level using the  FO, SO and SO-$\Theta$ schemes. \begin{figure}[h!]
     \centering
     \begin{subfigure}[b]{0.49\textwidth}
         \centering
         \includegraphics[width=\textwidth]{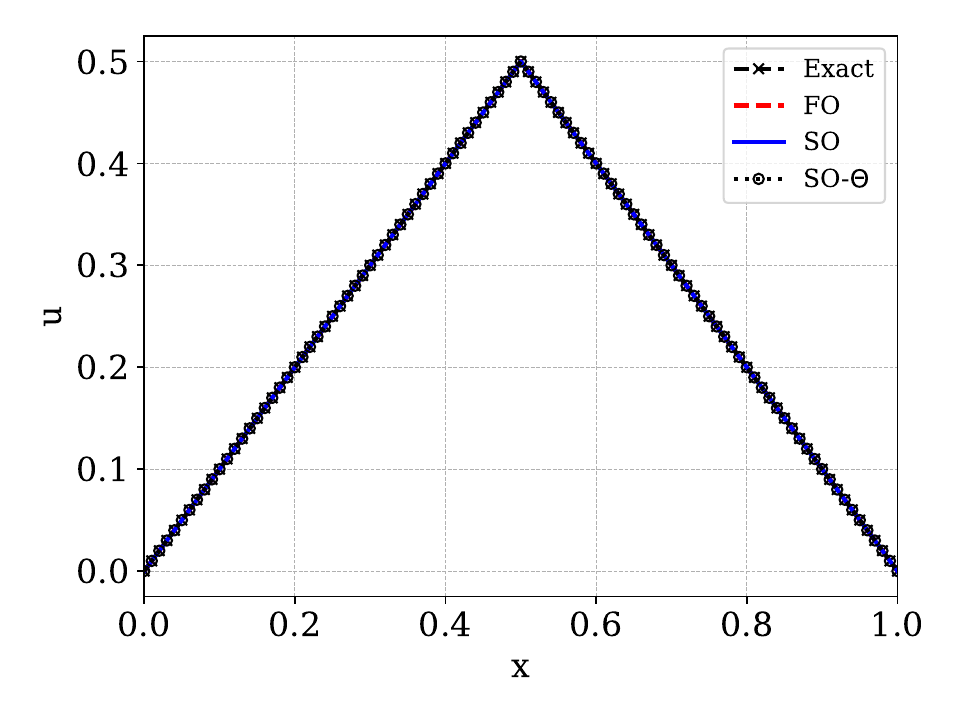}
         \caption{$u_{\Delta x}$}
         \label{u1}
     \end{subfigure}
     \begin{subfigure}[b]{0.49\textwidth}
         \centering
         \includegraphics[width=\textwidth]{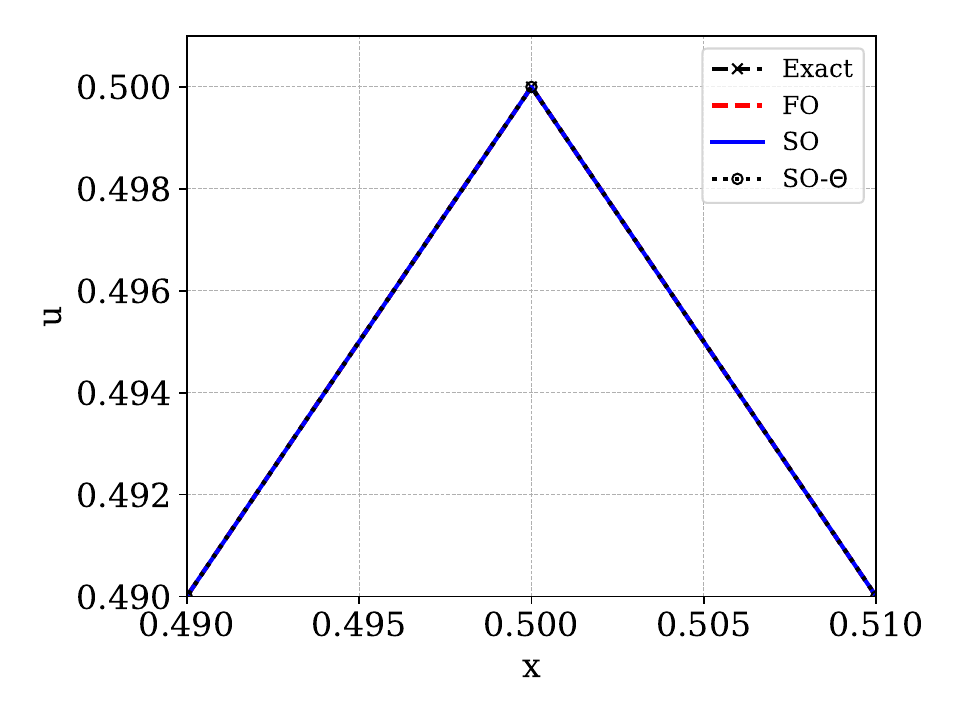}
         \caption{$u_{\Delta x},$ zoomed view}
           \label{u2}
     \end{subfigure}\\
     \begin{subfigure}[b]{0.49\textwidth}
         \centering
         \includegraphics[width=\textwidth]{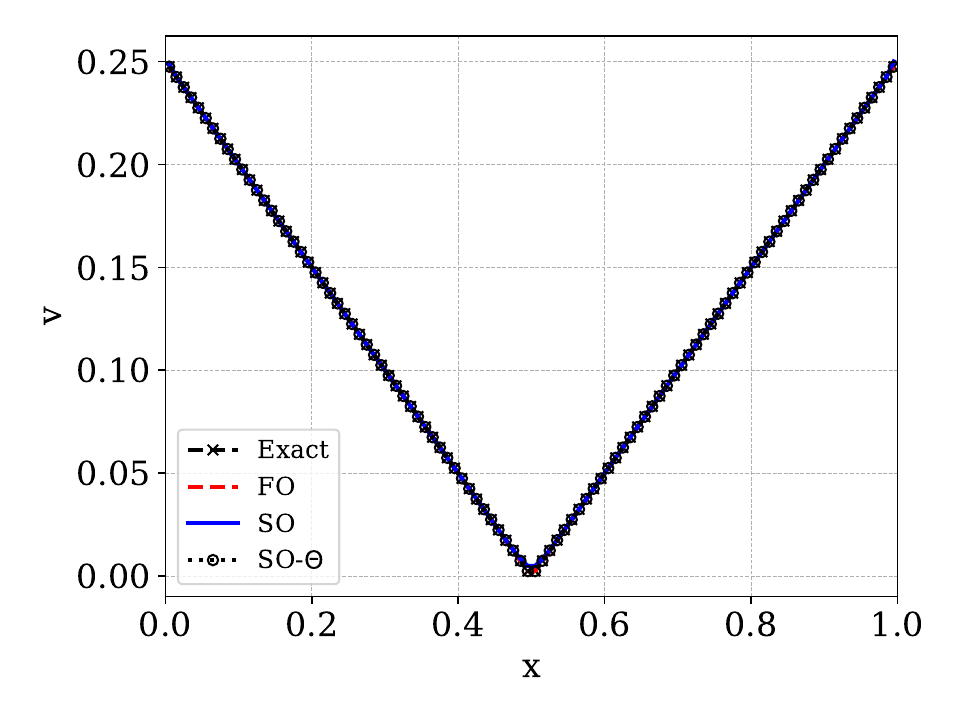}
         \caption{$v_{\Delta x}$}
         \label{fig:five over x}
     \end{subfigure}
     \begin{subfigure}[b]{0.49\textwidth}
         \centering
         \includegraphics[width=\textwidth]{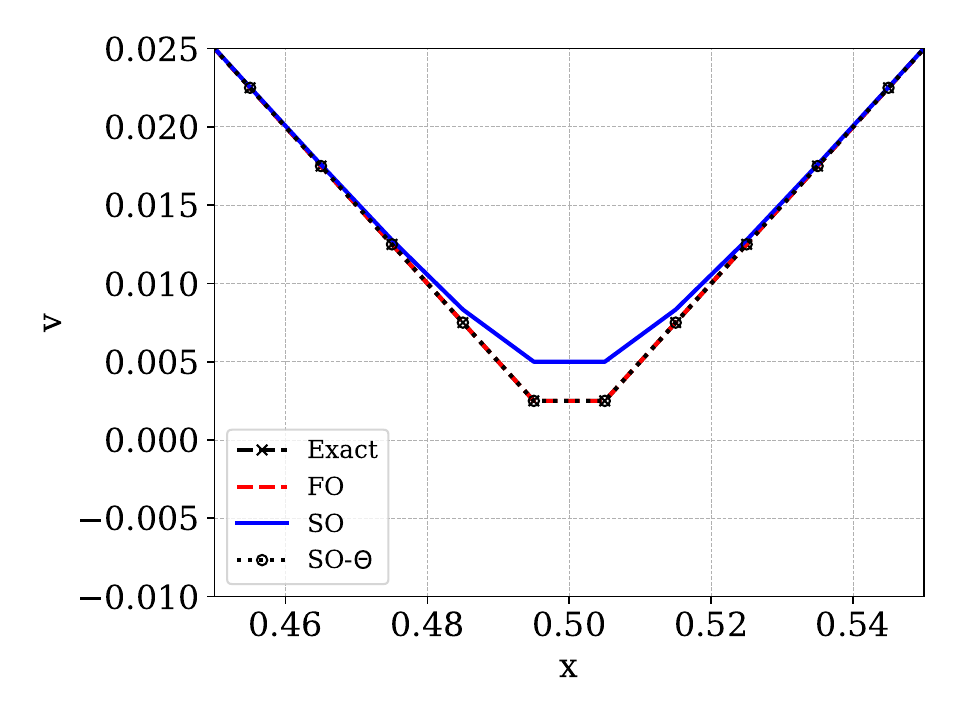}
         \caption{$v_{\Delta x},$ zoomed view}
     \end{subfigure}
\caption{Example 3 (1D): numerical solutions near the steady state, computed at time $T=450$ with $ \Delta x = 1/100, \Delta t = 0.45 \Delta x.$ (b) and (d) are enlarged view of (a) and (c), respectively. }
\label{fig:1dsteady}
\end{figure}For all the three schemes, the numerical solutions are computed at $T = 450$ with $\lambda =0.45$  and $\Delta x= 1/100.$  The numerical results are then compared with the exact steady state solution \eqref{B}, which are plotted through interpolation on the same mesh $\Delta x=1/100$ and are given in Fig.~\ref{fig:1dsteady}. The obtained results reveal that the FO and the adapted SO-$\Theta$ schemes align remarkably well with the steady state solution.   
On the other hand, the SO scheme, successfully reaches the  steady state for $u$ as shown in Fig.~\ref{fig:1dsteady} (a) and (b), but it exhibits poor performance for $v$, as shown in Fig.~\ref{fig:1dsteady} (c) and (d).
This emphasizes the significance of the adaptation strategy when employing a second-order scheme to accurately capture the steady state solution. 

\textbf{Example 4 (error versus number of iterations plots-1D)} 
In this example, we demostrate the efficiency of the SO-$\Theta$ scheme in reaching the steady state solution by showing that it takes less number of iterations compared to the FO scheme. To assess this, we plot the errors  $||u_{\Delta x}-u_s||_\infty$ and $||v_{\Delta x}-v_s||_1$ at each iteration against the number of iterations for solving the problem  \eqref{eq:deposit}-\eqref{eq:erosion2}. Here also we keep the same $\lambda = 0.45 $ with a mesh size of $\Delta x=1/100.$ We compare the outcomes from the FO, SO and SO-$\Theta$ schemes. The results are depicted in Fig. \ref{fig:errorvsiter1D}.
It is crucial to emphasize that the SO-$\Theta$ scheme reverts to first-order at the steady state. When comparing the solution $u,$ the SO scheme shows better performance than both the FO and SO-$\Theta$ schemes, but not in $v.$ The key challenge here is in achieving the steady state for $v$. In this context, the results reveal the efficiency of the SO-$\Theta$ scheme, which converges to the steady state with lesser iterations than the FO scheme, Fig. \ref{fig:errorvsiter1D} (a) and (b). Significantly, the non-adaptive SO scheme encounters difficulties in reaching the steady state for $v$. In conclusion, the SO-$\Theta$ scheme performs better compared to the FO scheme.
\begin{figure}
     \centering
     \begin{subfigure}[b]{0.49\textwidth}
         \centering
         \includegraphics[width=\textwidth]{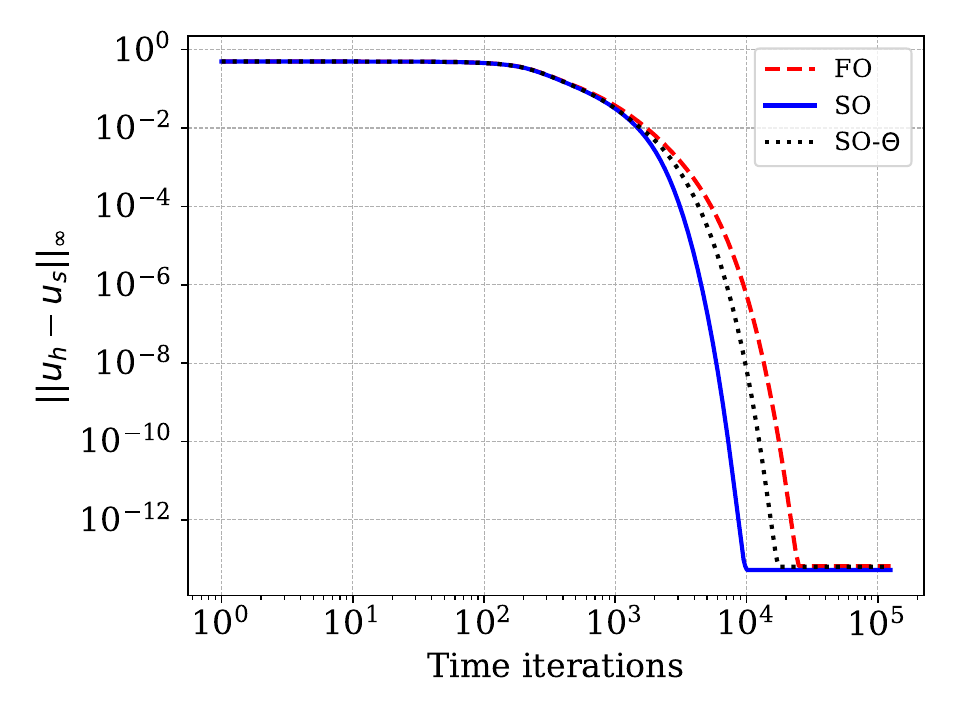}
         \caption{$u_{\Delta x}$}
     \end{subfigure}
     \begin{subfigure}[b]{0.49\textwidth}
         \centering
         \includegraphics[width=\textwidth]{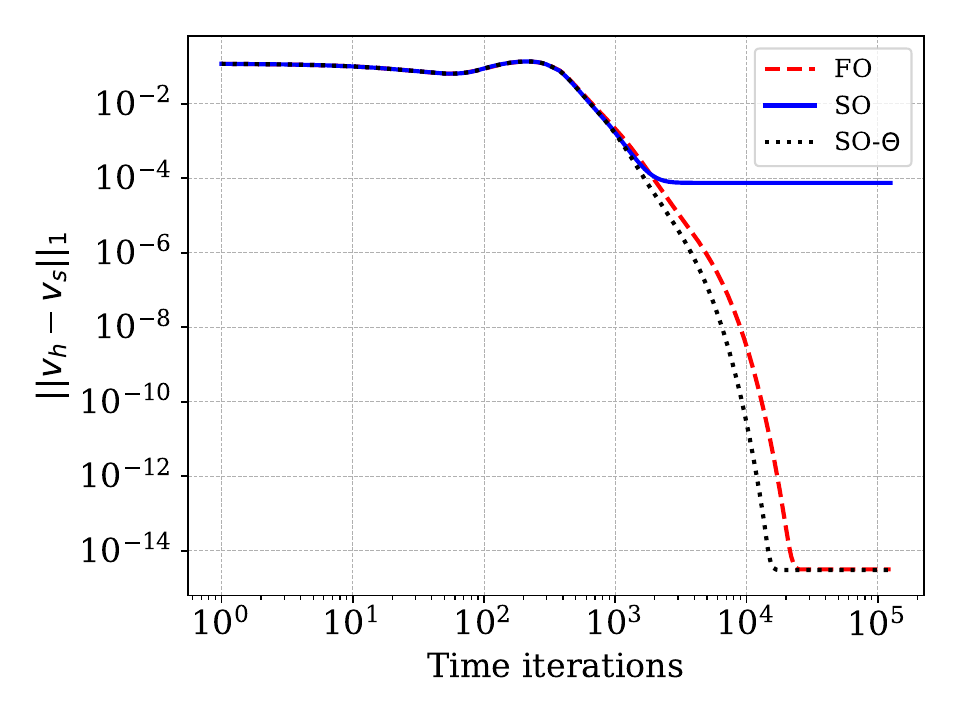}
         \caption{$v_{\Delta x}$}
     \end{subfigure}
\caption{Example 4 (1D): numerical errors versus number of time iterations plots. Numerical solutions are computed  till it reach near the state with $\Delta x  = 1/100, \Delta t = 0.45h.$ In each iteration, the errors produced by FO, SO and SO-$\Theta$ schemes are  compared.}
\label{fig:errorvsiter1D}
\end{figure}

\subsection{Examples in two-dimensions (2D)} 
We perform numerous numerical experiments in a two-dimensional setting, employing a computational domain $[0,1]\times [0,1].$ As mentioned previously, we discretize this domain into uniform Cartesian grids, denoted as $h = \Delta x = \Delta y$. Our focus here centers on investigating SO-$\Theta$ scheme for three specific types of problems in the two-dimensional context: those involving a complete source, scenarios with a discontinuous source, and problems related to partially open boundaries, detailed in \cite{adimurthi2016,adimurthi2016a}. The boundary conditions for each examples in this test cases are described in Section \ref{sec:boundary} and the initial conditions are set as $u =0 $ and $v = 0$ in $ (0,1)\times (0,1).$ In all the examples below we take the source function as $f=0.5$ on
$[0,1]\times [0,1].$\\

\textbf{Example 5 (convergence test case-2D)} 
Here, for the open table problem, we analyze the E.O.C. of  SO-$\Theta$  scheme with respect to the steady state solution and compare it with the FO and SO schemes. \begin{table}[h!]
\begin{center}  {\small   \addtolength{\tabcolsep}{-2.8pt} 
\begin{tabular}{|l|l|l|l|l|} \hline
h & $||u_{\Delta x} - u_r||_\infty$ & E.O.C. & $||v_{\Delta x}-v_r||_1$ & E.O.C.\\ 
\hline
\multicolumn{5}{|c|}{FO scheme}  \\\hline 
0.05    &        0.0662  &     -                  &       0.0031   &    -              \\ \hline        
0.025   &        0.0334  &     0.9841 &       0.0009  &    1.7513   \\ \hline   
0.0125   &       0.0167 &    1.0037 &       0.0002  &    1.8799  \\ \hline    
0.00625  &       0.0084  &    0.9876 &      6.4819e-05  &    1.9368 \\ \hline
\multicolumn{5}{|c|}{SO scheme}  \\\hline 
0.05  &          0.0599 &       -                &         0.0038&    -            \\ \hline            
0.025 &          0.0299 &     0.9992 &       0.0017 &     1.1327      \\ \hline 
0.0125 &         0.0150 &     0.9994&       0.0008 &    1.0709 \\ \hline     
0.00625 &        0.0075 &    0.9986 &       0.0004  &    1.0385  \\ \hline
 \multicolumn{5}{|c|}{SO-$\Theta $ scheme}  \\\hline 
0.05    &        0.0600&       -                  &       0.0018  &     -                        \\ \hline
0.025    &       0.0300 &    0.9995 &      0.0005 &    1.9026    \\ \hline   
0.0125    &      0.0150 &     0.9991 &       0.0001&    1.9178 \\ \hline      
0.00625   &      0.0075 &     0.9984 &       3.3913e-05  &     1.9059\\ \hline
\end{tabular}
}
\end{center} 
\caption{Example~5 (2D): numerical errors produced by FO, SO and SO-$\Theta$ schemes, computed up to time $T = 26$ with the time step  $\Delta t = 0.35h.$}
\label{table:eoc2d}
\end{table}The numerical solutions are computed near the steady state, by running the simulation till the time  $T = 26,$ as studied in \cite{adimurthi2016a}. Considering the value  $\lambda= 0.7 $ used in \cite{adimurthi2016a} for the FO scheme, we adopt a natural choice for the SO and SO-$\Theta$ schemes with $\lambda = 0.35.$ Here, for a fair comparison, we use the same $ \lambda= 0.35$ for all the three schemes FO, SO and SO-$\Theta$. The E.O.C. is computed as:
\begin{align*}
    \varphi (h):= \log\Bigl (e (h)/ e (h/2) \Bigr)/\log 2,
\end{align*}
where $e (h)$ denotes $||u_{\Delta x}-u_s||_\infty$ and $||v_{\Delta x}-v_s||_1$ in respective cases. 
The outcomes presented in Table \ref{table:eoc2d} reveal that the SO-$\Theta$ scheme exhibits marginally superior convergence when compared with the FO scheme despite the fact that the SO-$\Theta$ scheme reduces to the FO scheme in the vicinity of the steady state. On the other hand, it is evident that the SO scheme displays a diminished rate of convergence in comparison to both the FO and SO-$\Theta$ schemes for $v.$ This emphasizes the robustness and significance of the proposed SO-$\Theta$ scheme.

\begin{figure}
     \centering
     \begin{subfigure}[b]{0.45\textwidth}
         \centering
         \includegraphics[width=\textwidth]{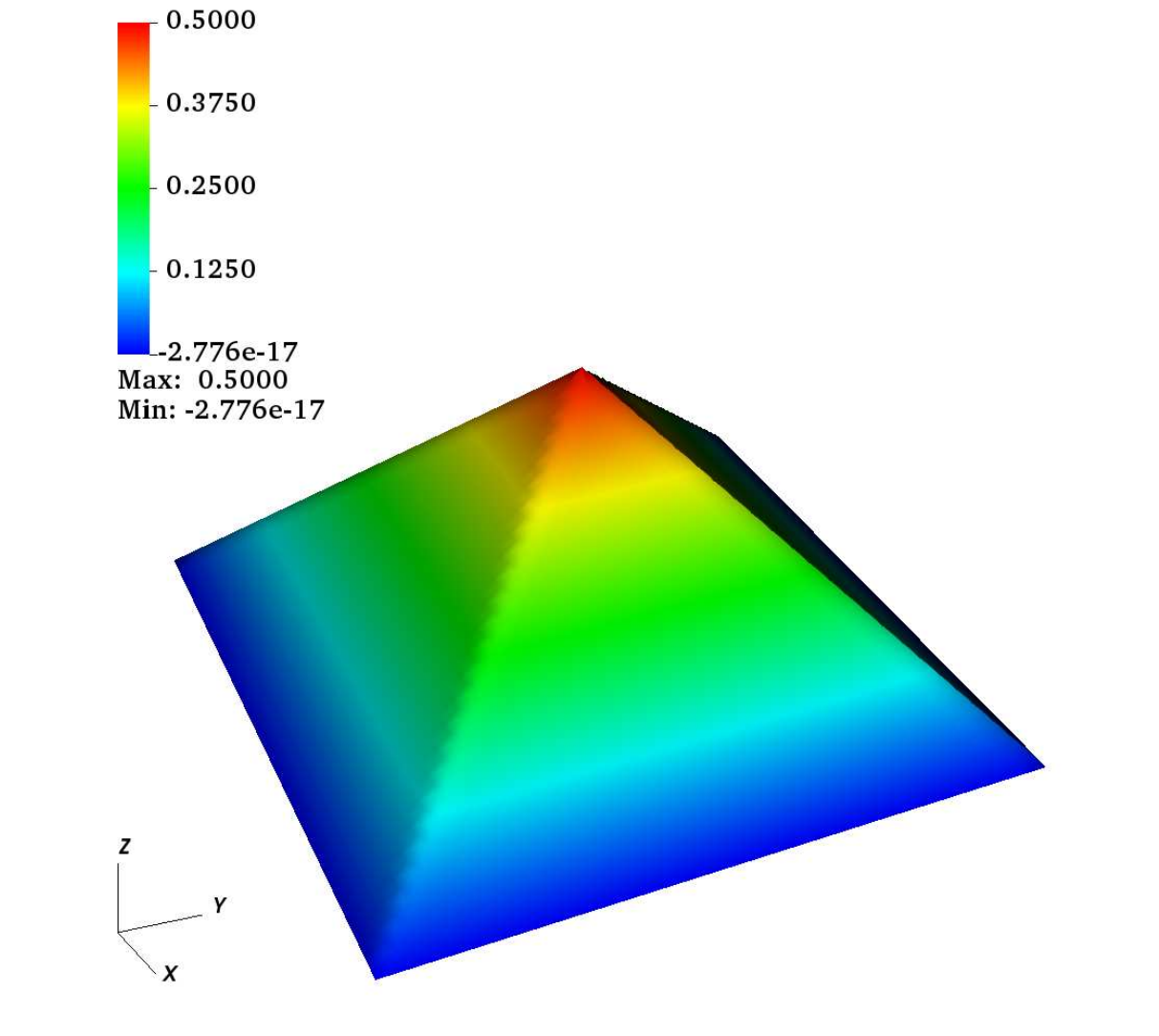}
         \caption{$u:$ exact steady state solution} 
     \end{subfigure}
     \begin{subfigure}[b]{0.45\textwidth}
         \centering
         \includegraphics[width=\textwidth]{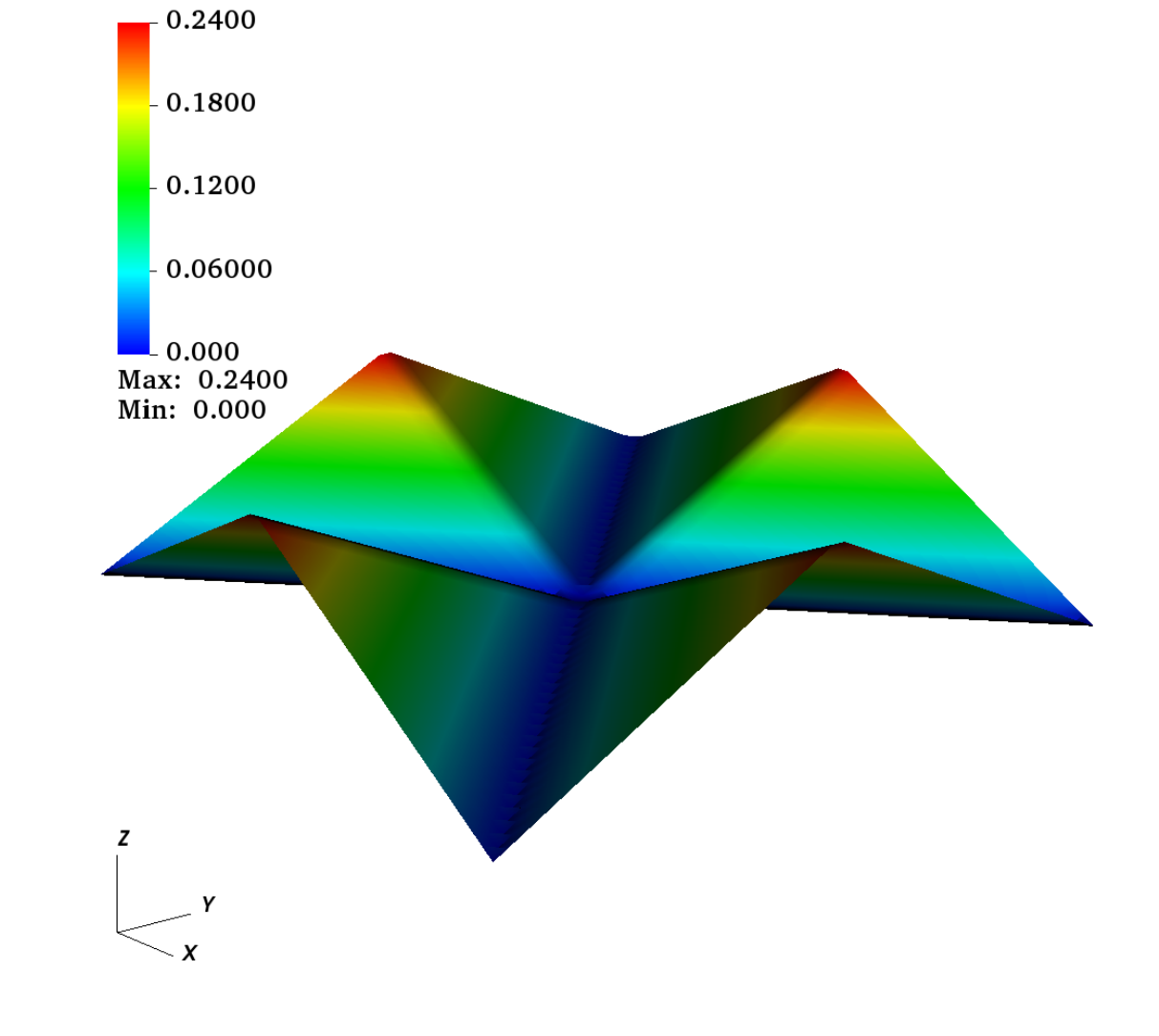}
         \caption{$v:$ exact steady state solution}
       
     \end{subfigure}\\
     \begin{subfigure}[b]{0.45\textwidth}
         \centering
         \includegraphics[width=\textwidth]{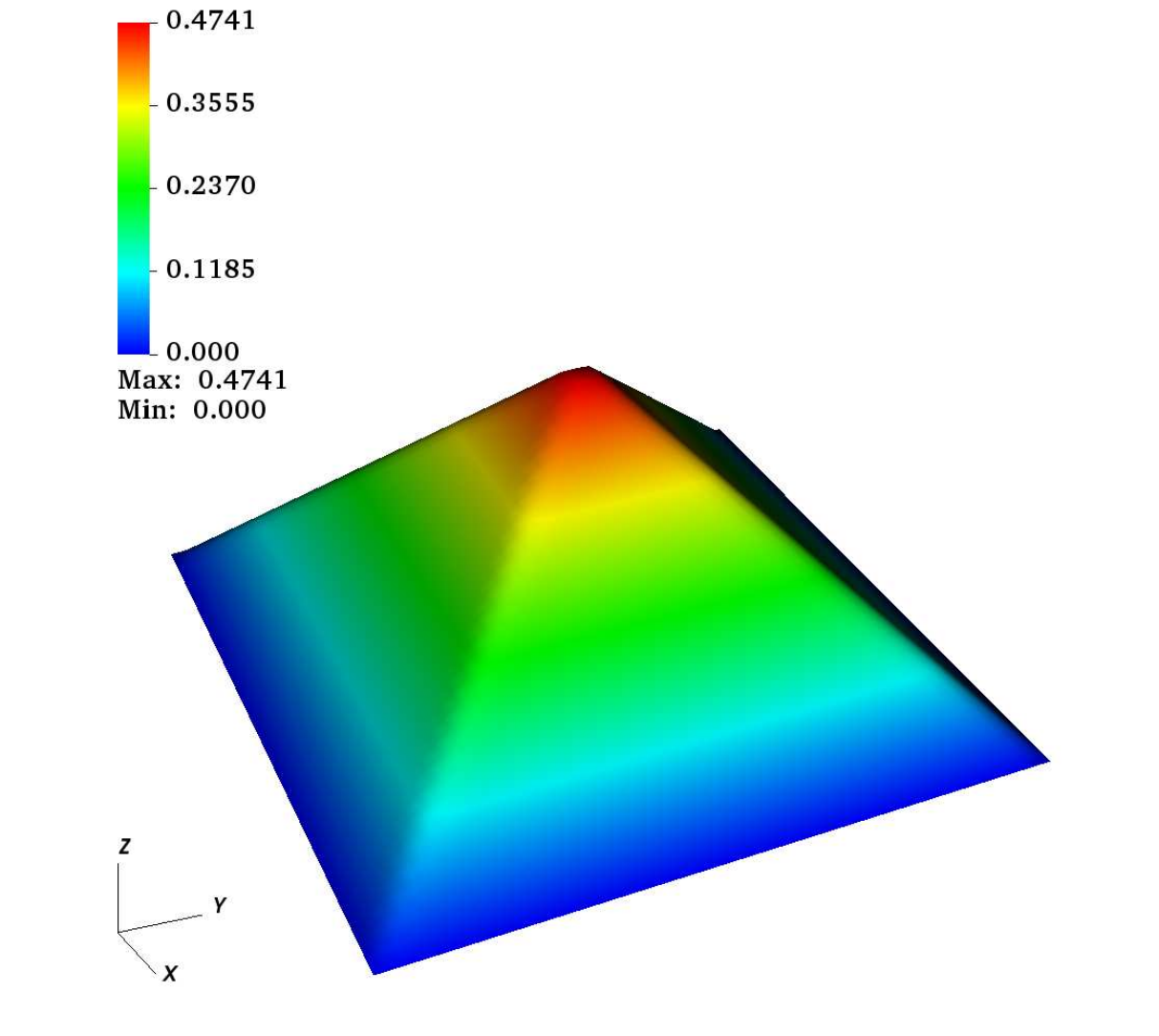}
         \caption{$u_{h}:$ using FO scheme}
     \end{subfigure}
     \begin{subfigure}[b]{0.45\textwidth}
         \centering
         \includegraphics[width=\textwidth]{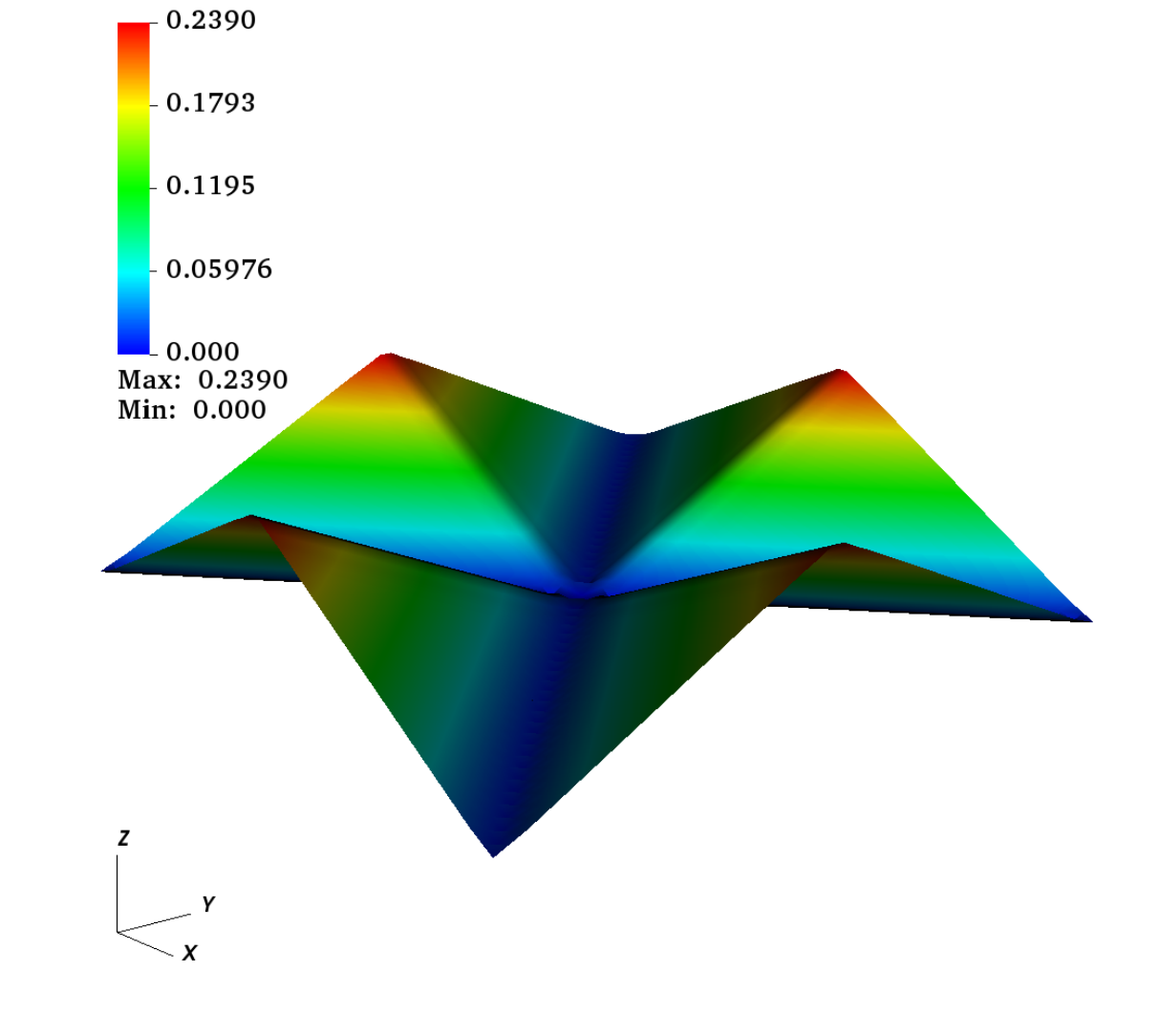}
         \caption{$v_{h}:$ using FO scheme}
     \end{subfigure}
       \begin{subfigure}[b]{0.45\textwidth}
         \centering
         \includegraphics[width=\textwidth]{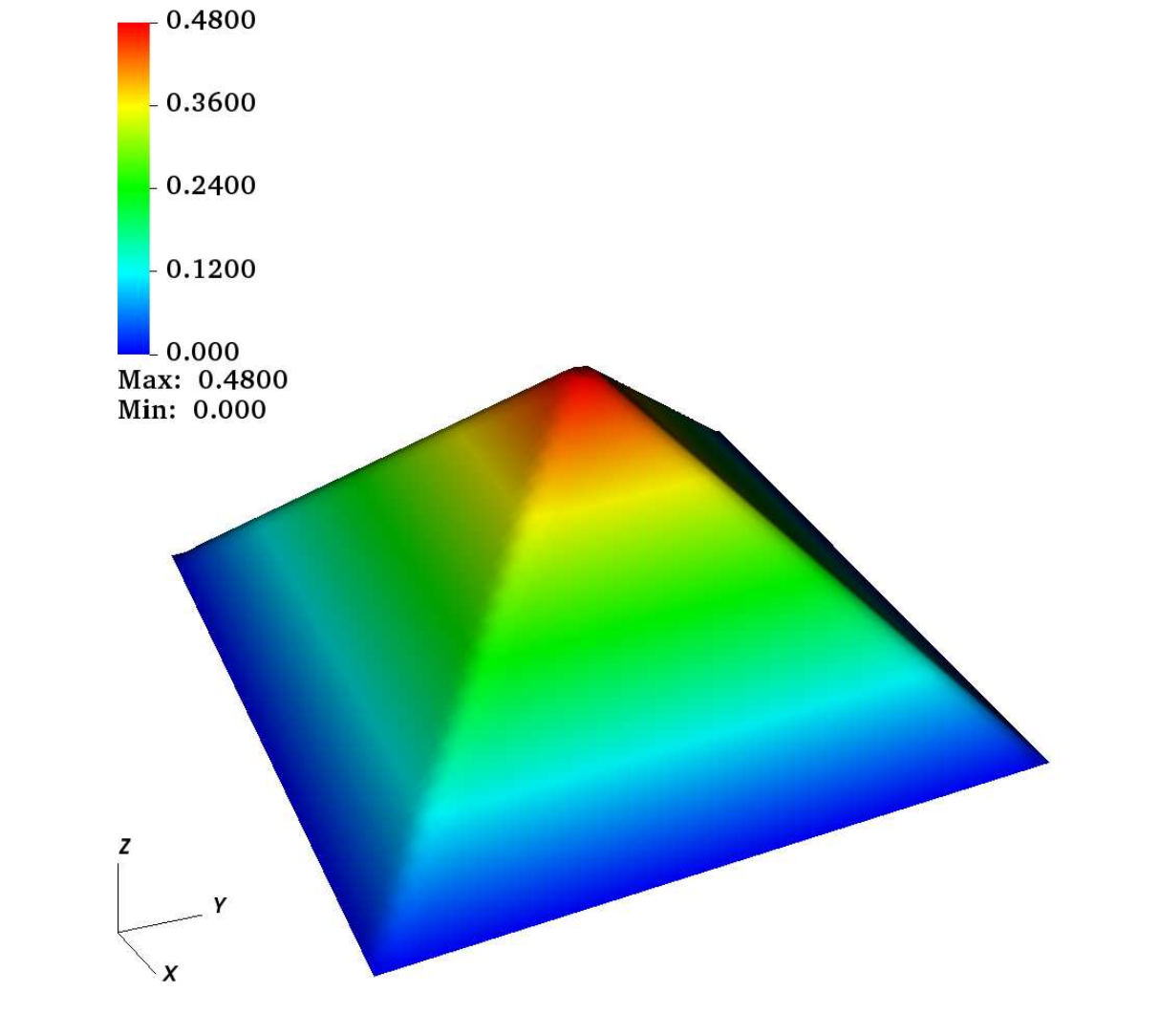}
         \caption{$u_{h}:$ using SO-$\Theta$ scheme}
     \end{subfigure}
     \begin{subfigure}[b]{0.45\textwidth}
         \centering
         \includegraphics[width=\textwidth]{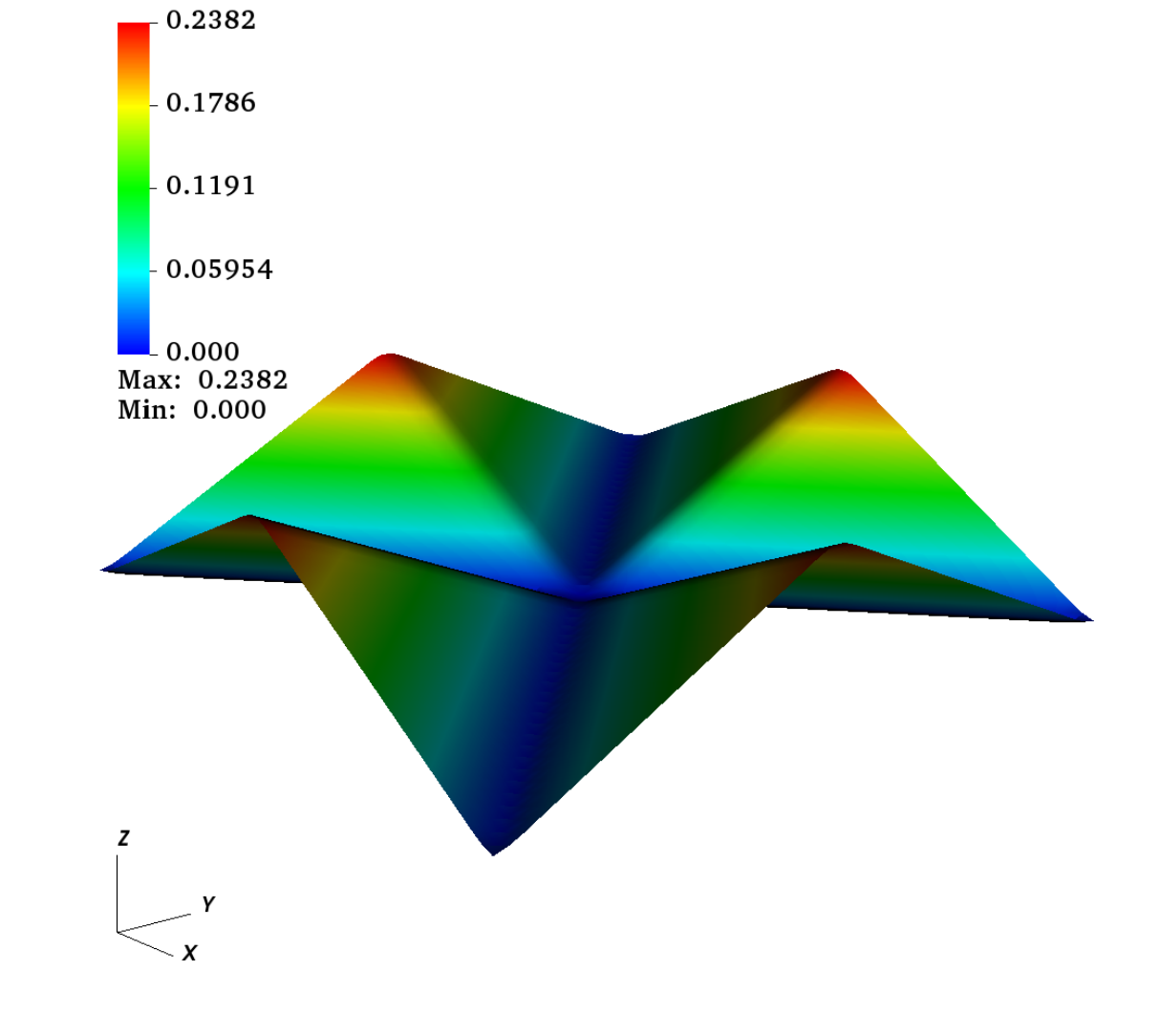}
         \caption{$v_{h}:$ using SO-$\Theta$ scheme}
     \end{subfigure}
\caption{Example~5 (2D): numerical solutions  near the steady state computed  with $h =  1/50 $ and $ \Delta t = 0.35 h $ at final time $T = 196.$}
\label{fig:2dopensteady}
\end{figure}

Next, to visualize the solution near the  steady state, we run the simulation till the  time $T=196$, with the same   $\lambda=0.35$  and $h = 1/50.$  We plot the approximate solutions obtained using the FO and SO-$\Theta$ schemes with the exact steady solutions in Fig. \ref{fig:2dopensteady}, where exact steady state solutions are plotted with linear interpolation on a finer mesh of size $\Delta x=0.5.$ 
Since the approximate solutions are very close to the steady state solutions, visual distinctions are difficult to discern from the given plots.

\textbf{Example 6 (error versus number of iterations-2D)}  We examine the error versus number of iterations plot obtained using the FO, SO, and SO-$\Theta$ schemes for the open table problem. To conduct this comparison,as for  Example 5, we take $\lambda=0.35$ with a mesh size of $h=1/50$. Errors $||u_{\Delta x}-u_s||_\infty$ and $||v_{\Delta x}-v_s||_1,$ are computed at each iteration and plotted against the number of iterations. The results depicted in Fig. \ref{fig:2dexp1} clearly indicate that the SO-$\Theta$ scheme outperforms the FO and SO schemes. In other words, the SO-$\Theta$ scheme converges to the steady state more rapidly in comparison to the FO and SO schemes, demonstrating its superior efficiency.


\begin{figure}
     \centering
     \begin{subfigure}[b]{0.49\textwidth}
         \centering
         \includegraphics[width=\textwidth]{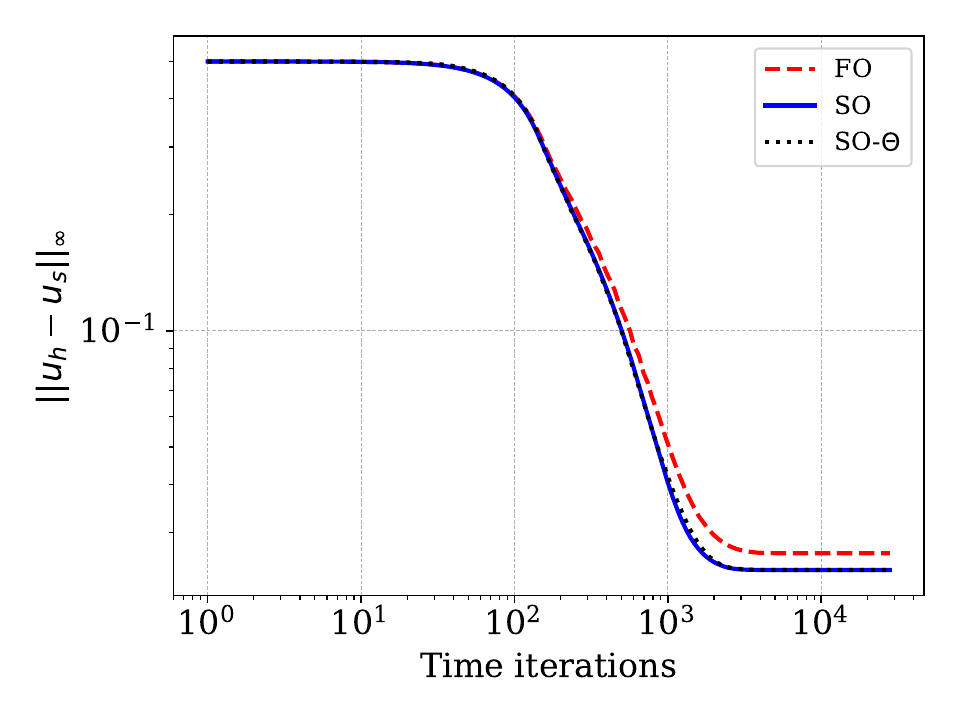}
         \caption{$u_h$}
     \end{subfigure}
     \begin{subfigure}[b]{0.49\textwidth}
         \centering
         \includegraphics[width=\textwidth]{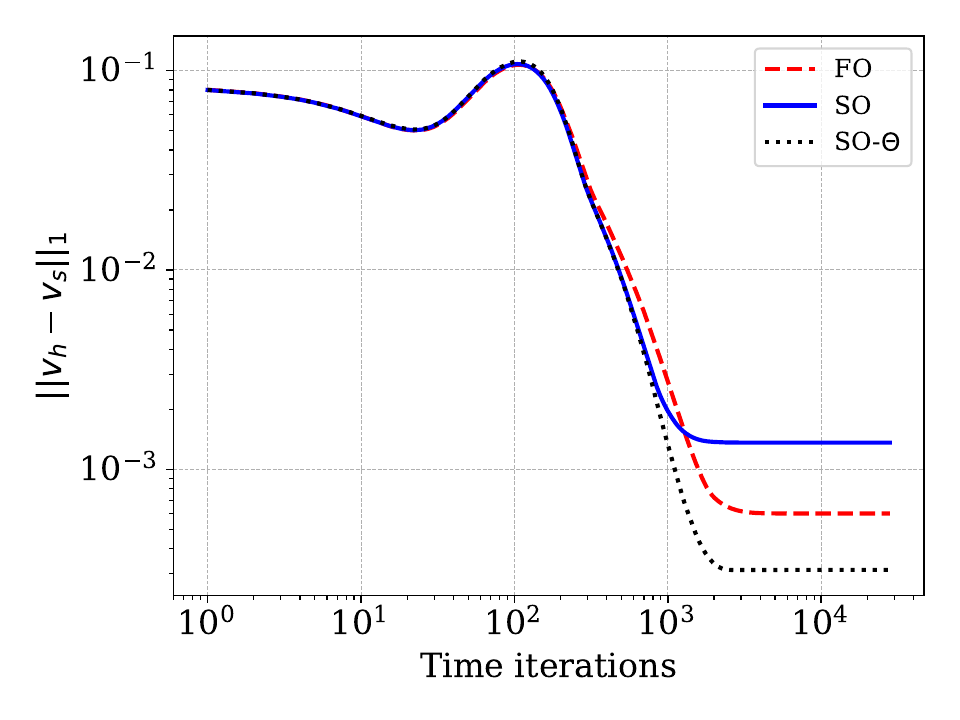}
         \caption{$v_h$}
     \end{subfigure}
\caption{Example~6 (2D): numerical errors versus number of time iterations plots for the open table problem with $h = 1/50$  and $\Delta t = 0.35 h.$ Numerical solutions are evolved till they reach near the steady state.}
\label{fig:2dexp1}
\end{figure}

\begin{figure}
     \centering
     \begin{subfigure}[b]{0.45\textwidth}
         \centering
         \includegraphics[width=\textwidth]{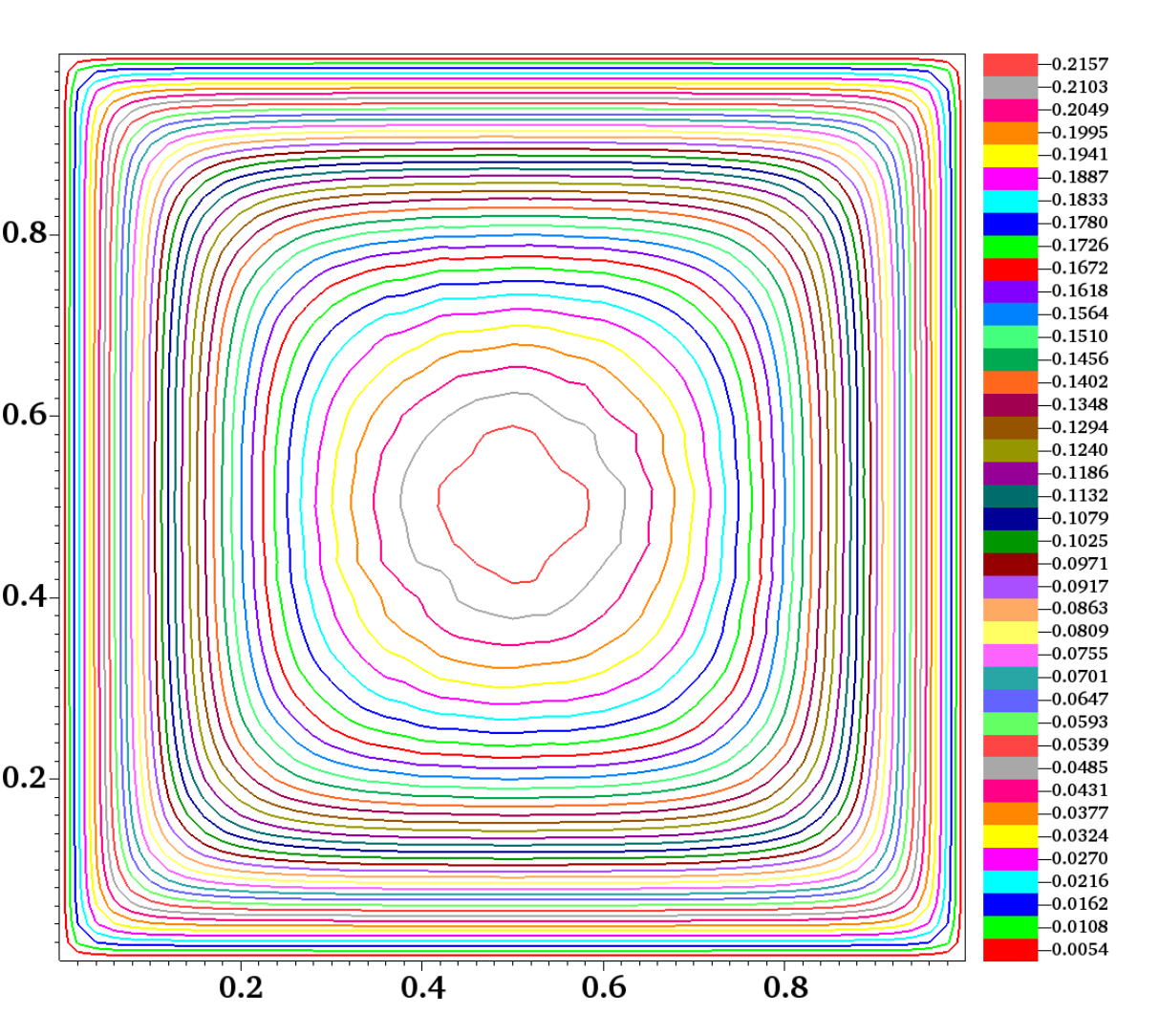}
         \caption{$u_h:$ using FO scheme, $\Delta t = 0.35 h$} 
     \end{subfigure}
     \begin{subfigure}[b]{0.45\textwidth}
         \centering
         \includegraphics[width=\textwidth]{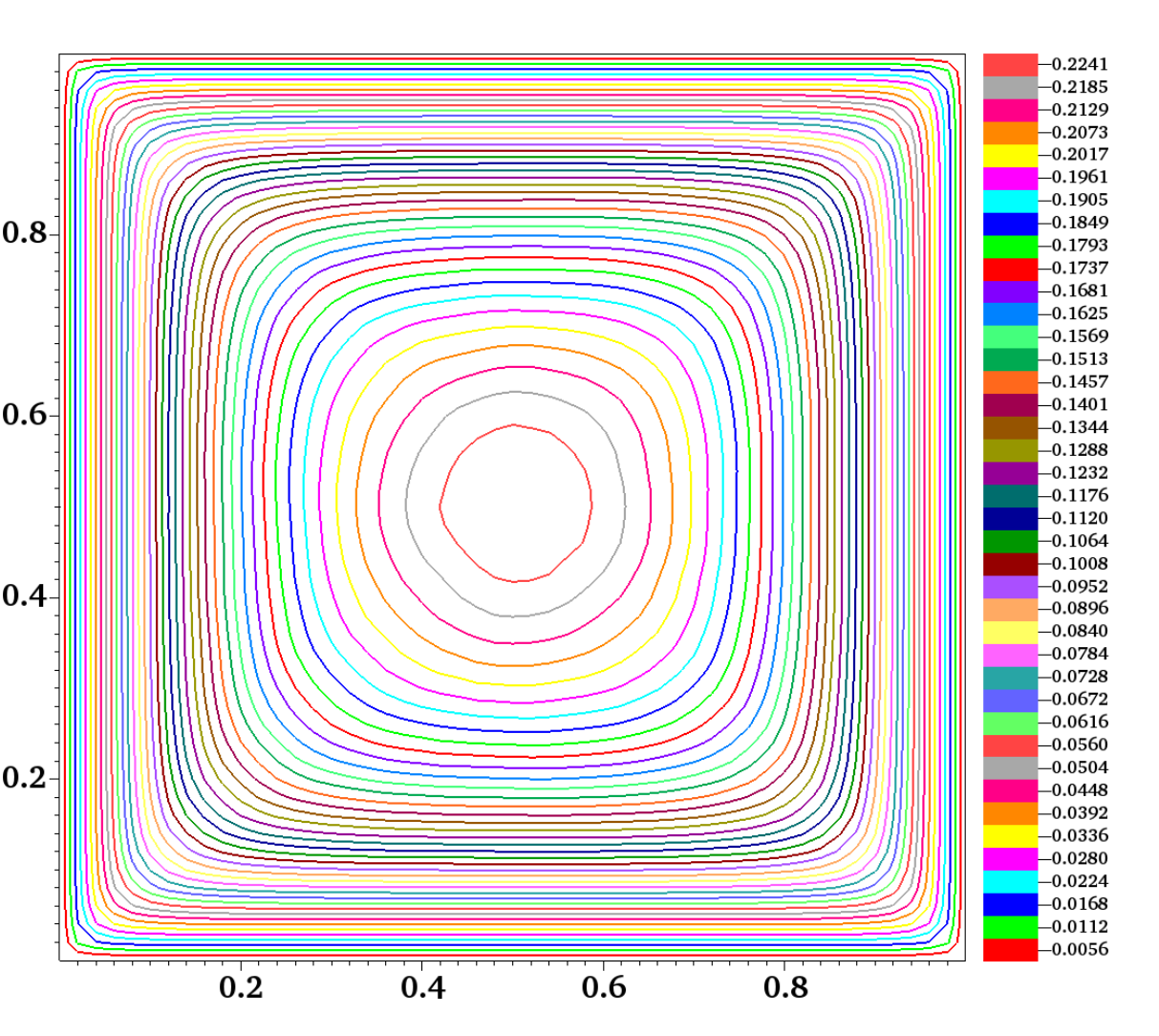}
         \caption{$u_h:$ using  SO-$\Theta$ scheme, $\Delta t = 0.35 h$}
     \end{subfigure}\\
          \begin{subfigure}[b]{0.45\textwidth}
         \centering
         \includegraphics[width=\textwidth]{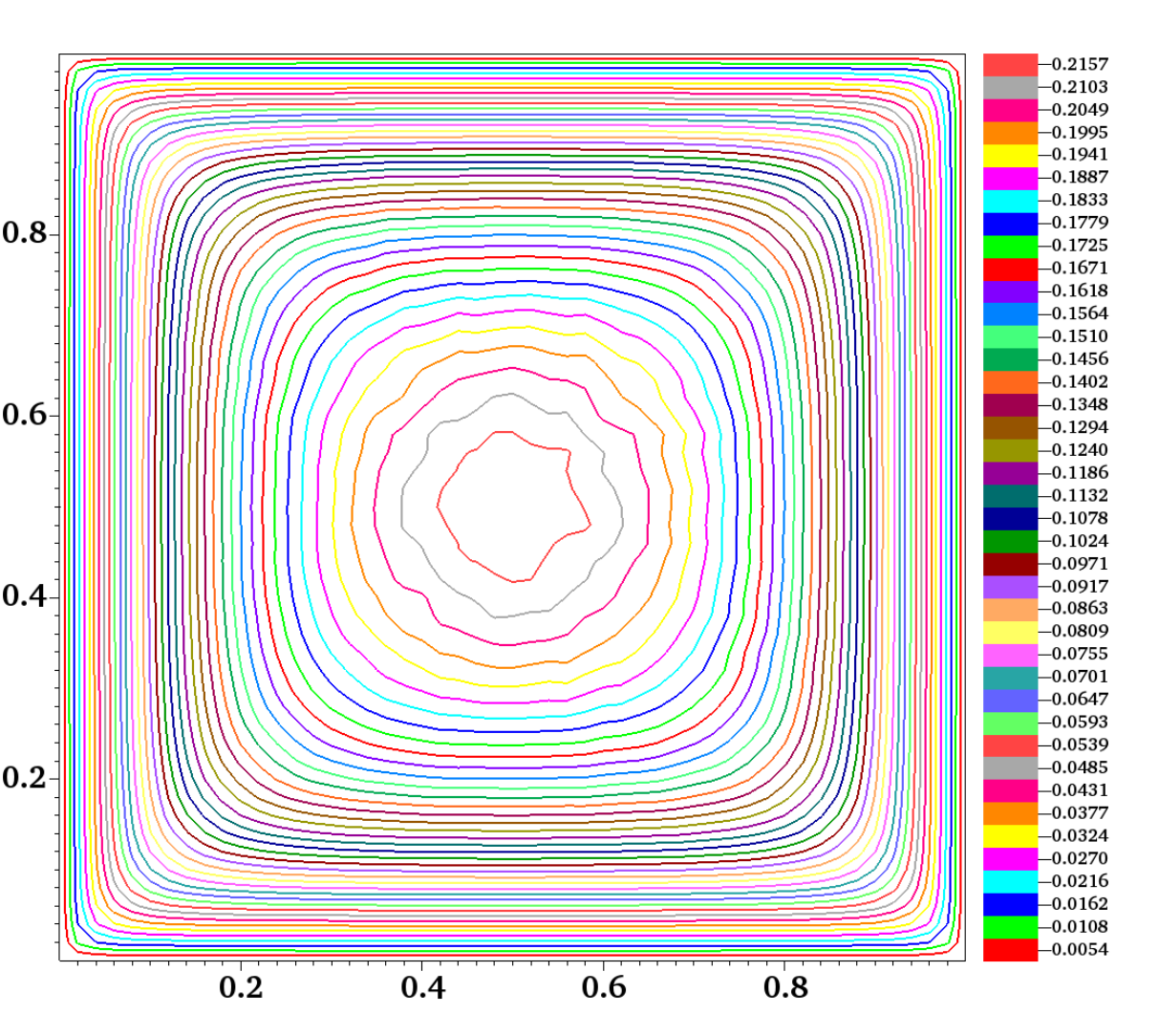}
         \caption{$u_h:$ using FO scheme, $\Delta t = 0.45 h$} 
     \end{subfigure}
     \begin{subfigure}[b]{0.45\textwidth}
         \centering
         \includegraphics[width=\textwidth]{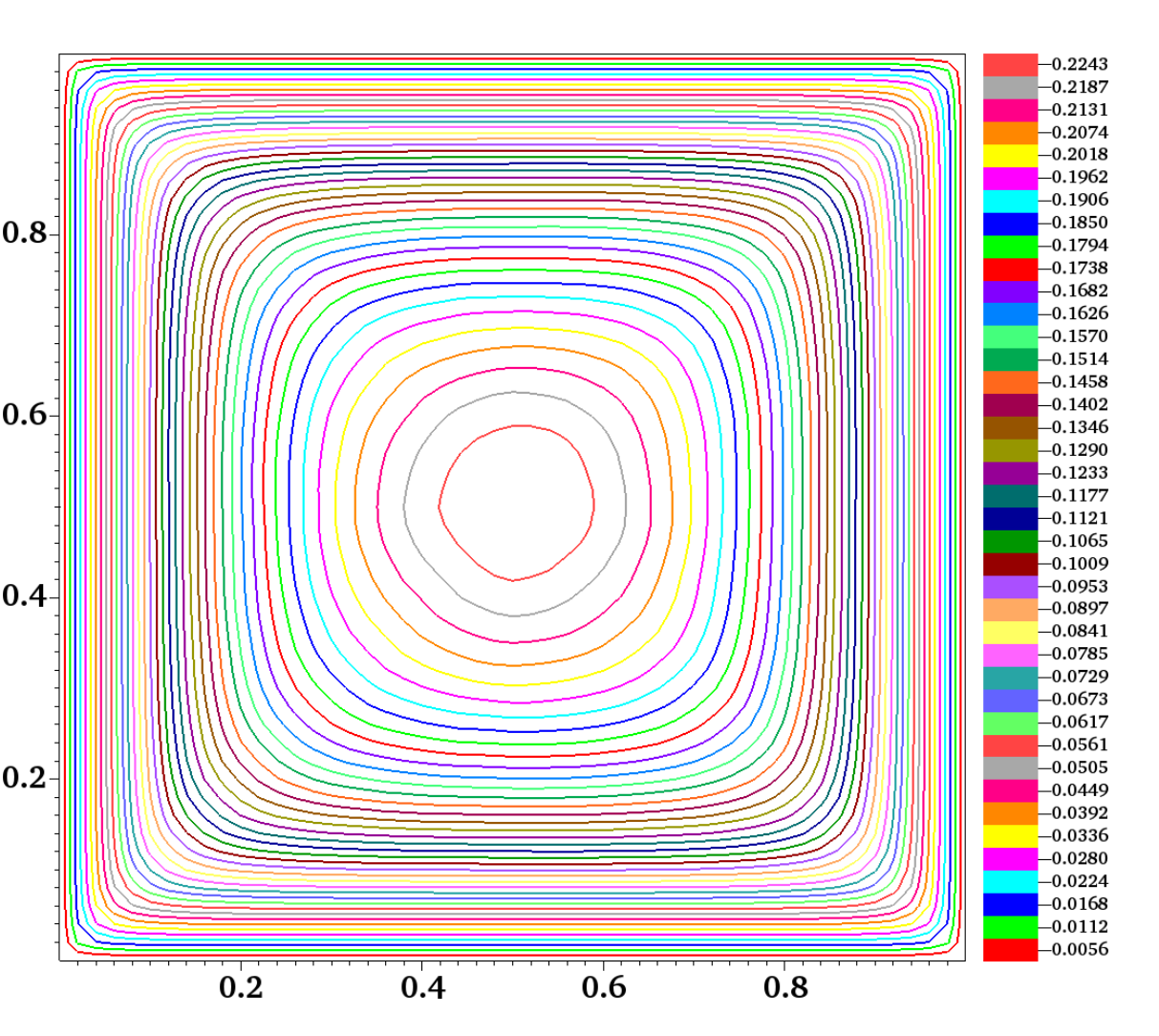}
         \caption{$u_h:$ using  SO-$\Theta$ scheme, $\Delta t = 0.45 h$}
     \end{subfigure}\\
           \begin{subfigure}[b]{0.45\textwidth}
         \centering
         \includegraphics[width=\textwidth]{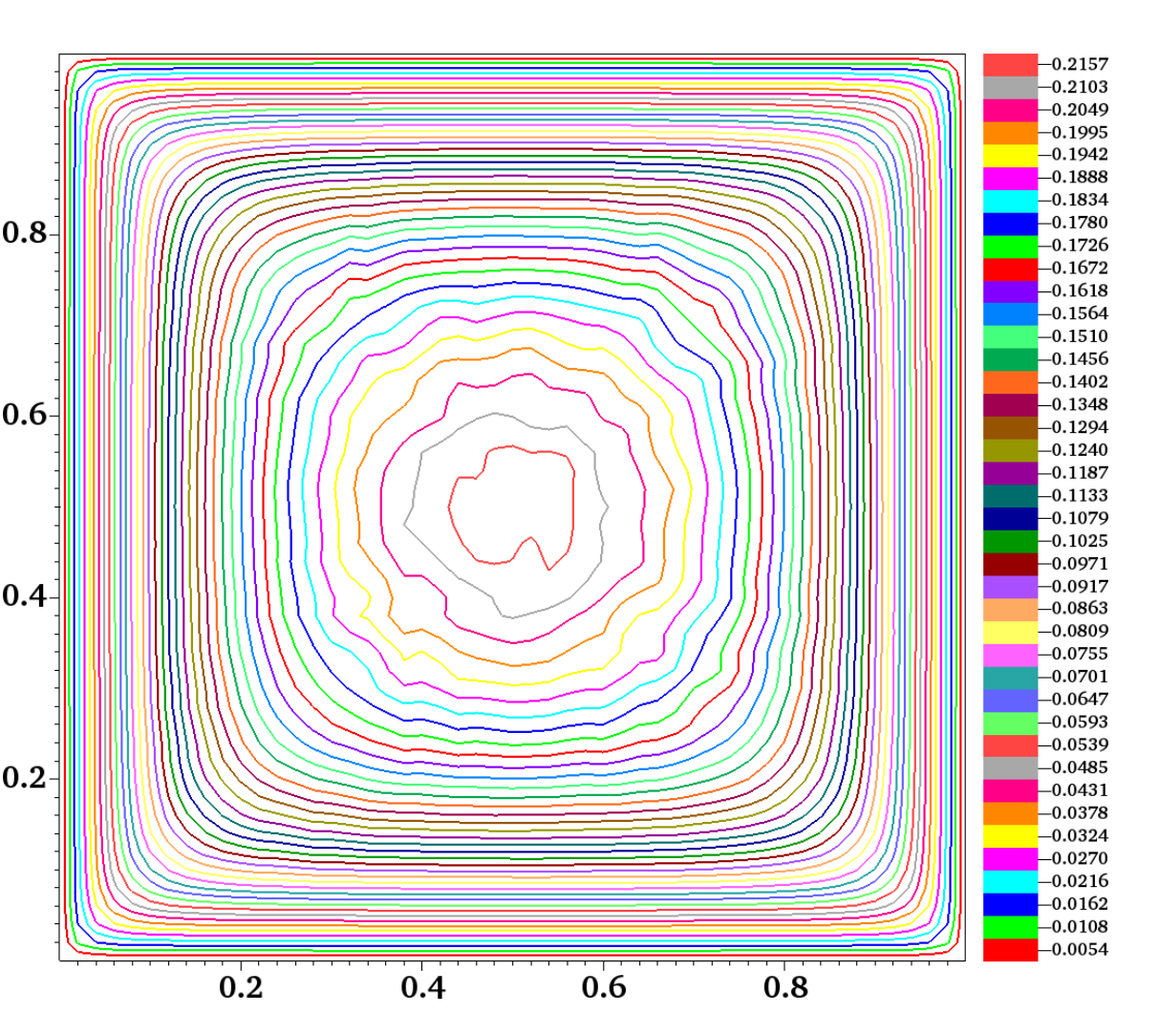}
         \caption{$u_h:$ using FO scheme, $\Delta t = 0.7 h$} 
     \end{subfigure}
     \begin{subfigure}[b]{0.45\textwidth}
         \centering
         \includegraphics[width=\textwidth]{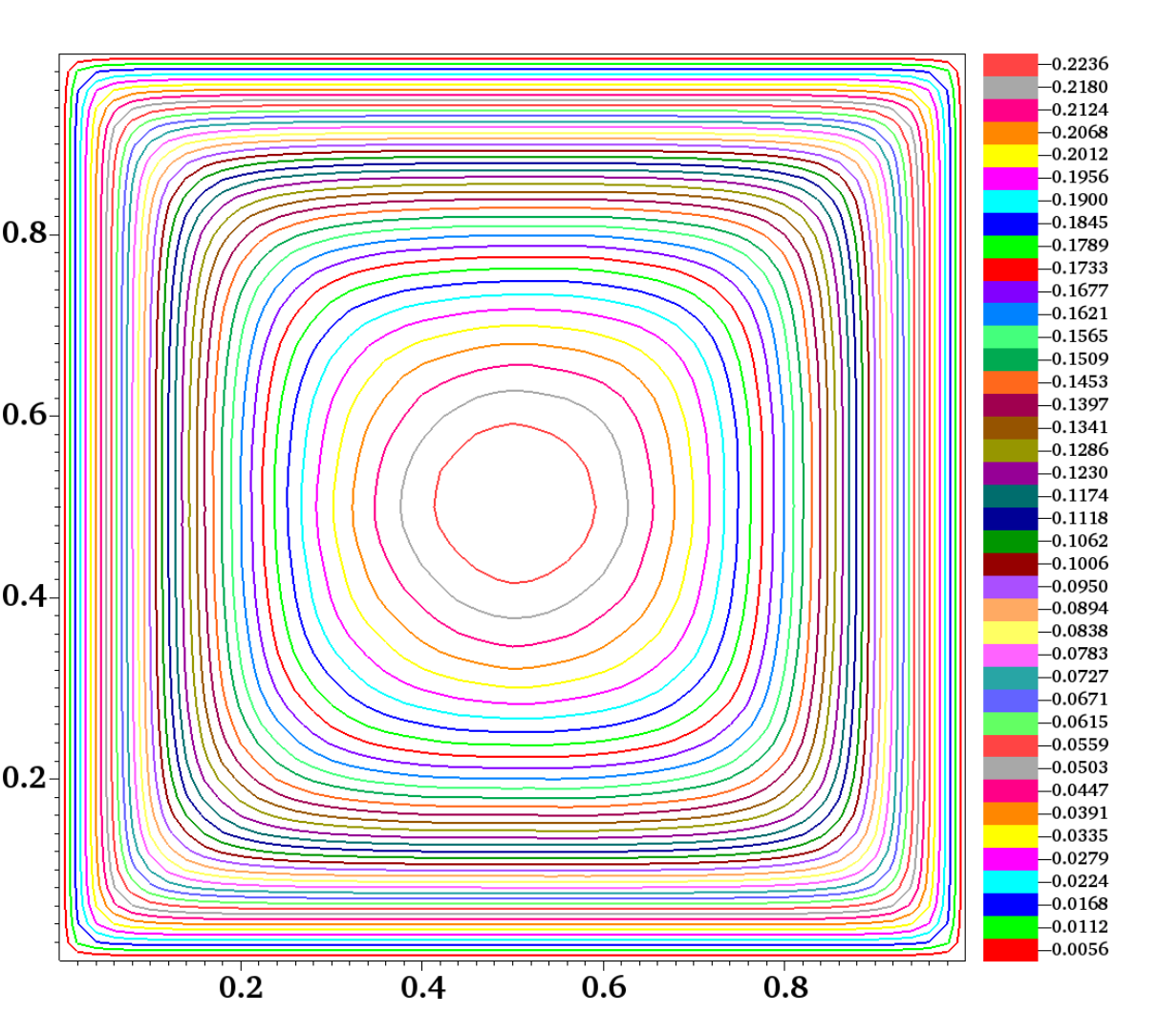}
         \caption{$u_h:$ using  SO-$\Theta$ scheme, $\Delta t = 0.7 h$}
     \end{subfigure}
\caption{Example~7 (2D): numerical solutions computed at  time $T = 1.2$ with a mesh size $h  = 1/50 $ and using values of $\lambda:$ 0.35, 0.45 and 0.7. The contour plots use 40 contour curves.}
\label{fig:2dopencflu}
\end{figure}

\begin{figure}
     \centering
     \begin{subfigure}[b]{0.45\textwidth}
         \centering
         \includegraphics[width=\textwidth]{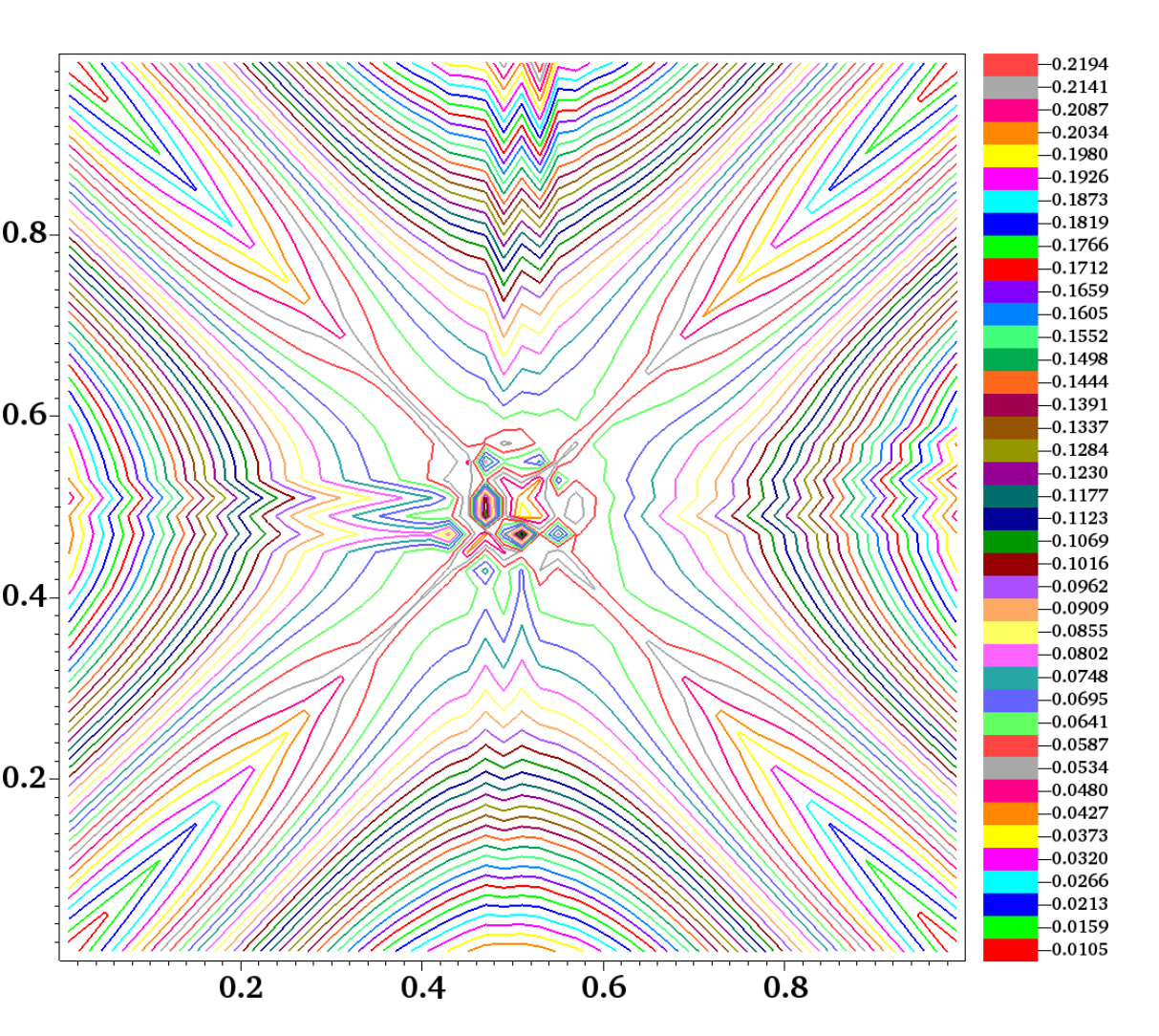}
         \caption{$v_h:$ using FO scheme, $\Delta t = 0.35 h$} 
     \end{subfigure}
     \begin{subfigure}[b]{0.45\textwidth}
         \centering
         \includegraphics[width=\textwidth]{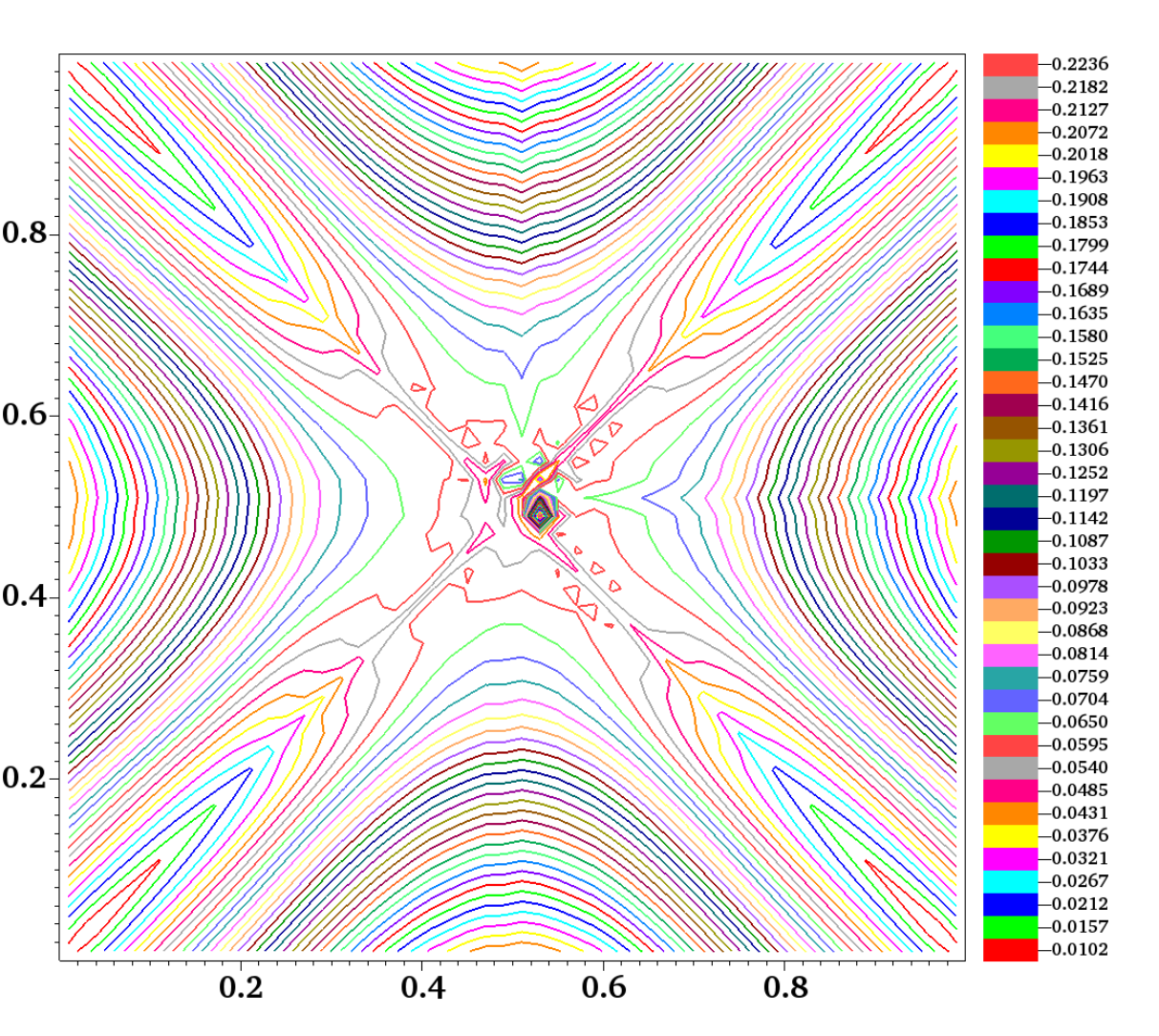}
         \caption{$v_h:$ using  SO-$\Theta$ scheme, $\Delta t = 0.35 h$}
     \end{subfigure}\\
          \begin{subfigure}[b]{0.45\textwidth}
         \centering
         \includegraphics[width=\textwidth]{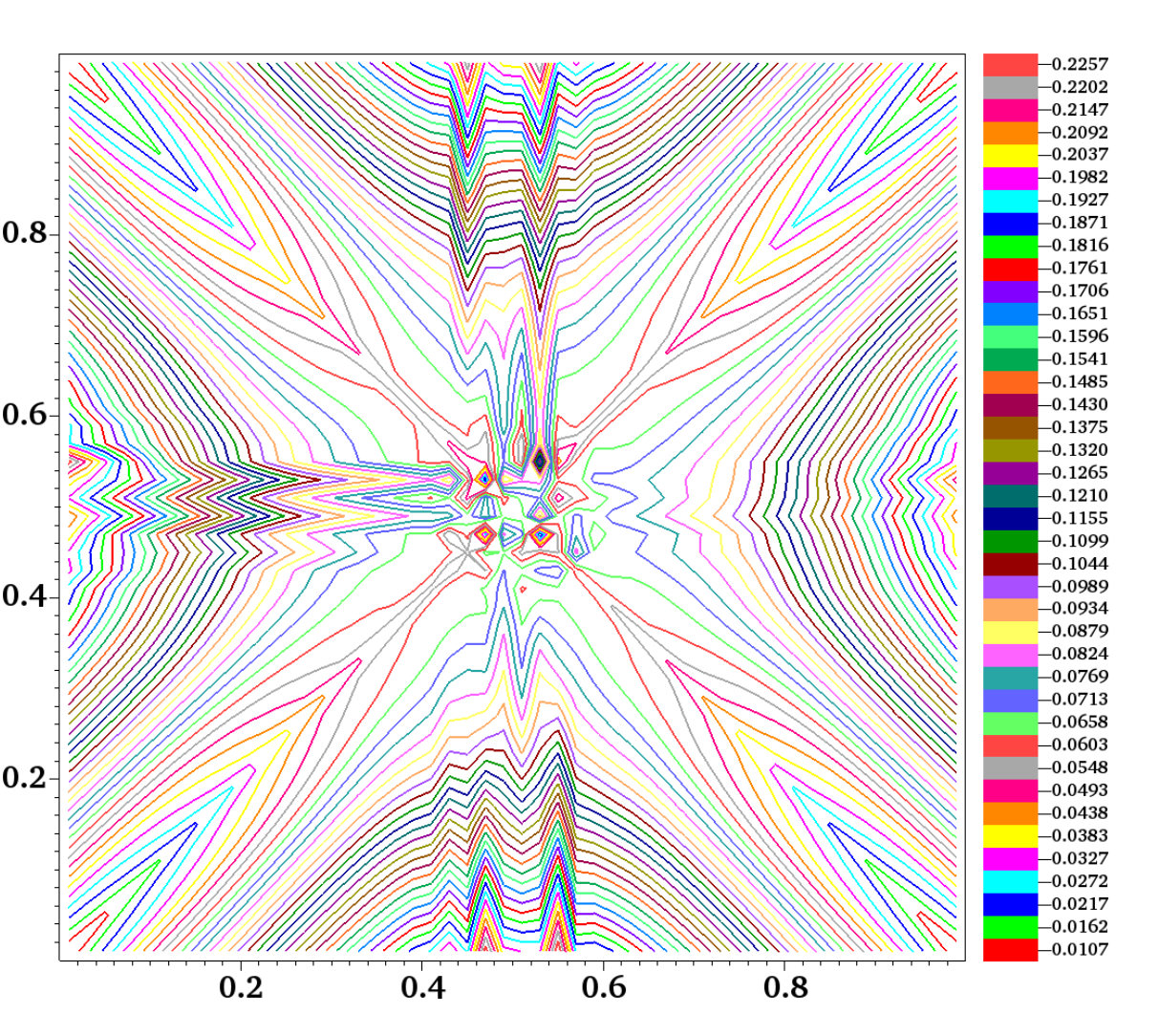}
         \caption{$v_h:$ using FO scheme, $\Delta t = 0.45 h$} 
     \end{subfigure}
     \begin{subfigure}[b]{0.45\textwidth}
         \centering
         \includegraphics[width=\textwidth]{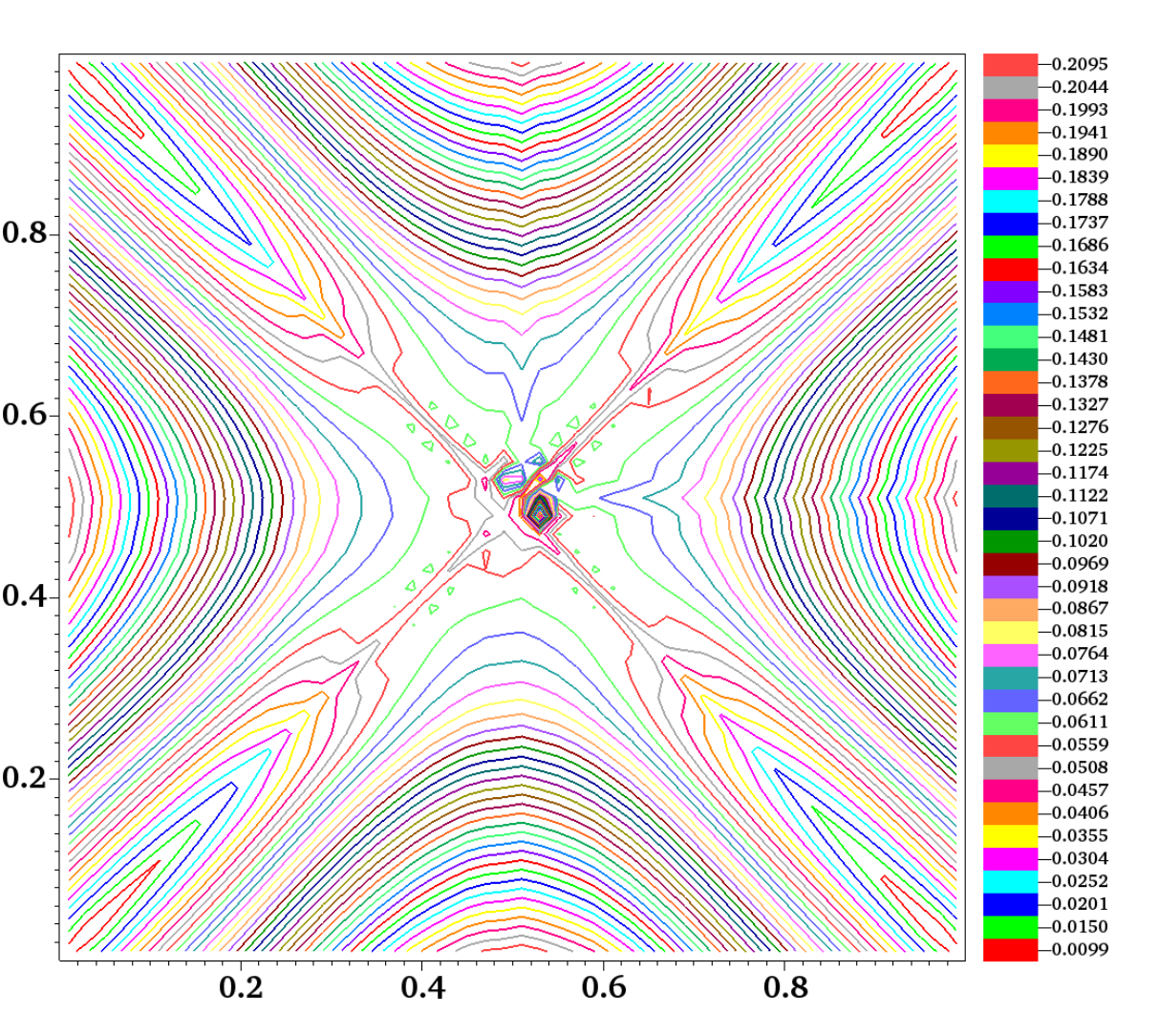}
         \caption{$v_h:$ using  SO-$\Theta$ scheme, $\Delta t = 0.45 h$}
     \end{subfigure}\\
           \begin{subfigure}[b]{0.45\textwidth}
         \centering
         \includegraphics[width=\textwidth]{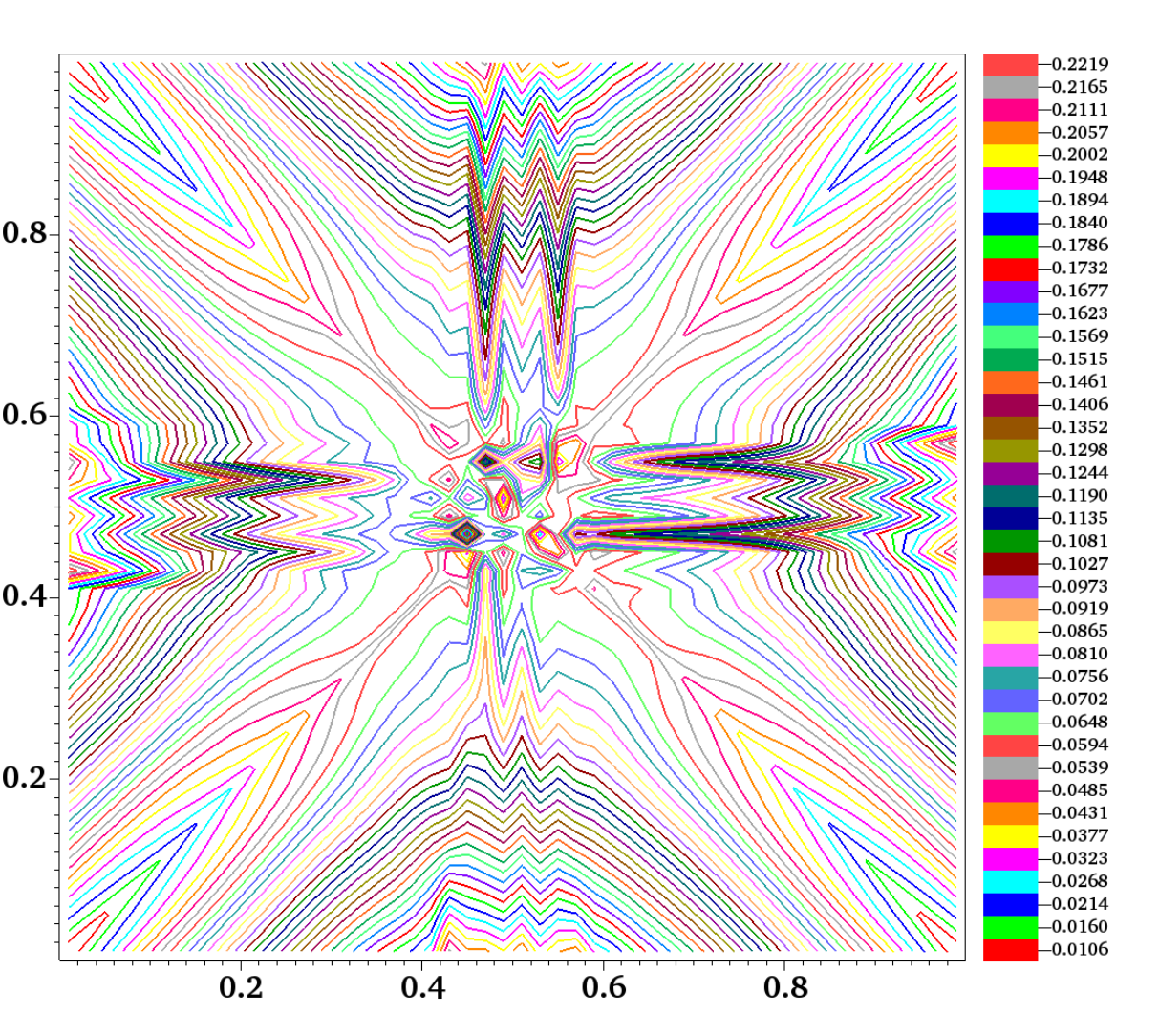}
         \caption{$v_h:$ using FO scheme, $\Delta t = 0.7 h$} 
     \end{subfigure}
     \begin{subfigure}[b]{0.45\textwidth}
         \centering
         \includegraphics[width=\textwidth]{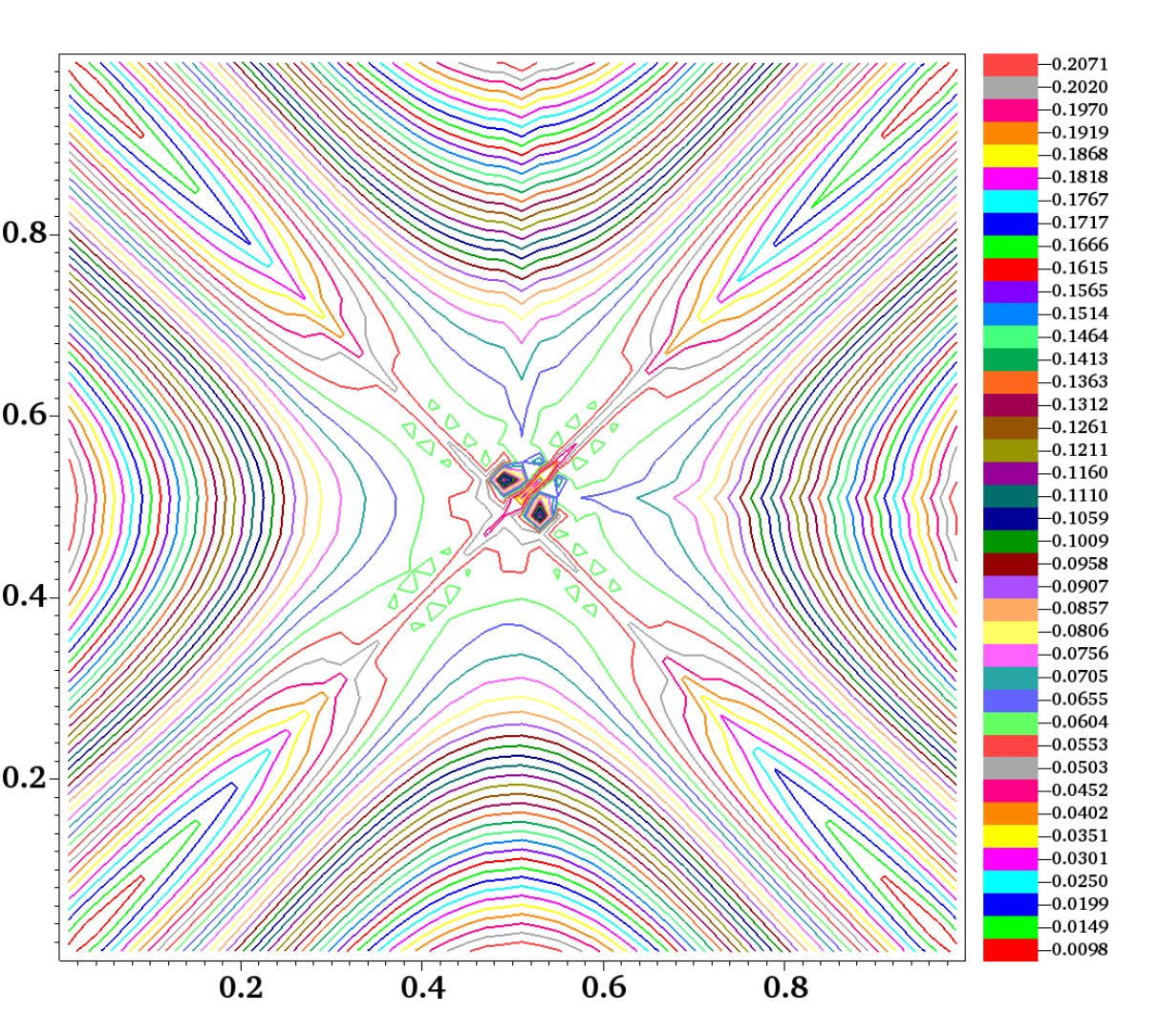}
         \caption{$v_h:$ using  SO-$\Theta$ scheme, $\Delta t = 0.7 h$}
     \end{subfigure}
\caption{Example~7 (2D):  numerical solutions computed at  time $T = 3.0$ with a mesh size  $h  = 1/50 $ and using values of $\lambda$: 0.35, 0.45 and 0.7. The contour plots use 40 contour curves.}
\label{fig:2dopencfl}
\end{figure}

\textbf{Example 7 (solutions at finite time-2D)} 
Here again   we consider the open table problem  and  compare the solution profiles at a finite time, specifically at $T=3.0.$
The main aim here is to show that the SO-$\Theta$ scheme performs well even for the larger $\lambda$.
We compute the approximate solution $v_h$ using both the FO and SO-$\Theta$ schemes and plot  with contour curves. The computations are performed for different values of $\lambda$: 0.35, 0.45, and 0.7, all with a mesh size of $h=1/50$  and the  results are given in Fig. \ref{fig:2dopencfl}. 
For $\lambda = 0.35,$ we see a clear distinction between the two schemes. In the case of $v:$ the FO scheme exhibits oscillations, particularly noticeable when observing the contour plots along the lines $\{ (x,y):x=0.5, 0\le y \le 1\}$ and $\{ (x,y):y=0.5,0\le x \le 1\}$ in Fig. \ref{fig:2dopencfl} (a) and (b). In contrast, the SO-$\Theta$ scheme generates fewer oscillations, indicating its robustness.
Furthermore, as the  value of $\lambda$ increase to 0.45 and 0.7, the oscillations intensify in the FO schemes. However, the SO-$\Theta$ scheme remains stable and less-oscillatory. This is  evident in Fig. \ref{fig:2dopencfl} (b), (d) and (f). In the case of $u:$ we observe this same behaviour  at time $T =1.2,$ which is given in Fig. \ref{fig:2dopencflu}.
This comparison highlights the robustness of the SO-$\Theta$ scheme in producing accurate solutions at finite times, emphasizing its significance in numerical simulations. 

\textbf{Example 8 (discontinuous source test case-2D)} In this example we consider the open table problem in the computational domain $[0,1]\times[0,1]$ with a discontinuous source function given by (see Example 2  of \cite{adimurthi2016a}, page no. 1110)
\begin{align*}
    f (x,y) = \begin{cases}
        0.5 &\mbox{ if } (x,y) \in D_f,\\
        0 & \mbox{ elsewhere,}
    \end{cases}
\end{align*}
where $D_f:= [0.1, 0.3]\times [0.5, 0.7] \cup [0.5,0.7]\times[0.7, 0.9].$ The decomposition of source function $f$ is as given in \cite{adimurthi2016a}. As explained in \cite{adimurthi2016a}, the discretization of the domain $[0,1]\times[0,1]$ is carefully done such that corners of the interior squares (see Fig. 21 in \cite{adimurthi2016a}) fall at the center of the computational cell. Due to this reason, in this test case we choose a mesh size of $h = 1/55.$ For more details on this experiment we refer to \cite{adimurthi2016a}. We compute the solutions at large time $T= 200,$ where it reaches the steady state with $\lambda =0.35.$ The results are printed in Fig. \ref{fig:disc2d}. It is evident that the SO-$\Theta$ scheme produces a better sharp resolution near the crest formation compared to the FO scheme. 

\begin{figure}
     \centering
     \begin{subfigure}[b]{0.45\textwidth}
         \centering
         \includegraphics[width=\textwidth]{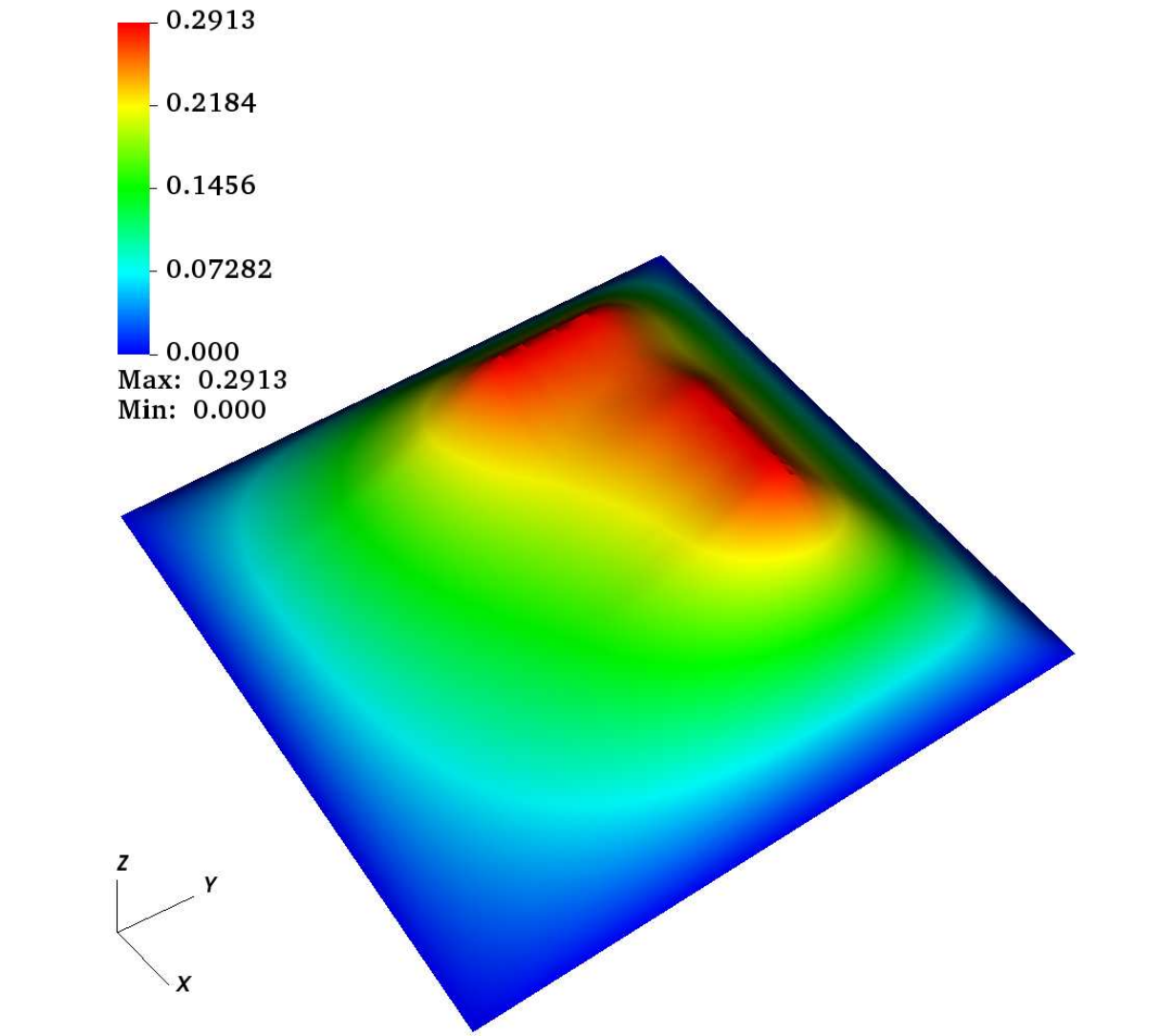}
         \caption{$u_{h}:$ using FO scheme}
     \end{subfigure}
     \begin{subfigure}[b]{0.45\textwidth}
         \centering
         \includegraphics[width=\textwidth]{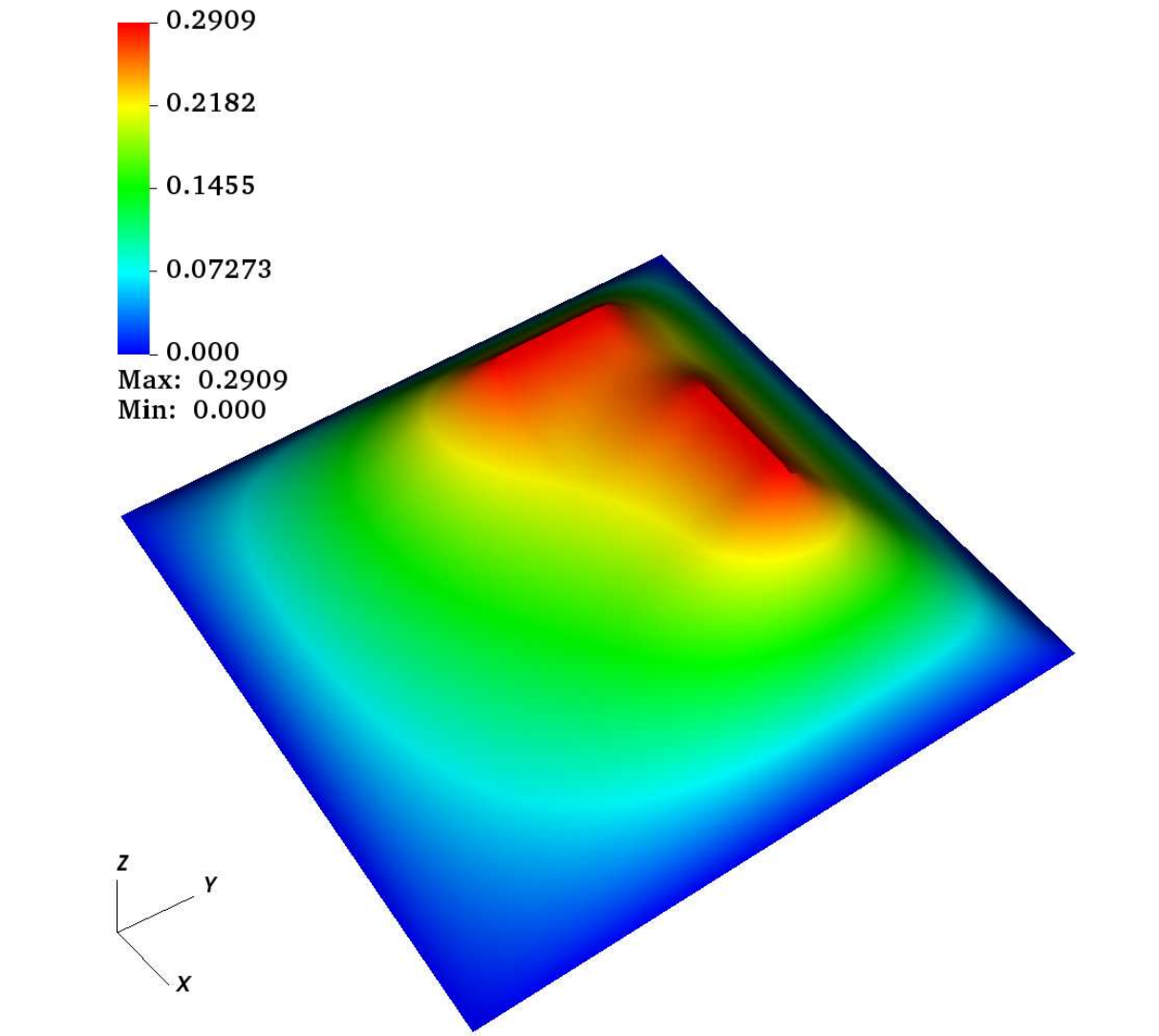}
         \caption{$u_{h}:$ using SO-$\Theta$ scheme}
         \label{fig:three sin x11}
     \end{subfigure}
\caption{Example~8 (2D): numerical solutions  corresponding to a discontinuous source, computed at time $T = 200$ with $h = 1/55$ and $ \Delta T = 0.35 h.$ }
\label{fig:disc2d}
\end{figure}

\textbf{Example~9 (partially open table problem-2D)} We consider the domain $ (0,1)\times (0,1)$ with the boundary walls $\Gamma_w$ and open  boundary $\Gamma_o$ as given in Fig. \ref{fig:transport} (b). In this scenario, we aim to solve the problem defined by equations \eqref{HKw1}-\eqref{HKw4}, with a given source function 
 \begin{align*}
    f (x,y) = 0.5 \mbox{ for all } (x,y) \in [0,1]\times [0,1].
\end{align*}
Further details about the source decomposition and problem description can be found in \cite{adimurthi2016}. 
The computations are performed with a mesh size $h=1/50$ and $\lambda=0.1.$ 
It is worth noting that we have opted here  a relatively small value of $\lambda$ in comparison to the previous examples. This is due to the fact that solution $v$ takes larger values, and $\lambda$ has to satisfy the condition: $\lambda \max_{i} v_{i}^{n}\le 1/2$  of Theorem~\ref{stablity}.
The simulations are done using both the FO and SO-$\Theta$ schemes until a large time $T=200$, allowing the numerical solutions to approach the steady state. For this example, the exact steady state solutions are available  and are given by
\begin{eqnarray*}\label{vf}
 u_s (x,y)=\left\{\begin{array}{llll}
y& \, \, \mbox{if}\, \, & x\le0.5,\\
\sqrt{ (x-\half)^2+y^2}& \, \, \mbox{if}\, \, & x>0.5,\\
\end{array}\right., \;\;
v_s (x,y)=\left\{\begin{array}{llll}
1-y& \, \, \mbox{if}\, \, & x\le0.5,\\
\displaystyle\frac{1}{d_\Gamma (x,y)}\int_{d_\Gamma (x,y)}^{l (x,y)}\rho d\rho& \, \, \mbox{if}\, \, & x>0.5,\\
\end{array}\right.\
\end{eqnarray*}
\mbox{where}\\
\[l (x,y)=\left\{\begin{array}{llll}
\sqrt{ (1-\half)^2+ (\frac{0.5y}{x-\half})^2}& \, \, \mbox{if}\, \, & \frac{y}{x-\half}\le\frac{1}{0.5},\\
\sqrt{1+ (\frac{x-\half}{y})^2}& \, \, \mbox{if}\, \, & \frac{y}{x-\half} > \frac{1}{0.5},\\
\end{array}\right.,\;\;d_\Gamma (x,y)=\sqrt{ (x-\half)^2+y^2}\]
These solutions  are plotted using  the linear interpolation on a mesh of size $h = 1/200.$ The comparison of numerical solutions obtained from the FO and SO-$\Theta$ schemes with the exact steady state solutions is presented in Fig. \ref{fig:pot2a} and \ref{fig:pot2contur}. The results are visualized using 40 contour curves in Fig. \ref{fig:pot2contur}. Notably, in the vicinity of the point $ (0.5,0)$ where the solution $v_h$ exhibits a singularity, the SO-$\Theta$ scheme provides significantly better resolutions compared to the FO scheme. This enhancement is particularly evident when comparing the numerical solutions with the exact steady state solution. These results clearly demonstrate the advantage of the SO-$\Theta$ scheme over the FO scheme in the partially open table test case.

\begin{figure}
     \centering
     \begin{subfigure}[b]{0.45\textwidth}
         \centering
         \includegraphics[width=\textwidth]{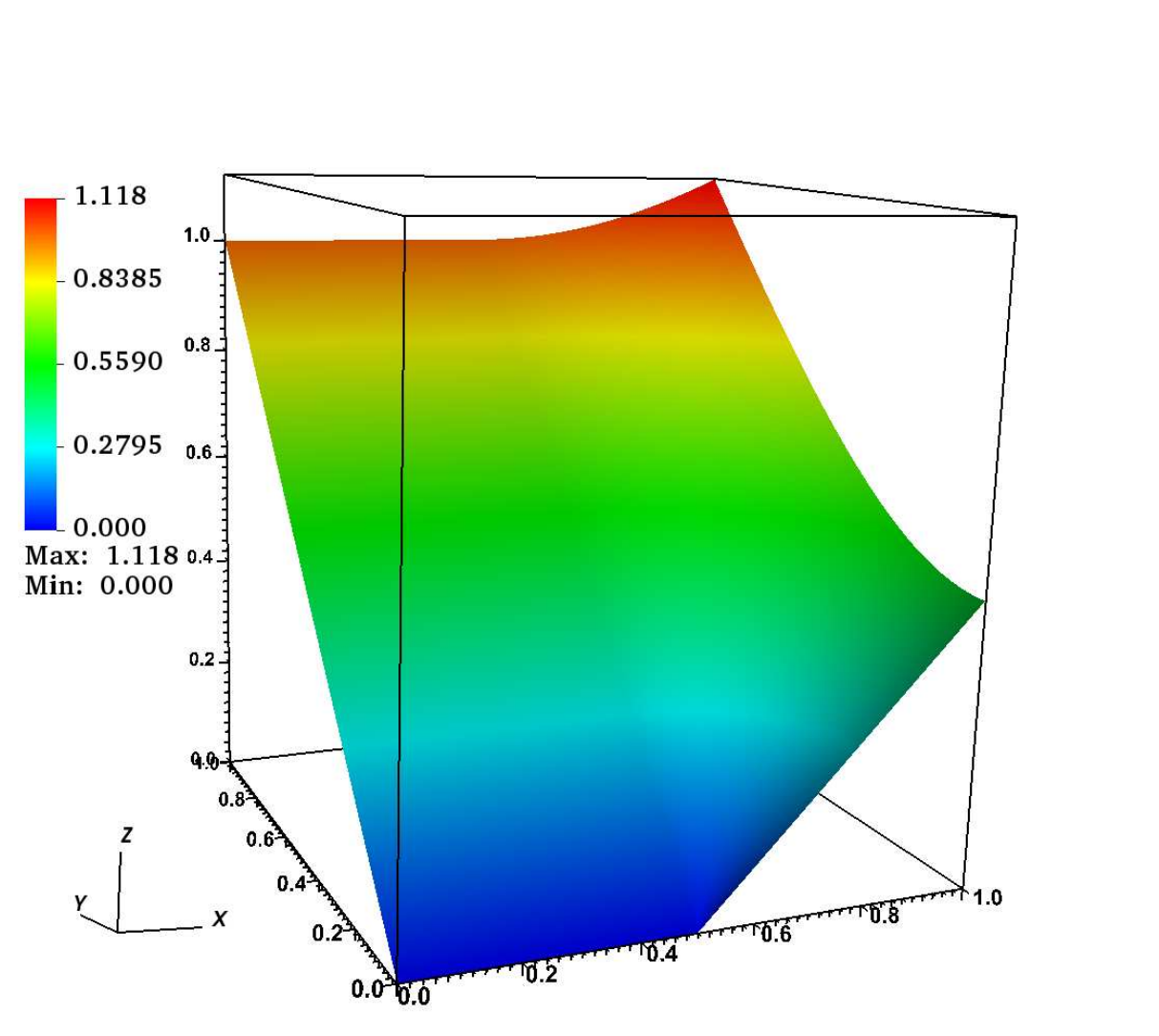}
         \caption{$\overline{u}:$ exact steady state solution}
     \end{subfigure}
     \begin{subfigure}[b]{0.45\textwidth}
         \centering
         \includegraphics[width=\textwidth]{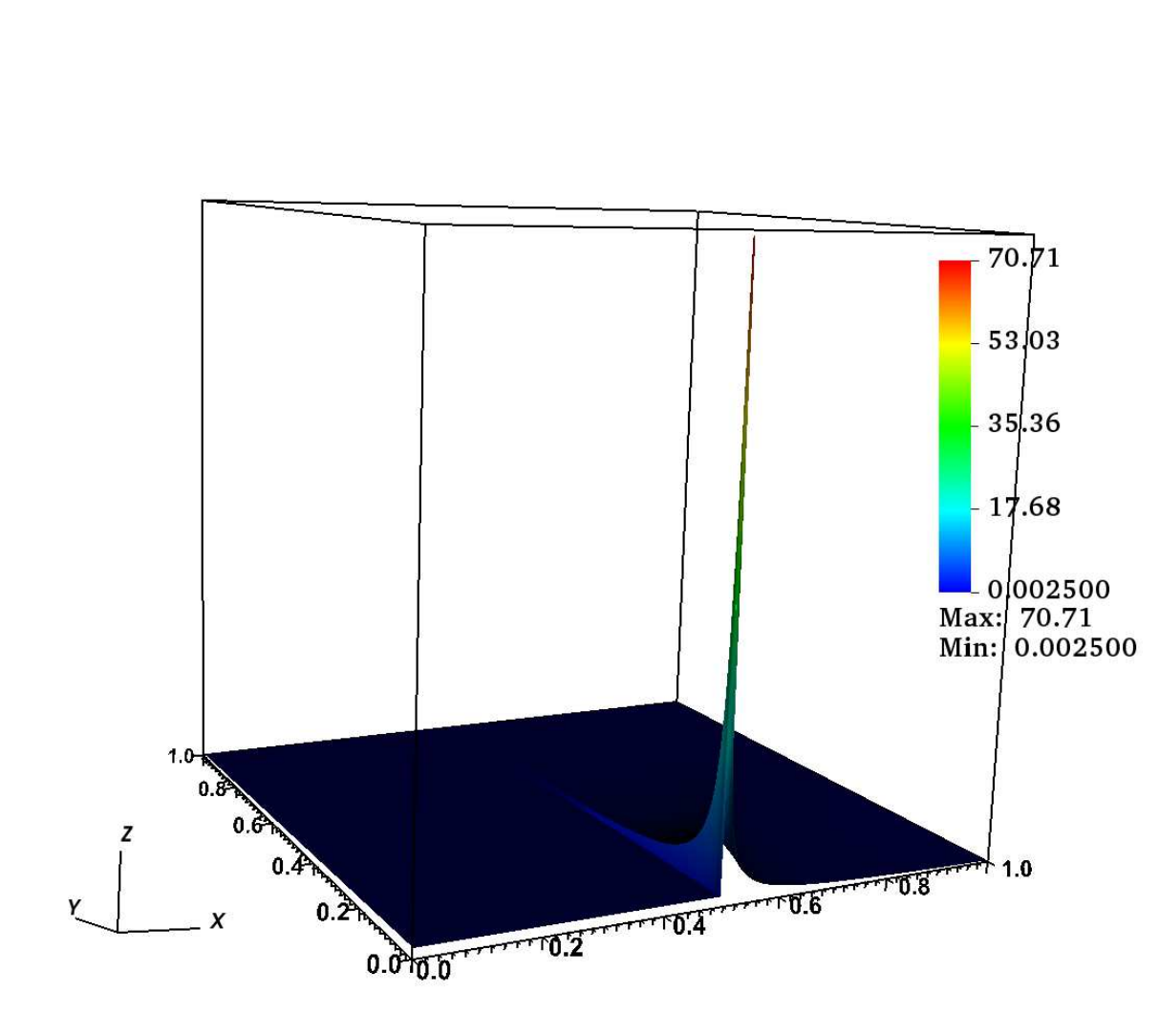}
         \caption{$\overline{v}:$ exact steady state solution}
         \label{fig:three sin x1}
     \end{subfigure}\\
     \begin{subfigure}[b]{0.45\textwidth}
         \centering
         \includegraphics[width=\textwidth]{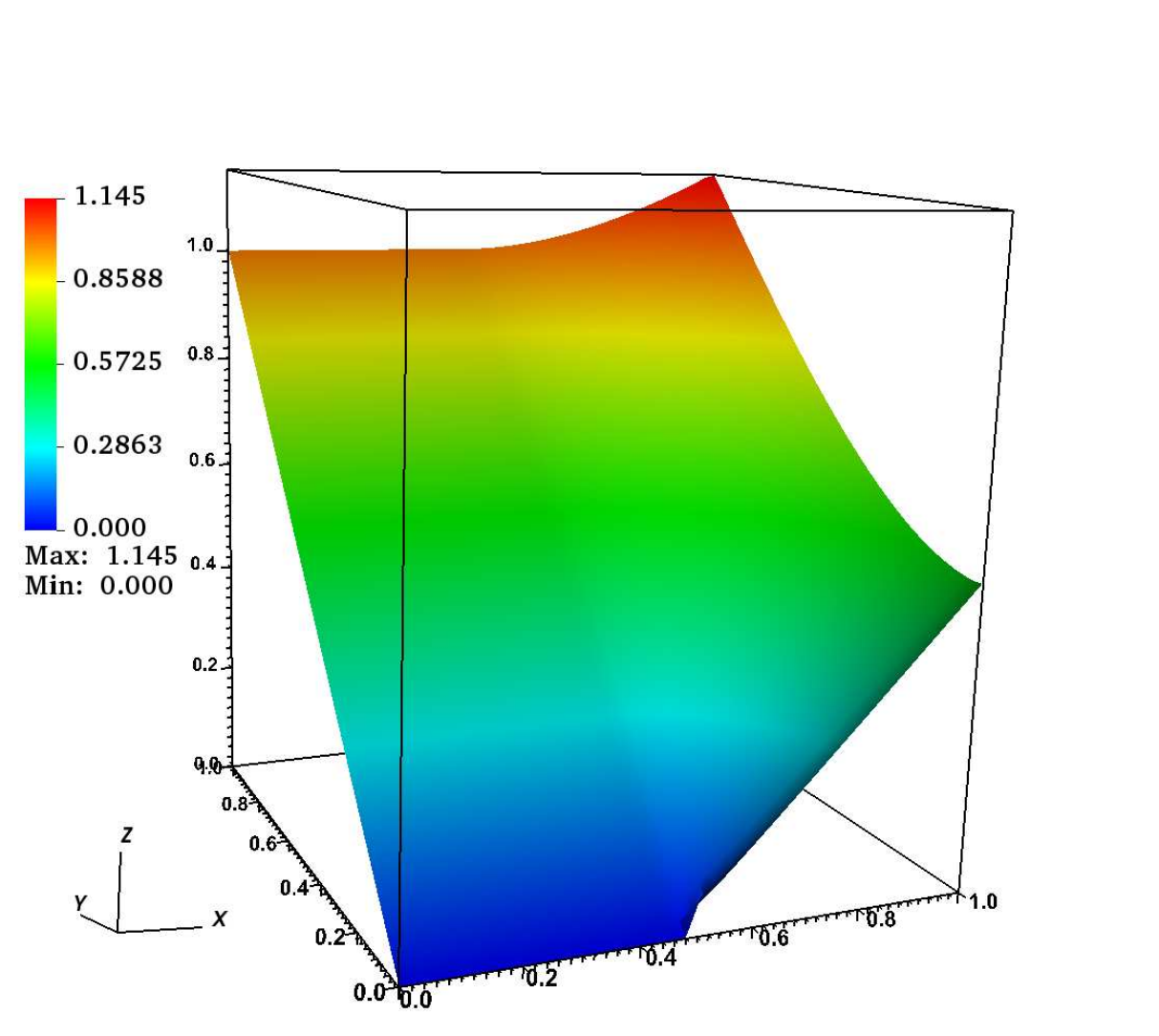}
         \caption{$u_h:$ using FO scheme}
     \end{subfigure}
     \begin{subfigure}[b]{0.45\textwidth}
         \centering
         \includegraphics[width=\textwidth]{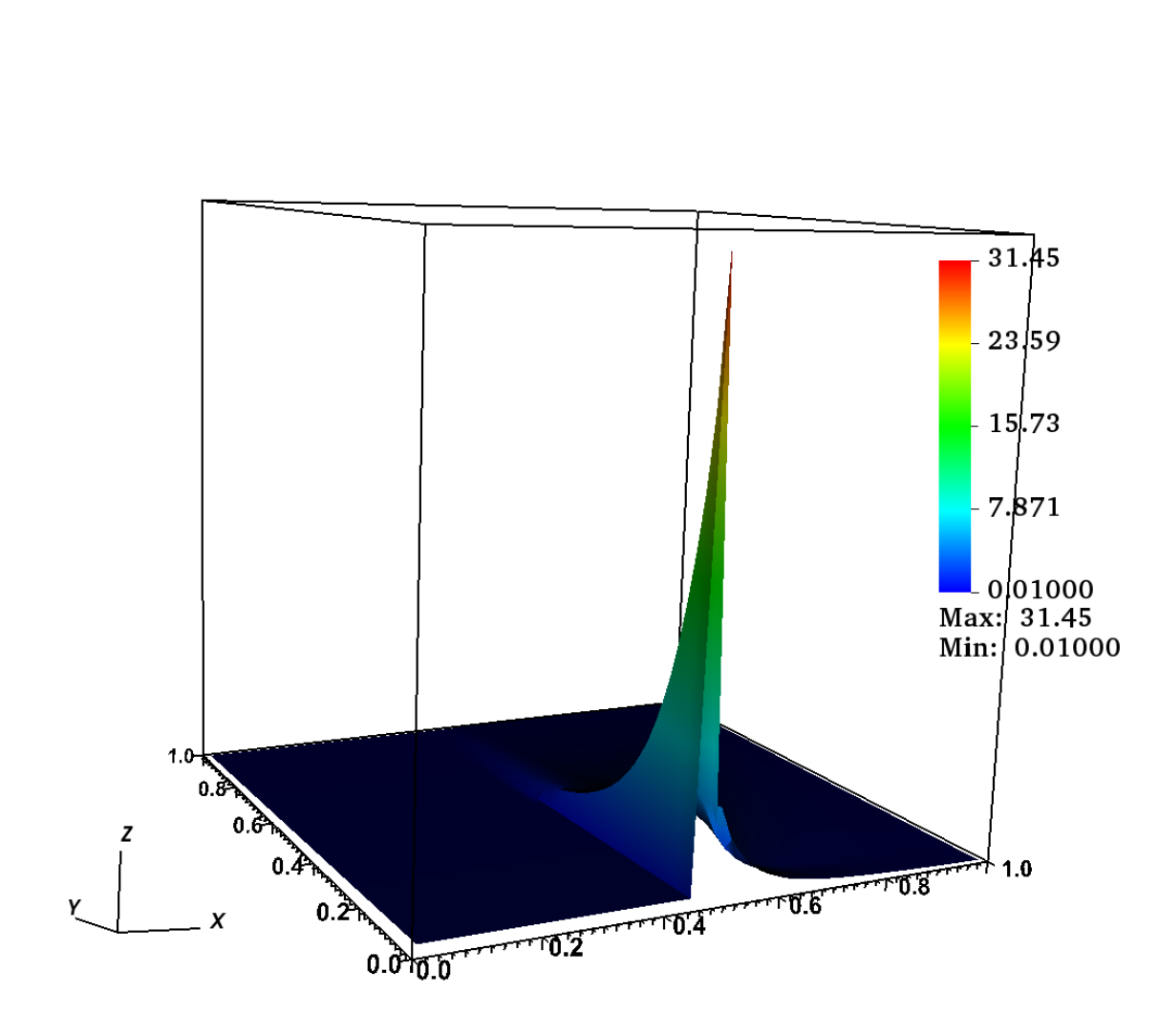}
         \caption{$v_h:$ using FO scheme }
         \label{fig:three sin x2}
     \end{subfigure}\\
              \begin{subfigure}[b]{0.49\textwidth}
         \centering
         \includegraphics[width=\textwidth]{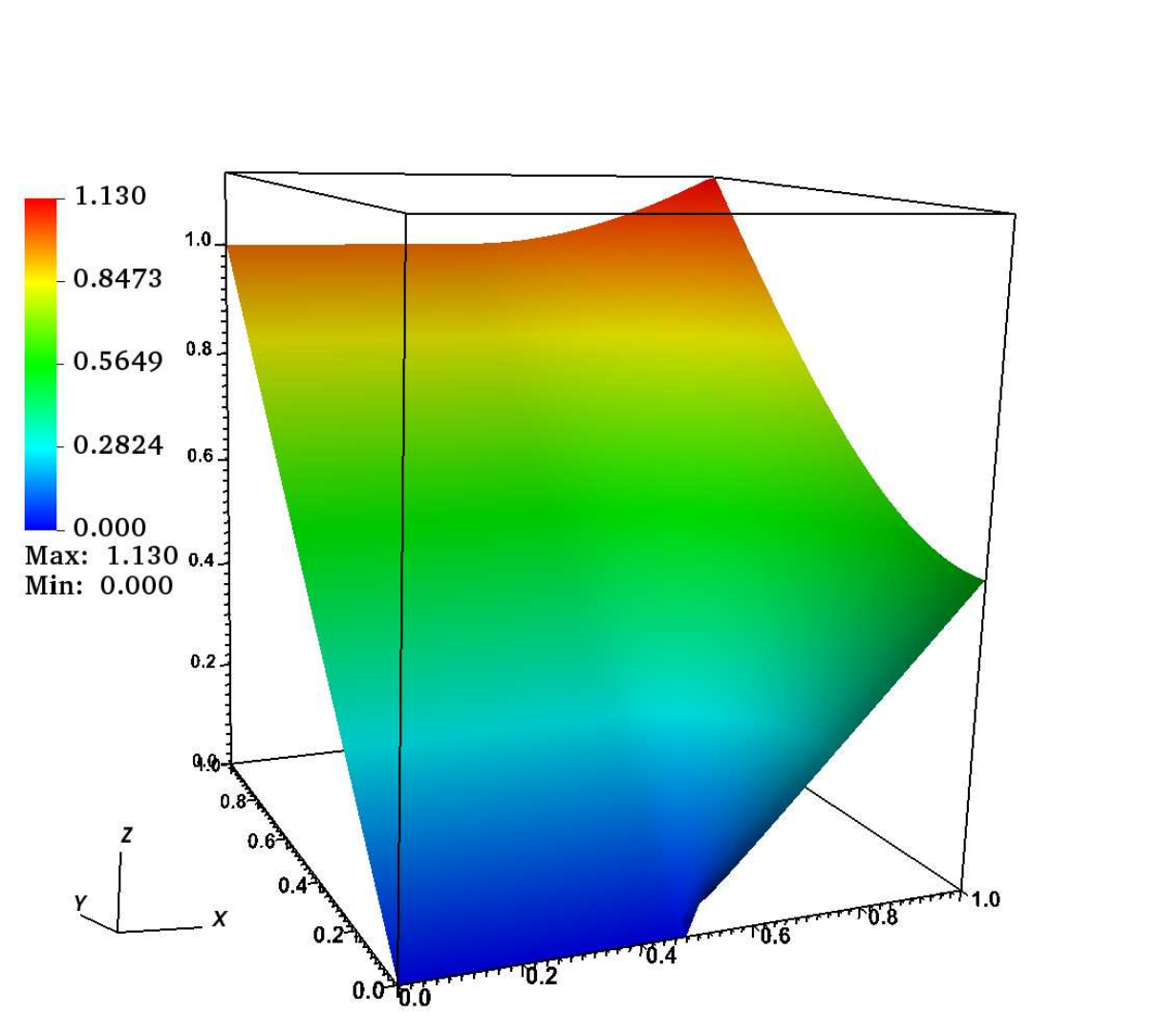}
         \caption{$u_h:$ using SO-$\Theta$ scheme}
     \end{subfigure}
     \begin{subfigure}[b]{0.49\textwidth}
         \centering
         \includegraphics[width=\textwidth]{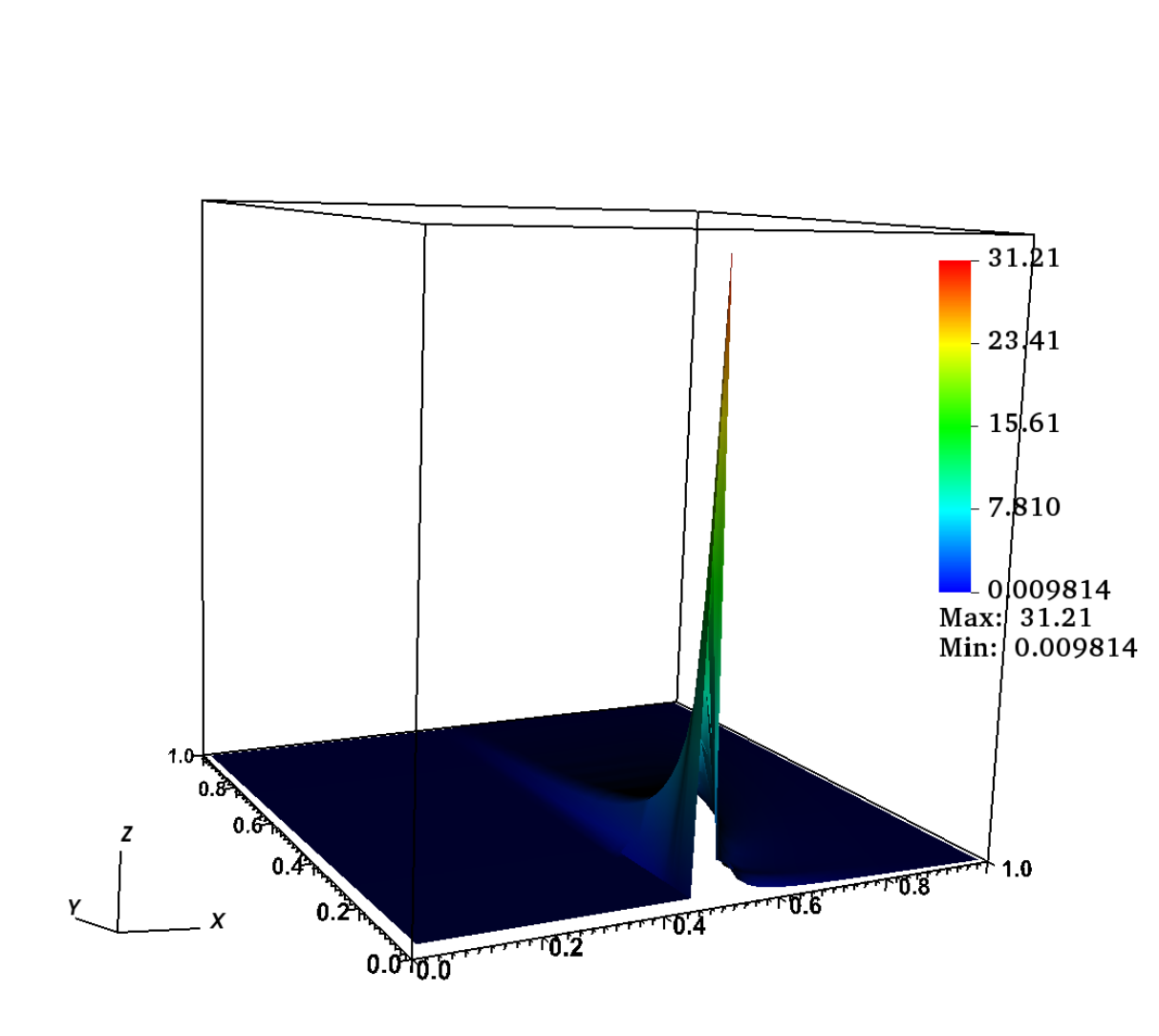}
         \caption{$v_h:$ using SO-$\Theta$ scheme }
         \label{fig:three sin x4}
     \end{subfigure}
\caption{ Example~9 (2D): numerical solutions of partially open table problem  with $h = 1/50$  and $ \Delta t = 0.1 h$ computed at the final time $T = 200.$}
\label{fig:pot2a}
\end{figure}

\begin{figure}
     \centering
     \begin{subfigure}[b]{0.45\textwidth}
         \centering
         \includegraphics[width=\textwidth]{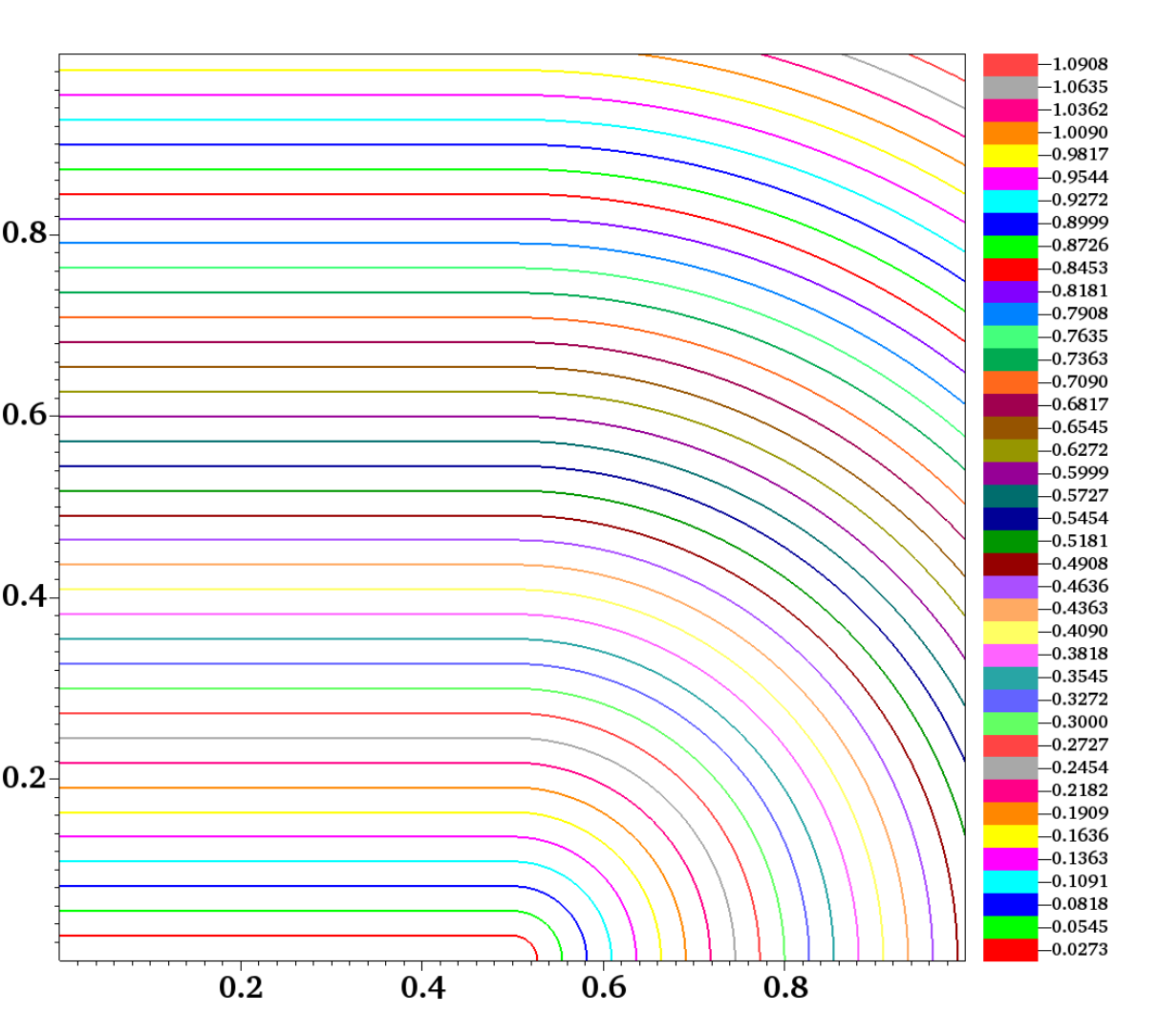}
         \caption{$u_{s}: $ exact steady state solution}
     \end{subfigure}
     \begin{subfigure}[b]{0.45\textwidth}
         \centering
         \includegraphics[width=\textwidth]{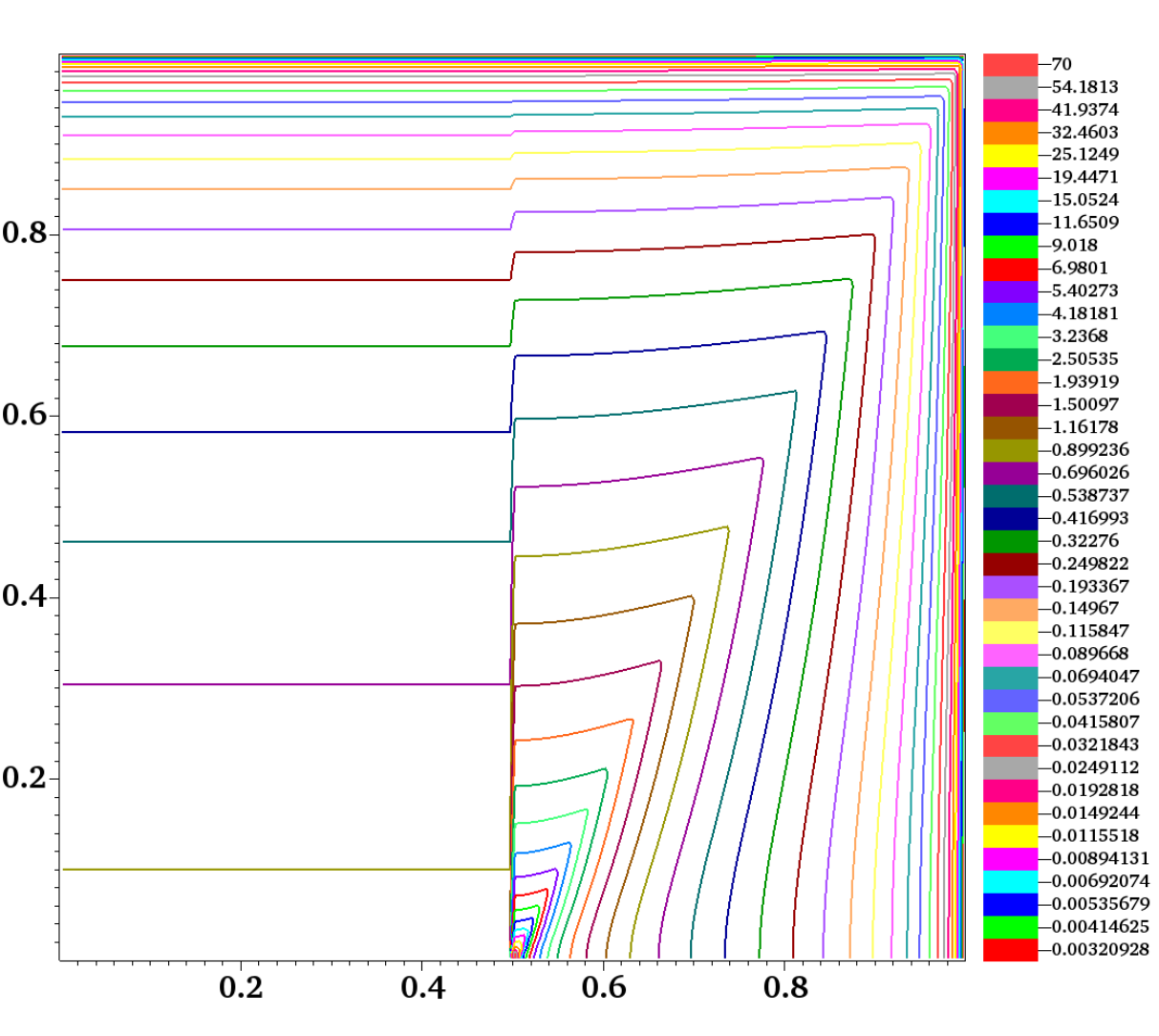}
         \caption{$\overline{v}:$ exact steady state solution }
         \label{fig:three sin x5}
     \end{subfigure}\\
     \begin{subfigure}[b]{0.45\textwidth}
         \centering
         \includegraphics[width=\textwidth]{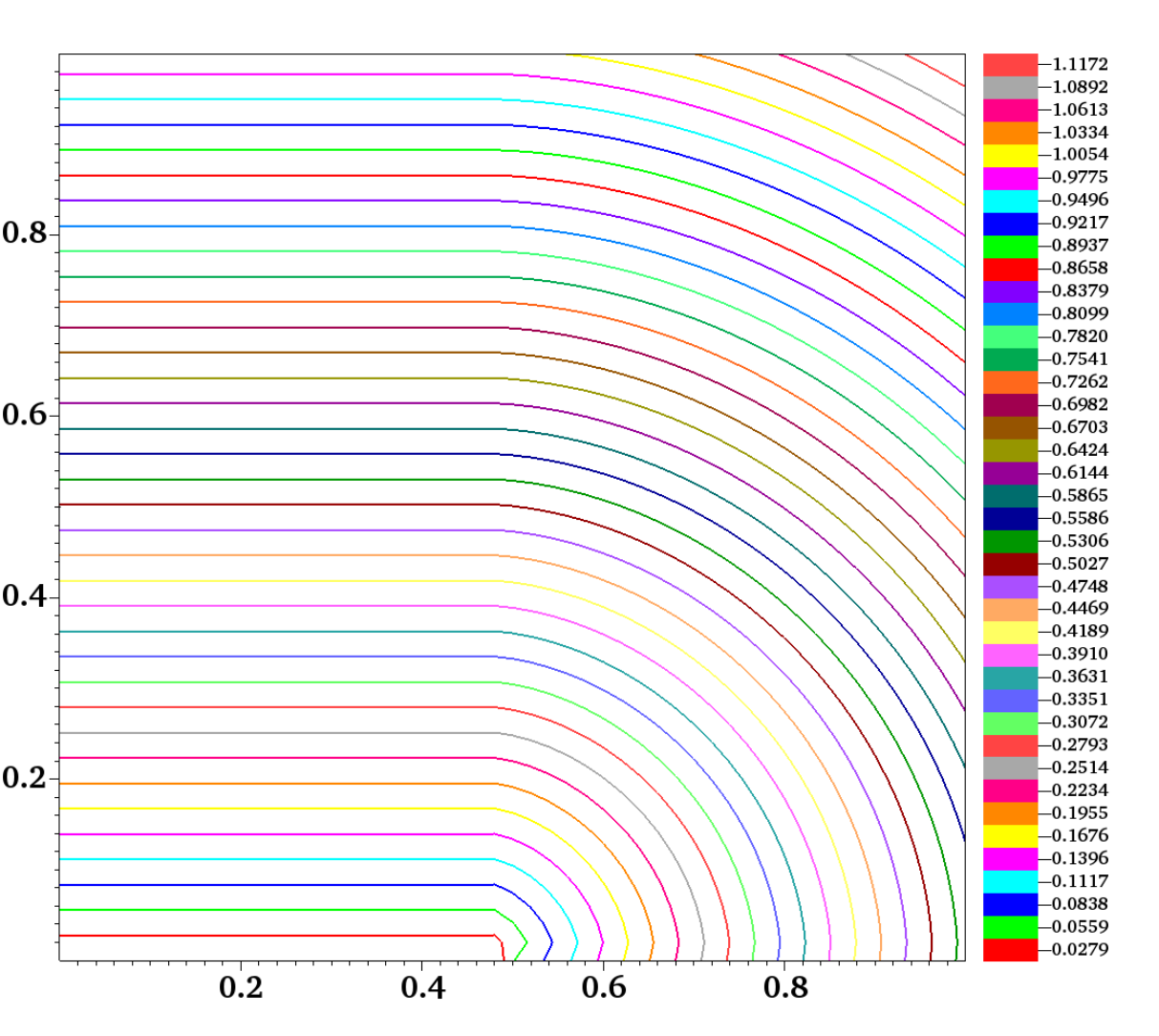}
         \caption{$u_h:$ using FO scheme}
     \end{subfigure}
     \begin{subfigure}[b]{0.45\textwidth}
         \centering
         \includegraphics[width=\textwidth]{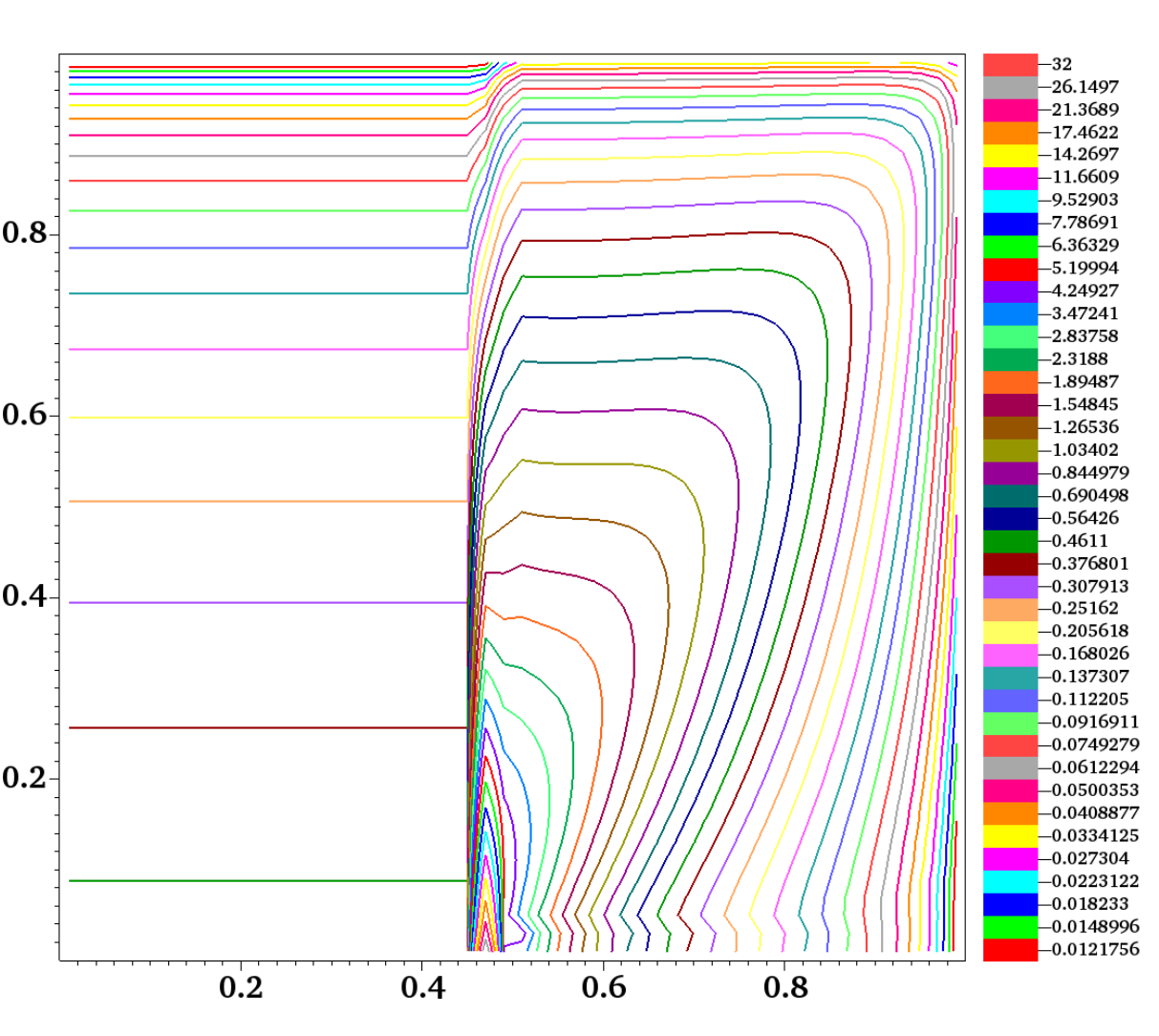}
         \caption{$v_h:$ using FO scheme }
         \label{fig:three sin x10}
     \end{subfigure}\\
              \begin{subfigure}[b]{0.45\textwidth}
         \centering
         \includegraphics[width=\textwidth]{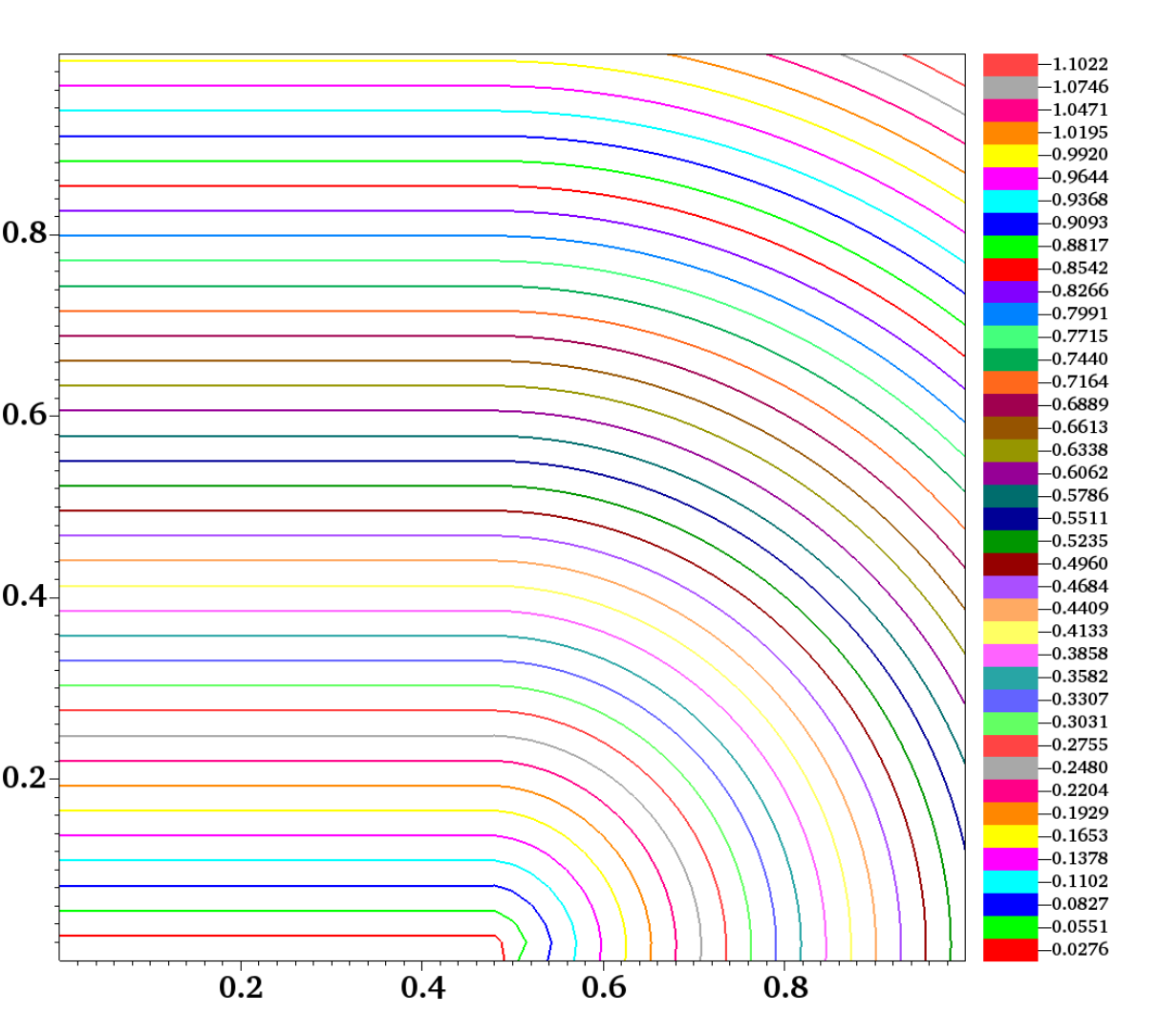}
         \caption{$u_h:$ using SO-$\Theta$ scheme}
     \end{subfigure}
     \begin{subfigure}[b]{0.45\textwidth}
         \centering
         \includegraphics[width=\textwidth]{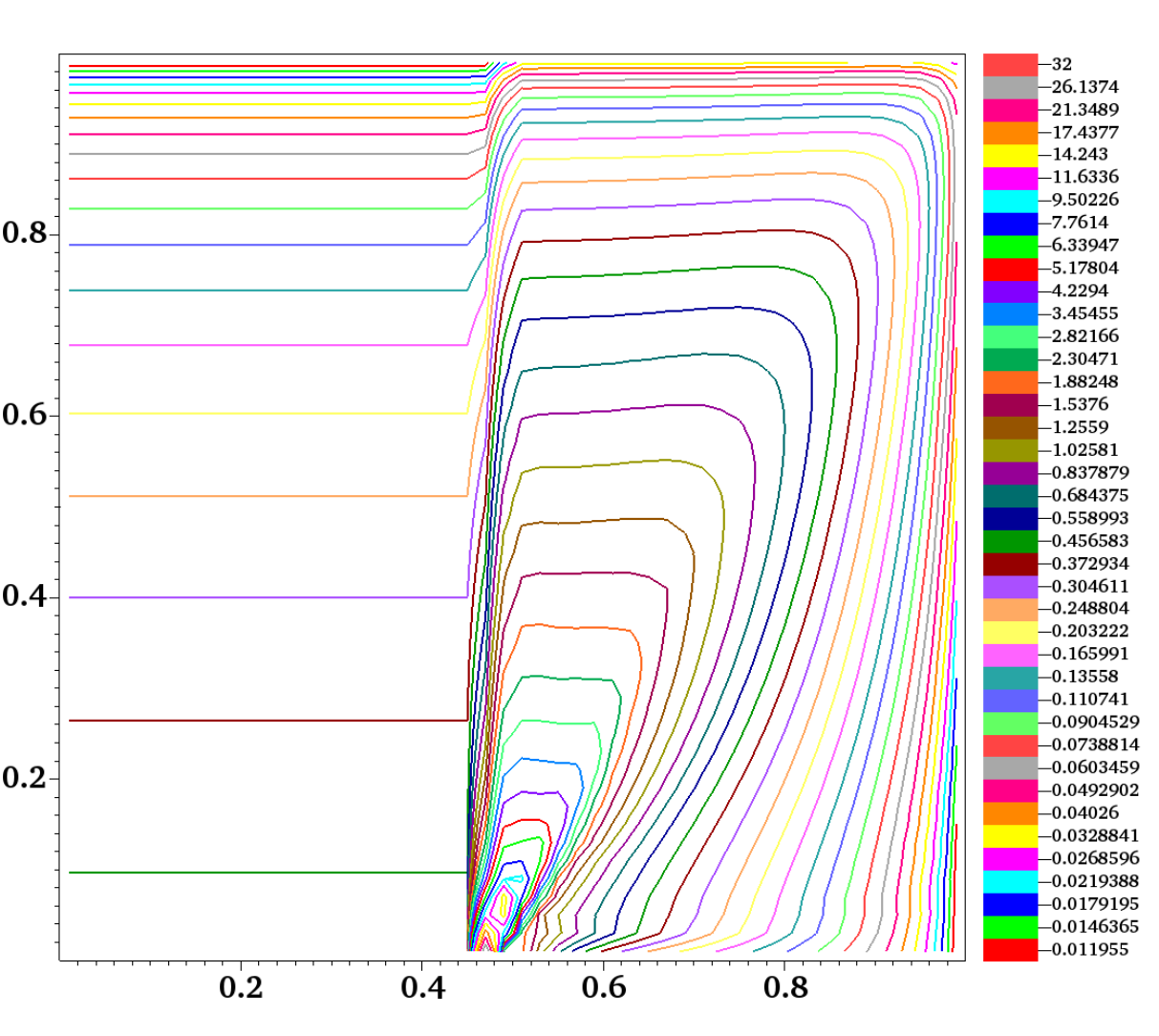}
         \caption{$v_h:$ using SO-$\Theta$ scheme }
         \label{fig:three sin x}
     \end{subfigure}
\caption{Example~9 (2D): numerical solutions of partially open table problem computed at the final time $T = 200$ with $h = 1/50, \Delta t = 0.1 h.$ The contour plots use  40 contour curves.}
\label{fig:pot2contur}
\end{figure}

\section{Conclusion}\label{sec:conclusion}
In this study, we address the challenges associated with numerically approximating the Hadler and Kuttler (\textbf{HK}) model, a complex system of non-linear partial differential equations describing granular matter dynamics. Focusing on  second-order schemes, we study the issues such as initial oscillations and delays in reaching steady states. We present a second-order scheme that incorporates a MUSCL-type spatial reconstruction and a strong stability preserving Runge-Kutta time-stepping method, building upon the  first-order scheme. Through an adaptation procedure that employs a modified limiting strategy in the linear reconstruction, our scheme achieves the well-balance property. We extend our analysis to two dimensions and demonstrate the effectiveness of our adaptive scheme through numerical examples. Notably, our resulting scheme significantly reduces initial oscillations, reaches the steady state solution faster than the first-order scheme, and provides a sharper resolution of the discrete steady state solution.

\section*{Acknowledgement} This work was done while one of the authors, G. D. Veerappa
Gowda, was a Raja Ramanna Fellow at TIFR-Centre for Applicable Mathematics, Bangalore. The work of Sudarshan Kumar K. is supported by the  Science and Engineering Research Board, Government of India, under MATRICS project no.~MTR/2017/000649.

\bibliographystyle{siam}
\bibliography{ref}
\end{document}